\documentclass[reqno]{amsart}
\usepackage{amssymb,amsmath}
\usepackage{amsthm}
\usepackage{color,graphicx}
\usepackage{xcolor}
\usepackage{hyperref}
\usepackage{verbatim}

\usepackage{enumitem}

\newcommand{\les}{\lesssim}

\newcommand{\R}{\mathbb R}

\newcommand{\Z}{\mathbb Z}

\newcommand{\p}{\partial}

\numberwithin{equation}{section}

\newtheorem{theorem}{Theorem}[section]
\newtheorem{proposition}[theorem]{Proposition}
\newtheorem{remark}[theorem]{Remark}
\newtheorem{lemma}[theorem]{Lemma}
\newtheorem{corollary}[theorem]{Corollary}

\newtheorem{claim}[theorem]{Claim}


\begin{document}
\vglue-1cm \hskip1cm
\title[
The generalized fractional KdV equation in  weighted Sobolev spaces
]{The generalized fractional KdV equation in  weighted Sobolev spaces
}

\author[A. Cunha]{Alysson Cunha}
\address{Universidade Federal de Goi\'as (UFG), Campus Samambaia, 131, 74001-970, Goi\^ania, Bra\-zil}
\email{alysson@ufg.br}

\author[O. Ria\~no]{Oscar Ria\~no}
\address{Departamento de Matem\'aticas, Universidad Nacional de Colombia Carrera 45 No. 26-85, Edificio Uriel Guti\'errez Bogot\'a D.C., Colombia}
\email{ogrianoc@unal.edu.co}

\begin{abstract}

This work concerns the study of persistence property in polynomial weighted spaces for solutions of the generalized fractional KdV equation in any spatial dimension $d\geq 1$. By establishing well-posedness results in conjunction with some asymptotic at infinity unique continuation principles, it is verified that dispersive effects and dimensionality mainly determine the maximum spatial decay allowed by solutions of this model. In particular, we recover and extend some known results on weighted spaces for different models such as the Benjamin-Ono equation, and the dispersion generalized Benjamin-Ono equation. The estimates obtained for the linear equation seem to be of independent interest, and they are useful to obtain persistence properties in weighted spaces for models with different nonlinearities as the fractional KdV equation with combined nonlinearities.

\end{abstract}

\maketitle


\section{Introduction}

We consider the initial value problem (IVP) associated to the generalized fractional Korteweg-de Vries equation (fKdV)
    \begin{equation}\label{HBO-ZK}
    \left\{\begin{aligned}
    &\partial_t u -\partial_{x_1}D^{a} u+\nu u^{k-1}\partial_{x_1}u=0,\quad x=(x_1,\dots,x_d)\in \mathbb{R}^{d}, \, \, t\in \mathbb{R}, \, \, a\in (0,2), \\
    &u(x,0)=u_0(x),
    \end{aligned}\right.
    \end{equation}
where $u=u(x,t)$ is a real-valued function, $k\geq 2$ is an integer, $\nu\in \{1,-1\}$,  $D^{s}f=(-\Delta)^{\frac{s}{2}}f$, $s\in \mathbb{R}$ denotes the Riesz potential of negative order, which is defined through the Fourier transform as $D^s f(x)=\mathcal{F}^{-1}\big(|\xi|^{s}\widehat{f}(\xi)\big)(x), \quad |\xi|=\sqrt{\xi_1^2+\cdots+\xi_d^2}$.  For several values of dispersion $0<a\leq 2$, and nonlinearity $k\geq 2$, the equation in \eqref{HBO-ZK} has been used to model physical situations. Setting $a=1$, and using the fact that $D^2=-\Delta$, we have that fKdV coincides with the $k$-generalized $d$-dimensional Benjamin-Ono equation
\begin{equation}\label{gBO}
    \partial_tu-\mathcal{R}_1\Delta u+ \nu u^{k-1}\partial_{x_1}u=0, \quad x\in \mathbb{R}^d,\, \, t \in \mathbb{R},
\end{equation}
where $\mathcal{R}_1=-\partial_{x_1}D^{-1}$ denotes the Riesz transform operator in the $x_1$ variable. In particular, when $d=1$, the Fourier multiplier of the Riesz transform agrees with that of the Hilbert transform, thus fKdV with $a=1$, $d=1$ is the $k$-generalized Benjamin-Ono equation, which has been used to describe wave phenomena, see for example \cite{AbBOnaFellSaut1989,Benjamin1967,BonaKalisch2004,Ono1975}. The case $d=2$ and $k=2$ in the equation in \eqref{gBO}  has been applied to model one-dimensional internal waves in stratified fluids in $\mathbb{R}^3$, see \cite{A,Shrira1989,PS,VS}.  When $a=2$, the equation in \eqref{HBO-ZK} agrees with
\begin{equation}\label{KdVtype}
\partial_{t}u+\partial_{x_1}\Delta u+\nu u^{k-1}\partial_{x_1}u=0, \quad x\in \mathbb{R}^d,\, \, t \in \mathbb{R}.
\end{equation}
Setting $d=1$, \eqref{KdVtype} is the generalized Korteweg–De Vries (KdV) equation, which appears in various physical contexts such as shallow-water waves with weakly non-linear restoring forces, long internal waves in a density-stratified ocean,  ion-acoustic  waves  in plasma,  among  many  others, we refer to \cite{BonaColinLannes2005,Bona1981,JeffreyKakutani1972,ScottChuMcLau1973}. The case $d\geq 2$ is known as the Zakharov-Kuznetsov equation (ZK), which has been applied to model weakly nonlinear ion-acoustic waves in the presence of a uniform magnetic field, see, for instance, \cite{NaumkinShishmar1994,ShriraVoronoVyac1996,ZakharovKuznet1974}.

The case $1<a<2$ serves as a natural dispersive interpolation between the generalized $d$-dimensional Benjamin-Ono equation \eqref{gBO} and the generalized KdV/ZK equations \eqref{KdVtype}. The case $d=1$ and $1<a<2$ in \eqref{HBO-ZK} is also recognized as the dispersion generalized Benjamin-Ono equation. Likewise, the case $0<a<1$ has been used to measure the effects of dispersion and nonlinearity in the dynamic of Burger's equation. Due to the flexibility between dimension, dispersion, and nonlinearity, the fKdV equation has been used to understand the competition between dispersion versus nonlinearity in a $d$-dimensional model. For example, we refer to the studies of wave breaking when $d=1$, $0<a<2$, $2\leq k\leq 3$ in \cite{Hur2017,VeraTaoL2014,KleinSautWang2022,SautWang2022}, the numerical investigations in \cite{KleinLinaresPilodSaut2018,KleinSaut2015,KleinSautWang2022,RianoRoudenkoYang2022}, the study of unique continuation principles \cite{KenigPilodPonceVega2020,FLinaPioncedGBO, Riano2021,LinaresPonce2023},  nonexistence results of breathers in \cite{LinaresMendezPonce2021}, construction of $N$-solitons solutions \cite{Eychenne2023}, existence of dipole solutions \cite{EychenneValet2023}, scattering results \cite{SautWang2021,SautWang2021II}, propagation of regularity principle \cite{Argenis2020,MendezA2020,Argenis2023},  and the study of blow-up solutions in \cite{OhPasqualotto2021}.

Setting $0<a<2$, $\nu=1$, $2\leq k< k^{\ast}$ be integer, where $k^{\ast}=\frac{d+a}{d-a}$ if $0<a<d$, and $k^{\ast}=\infty$, if $a\geq d$, the fKdV equation has a family of solitary-waves (or traveling wave) of the form $u(x_1,\dots,x_d,t)=Q_c(x_1-ct,\dots,x_d)$, where $c>0$, and $Q_c$ satisfies the equation
\begin{equation}\label{EQ:groundState}
 cQ_c+D^{a}Q_c-\frac{1}{k}Q^{k}_c=0.
\end{equation}
Alternatively, we find that $Q_c(x)=c^{\frac{1}{k-1}} Q(c^{\frac{1}{a}}x)$, where $Q$ solves equation \eqref{EQ:groundState} with $c=1$. The existence, uniqueness, and decay properties of solutions to \eqref{EQ:groundState} have been studied in \cite{FranLenzSilve2016,FranLenz2013}. When $2\leq k< k^{\ast}$, it is known that the positive radial symmetric solution $Q\in H^{a+1}(\mathbb{R})\cap C^{\infty}(\mathbb{R}^d)$ satisfies for all $x\in \mathbb{R}^d$
\begin{equation}\label{poldecayGS}
    \frac{A_1}{1+|x|^{d+a}} \leq Q(x) \leq \frac{A_2}{1+|x|^{d+a}},
 \end{equation}
for some positive constants $A_1$ and $A_2$. Regarding stability and instability issues of solitary wave solutions, we refer to \cite{Albert1992,AnguloBonaLinaresScialom2002,LinaresPilodSaut2015,Angulo2018,RianoRoudenko2022} and the references therein. 
\\ \\
The well-posedness in $H^s(\mathbb{R}^d)$ for the equation \eqref{HBO-ZK} has been extensively investigated. In the case $d=1$, we refer to \cite{KenigIonescu2007,MolinetPilod2012,IfrimTata2019,KenigTakaoka2006,MolinetRibaud2004,Vento2010,MolinetPilodVentp2018,HerrIonesKenKoch2010,KenigPonceVega1991,Guo2012,KochT}. In higher dimension $d\geq 2$, see \cite{Schippa2020,HickmanLinaresRiano2019}. For a more detailed description of the well-posedness in $H^s$ for fKdV, see the introduction in \cite{RianoRoudenko2022}. Concerning the well-posedness in weighted spaces, when $d=1$, $1\leq a<2$, $k=2$ in fKdV, see \cite{FLinaPioncedGBO,FonPO,Iorio1986,Iorio2003}, $d=1$, $-1\leq a<1$, $a\neq 0$, $k=2$, see \cite{Riano2021}, and $d\geq 2$, $a=1$ with $k=2$, see \cite{OscarWHBO}. 
\\ \\
Formally, for all $0<a<2$, $k \geq 2$ integer, real solutions of \eqref{HBO-ZK} satisfy three conservation laws:
\begin{equation}\label{ConservationlawsHBO-ZK}
\begin{aligned}
   I_{1}[u]&=\int_{\mathbb{R}} u(x_1,\dots,x_d,t) \, dx_1,\qquad I_{2}[u]=\int_{\mathbb{R}^d} \big(u(x,t)\big)^2 \, dx, \\
   I_{3}[u]&=\frac{1}{2}\int_{\mathbb{R}^d} \left|D^{\frac{a}{2}} u(x,t)\right|^2\, dx- \frac{\nu}{k(k+1)}\int_{\mathbb{R}^d}u^{k+1}(x,t) \, dx.
\end{aligned}
\end{equation}

This paper aims to investigate how dispersive effects and dimensionality influence the propagation of polynomial decay for solutions of dispersive equations, whose dispersion is characterized by some nonlocal operator. Typically, solutions to such models do not persist in arbitrary large polynomial weighted spaces, in particular, the Schwartz class of functions is not preserved by solutions of those equations. In this sense, we investigate the following question: \emph{What is the largest $r>0$, for which one can take an initial condition $u_0\in H^s(\mathbb{R}^d)\cap L^2(|x|^{2r}\, dx)$ such that there exists a nontrivial solution of \eqref{HBO-ZK} with initial condition $u_0$ which persists in the same space (i.e., $u\in C([0,T];H^s(\mathbb{R}^d)\cap L^2(|x|^{2r}\, dx))$?} To address such a problem, we consider the fKdV equation, which besides being a model of active investigations, it is interesting due to its versatility concerning dimension $d\geq 1$, dispersion $0<a<2$, and nonlinearity $k\geq 2$, which ultimately offer a robust model to test different effects in the evolution of nonlinear dispersive equations as well as to set a general framework on the road to a more unified approach, which might be extended to other equations. We emphasize that the type of questions investigated in this article are related to the study of \emph{asymptotic at infinity unique continuation principle} for dispersive equations, for more details see the recent work of Linares and Ponce \cite{LinaresPonce2023}, where the authors detail progress and studies on unique continuation principles that include those developed here. For other unique continuation involving the difference of solutions, we refer to \cite{Cunha2023,Cunha2022}.

Concerning our results, in the pure nonlocal case $0<a<2$, $k\geq 2$ integer in fKdV, we will show that under the effects of nonlinearity of polynomial type, dispersive effects and dimension rather than nonlinearity characterize the maximum $L^2$-polynomial decay allowed by solutions of \eqref{HBO-ZK}. Such behavior is consistent with the decay of the ground state solution $Q$ in \eqref{poldecayGS}. The studies in this paper are divided in two-fold: one is to prove a positive result in the class $H^s(\mathbb{R}^d)\cap L^2(|x|^{2r}\, dx)$, where we establish precise conditions on $s$ and $r>0$ that lead to a local well-posedness theory for solutions of \eqref{HBO-ZK}. Secondly, we prove some asymptotic at infinity unique continuation results, which are applied to obtain the sharpness of our local well-posedness results as well as to characterize the $L^2$-spatial decay for solutions of \eqref{HBO-ZK}. We remark that our results contrast with the case $a=2$ in fKdV (i.e., KdV, ZK equations \eqref{KdVtype}), which propagate exponential weights, and thus polynomial weights of arbitrary size, see \cite{CossetiFanelliLinares2019,EscaKenigPonVega2007,IsazaLinaresPonce2013,Kato1983}.

The study of the persistence property in polynomial weighted spaces for dispersive equations characterized by nonlocal operators has been actively investigated, for example, see \cite{FLinaPioncedGBO,FLinaPonceWeBO,FonPO,Iorio1986,Iorio2003,Riano2021}, and concerning models in higher dimensions, see \cite{CunhaPastor2014,CunhaPastor2016,CunhaPastor2021,Riano2021,OscarWHBO,Riano2021II}. A novelty in this paper is a detailed study of the linear equation associated to fKdV in any spatial dimension $d\geq 1$, and dispersion $0<a<2$, see the results in Lemmas \ref{linearestilemma} and \ref{unicontlemma}, which exactly determine the relation between dispersion and dimension in the propagation of fractional polynomials. The results for the linear equation in weighted spaces are quite useful as they can be applied to models with different nonlinearities and dimensions. As a consequence, we recover known results for some models such as the generalized Benjamin-Ono equation. Furthermore, for a given integer $n\geq 2$, our studies on the linear equation also yield new conclusions for different nonlinearities as in the case of the IVP 
\begin{equation}\label{CombHBO-ZK}
    \left\{\begin{aligned}
    &\partial_t u -\partial_{x_1}D^{a} u+\sum_{j=1}^n\nu_j u^{k_j-1}\partial_{x_1}u=0,\quad x=(x_1,\dots,x_d)\in \mathbb{R}^{d}, \, \, t\in \mathbb{R}, \, \, a\in (0,2), \\
&u(x,0)=u_0(x),
\end{aligned}\right.
\end{equation}
$\nu_j\in\{-1,1\}$, $k_j\geq 2$ be integer, $j=1,\dots,n$. Consequently, we observe that in the presence of polynomial type nonlinearities, the linear part of the equation in \eqref{HBO-ZK} mainly influences the $L^2$-spatial behavior of solutions of the model in \eqref{CombHBO-ZK}.


\subsection{Main Results}

Besides some constraints on the spatial decay, we also have some regularity restrictions, which are related through the following numbers:
\begin{equation}\label{defis1}
    s_{d,k,1}:=\frac{1}{2}\Big(\frac{d}{2}+\frac{k}{k-1} \Big)+\Big(\frac{1}{4}\Big(\frac{d}{2}+\frac{k}{k-1} \Big)^2-\frac{d}{2}\Big)^{\frac{1}{2}},
\end{equation}
and
\begin{equation}\label{defis2}
    s_{d,k,2}:=\frac{1}{2}\Big(\frac{d}{2}+\frac{k}{k-1} \Big)-\Big(\frac{1}{4}\Big(\frac{d}{2}+\frac{k}{k-1} \Big)^2-\frac{d}{2}\Big)^{\frac{1}{2}}.
\end{equation}
Our first result establishes local well-posedness for the Cauchy problem \eqref{HBO-ZK} in weighted spaces.

\begin{theorem}\label{LWPweights} Let $d\geq 1$ be integer, $0<a<2$, $k\geq 2$ be integer, and $\nu\in \{1,-1\}$.

\begin{itemize}
    \item[(i)] If $0<r<1$, $s>\frac{d}{2}+1$, 
    then the Cauchy problem \eqref{HBO-ZK} is locally well-posed in $H^{s}(\mathbb{R}^d)\cap L^2(|x|^{2r}\, dx)$.
    \item[(ii)] If $1\leq r< a+1+\frac{d}{2}$, $s\geq\{(\frac{d}{2}+1)^{+},ar+1\}$ with $s\in (0,s_{d,k,2})\cup (s_{d,k,1},\infty)$, then the Cauchy problem \eqref{HBO-ZK} is locally well-posed in $H^{s}(\mathbb{R}^d)\cap L^2(|x|^{2r}\, dx)$.
   \item[(iii)] If $a+1+\frac{d}{2} \leq  r< a+2+\frac{d}{2}$, $s\geq\{(\frac{d}{2}+1)^{+},ar+1\}$ with $s\in (0,s_{d,k,2})\cup (s_{d,k,1},\infty)$, then the Cauchy problem \eqref{HBO-ZK} is locally well-posed in 
    \begin{equation*}
     H^{s}(\mathbb{R}^d)\cap L^2(|x|^{2r}\, dx)\cap\{f\in H^s(\mathbb{R}^d): \int_{\mathbb{R}^d} f(x)\, dx=0\}.   
    \end{equation*}
\end{itemize}
\end{theorem}
The proof of parts (ii) and (iii) of Theorem \ref{LWPweights} relies on a careful analysis of the behavior of the linear part of the equation in \eqref{HBO-ZK} in weighted spaces.  Before we present such results for the linear equation, we introduce some notation. We denote by $U(t)$ the group of solutions determined by the equation $\partial_tu-\partial_{x_1}D^a u=0$, $x\in \mathbb{R}^d$, $t\in \mathbb{R}$, i.e.,
\begin{equation}\label{lineaeq}
U(t)f=(e^{it\xi_1|\xi|^a}\widehat{f}(\xi))^{\vee},
\end{equation}
where following the standard notation, $\wedge$ and $\vee$ denote the Fourier and the inverse Fourier transform operators, respectively. We also denote by $\langle x\rangle=(1+|x|^2)^{\frac{1}{2}}$, $x\in \mathbb{R}^d$. Below, $J^s$, $s\in \mathbb{R}$, denotes the Bessel operator, which is defined in Section \ref{notation}.
\begin{lemma}\label{linearestilemma}
Let $0<a<2$, $d\geq 1$ be integer. 
\begin{itemize}
\item[(i)] Consider $0<r<a+1+\frac{d}{2}$, and $f\in H^{a r}(\mathbb{R}^d)\cap L^2(|x|^{2 r}\, dx)$. Then for all $t\in \mathbb{R}$
\begin{equation}\label{linereq1}
    \|\langle x \rangle^r U(t)f\|_{L^2}\lesssim \langle t \rangle^{r}\big(\|J^{a r}f\|_{L^2}+\|\langle x \rangle^r f\|_{L^2}\big).
\end{equation}
\item[(ii)] The inequality \eqref{linereq1} also holds true for weights $a+m+\frac{d}{2}\leq r<a+1+m+\frac{d}{2}$ for some $m\in \mathbb{Z}^{+}$, if $f\in H^{a r}(\mathbb{R}^d)\cap L^2(|x|^{2 r}\, dx)$ with
\begin{equation}\label{linearEstimcomp1}
\begin{aligned}
\int_{\mathbb{R}^d} x^{\beta}f(x)\, dx=0, \text{ for all } |\beta|\leq m-1.
\end{aligned}    
\end{equation} 
\end{itemize}
\end{lemma}
Concerning the proof of Theorem \ref{LWPweights}, the idea to get parts (ii) and (iii) is to use an iterative argument based on the size of the weight $r>0$, in which we use the integral formulation of \eqref{HBO-ZK}  in conjunction with the linear estimates in Lemma \ref{linearestilemma}. It is worth noticing that such a scheme is also feasible due to Lemma \ref{lemmadecaynonlinear} and Corollary \ref{corollarynonlin} in which we deduce some key nonlinear estimates in weighted spaces.  However, the procedure just described requires deducing some initial polynomial decay properties for solutions of fKdV, which are needed to justify the integral formulation of \eqref{HBO-ZK} in some weighted spaces. Thus, we obtain a first persistence result in Theorem \ref{LWPweights} (i), where we use a different strategy based on energy methods and commutator estimates. We remark that such an approach has been used in the works of Fonseca, Linares, and Ponce \cite{FonPO,FLinaPonceWeBO}. In contrast to these works, our results in Theorem \ref{LWPweights} (i) deal with any spatial dimension $d\geq 1$, and dispersion of any order $0<a<2$, which apply to previous known results in the literature, and they are more involved in the analysis carried out on the dispersion in fKdV. Furthermore, comparing the proof of Theorem \ref{LWPweights} part (i) with parts (ii) and (iii), we observe that  Lemma \ref{linearestilemma} simplifies the deduction of the persistence results, and in a way, it is observed that the linear part of the equation seems to have more influence in the spatial behavior of solutions of fKdV.

One may ask whether the condition \eqref{linearEstimcomp1} is technical, or if it is a consequence of the dispersive effects of the linear group $U(t)$ in weighted spaces. The following lemma shows that \eqref{linearEstimcomp1} is necessary to have $U(t_j)f\in L^2(|x|^{2r}\, dx)$, $r\geq a+m+\frac{d}{2}$, at two different times $t_1$, $t_2$, $m\in \mathbb{Z}^{+}$. Consequently, we establish that condition \eqref{linearEstimcomp1} in Lemma \ref{linearestilemma} is sharp.

\begin{lemma}\label{unicontlemma}
Let $d\geq 1$ integer, $0<a<2$, $m\geq 1$ integer. Assume $f\in H^{a(a+m+\frac{d}{2})}(\mathbb{R}^d)\cap L^2 (|x|^{2({a+m+\frac{d}{2}})}\, dx)$. If there exists $t\neq 0$, such that
\begin{equation}\label{uniquecontlemma1}
U(t)f\in L^2(|x|^{2(a+m+\frac{d}{2})}\, dx),
\end{equation}
then \eqref{linearEstimcomp1} must hold true for all $|\beta|\leq m-1$. 
\end{lemma}

Lemma \ref{unicontlemma} and the arguments in the proof of Theorem \ref{LWPweights} (ii) and (iii) imply the following asymptotic at infinity unique continuation principle for solutions of \eqref{HBO-ZK}.

\begin{theorem}\label{theoremtwotimes}
Let $d\geq 1$ be integer, $k\geq 2$ be integer, $0<a<2$, $s\geq\{(\frac{d}{2}+1)^{+},a(a+1+\frac{d}{2})+1\}$ with $s\in (0,s_{d,k,2})\cup (s_{d,k,1},\infty)$, and $\nu \in \{-1,1\}$. Let $u\in C([0,T];H^s(\mathbb{R}^d)\cap L^2(|x|^{2(a+1+\frac{d}{2})^{-}}))$ be a solution of \eqref{HBO-ZK}. If there exists two different times $t_1, t_2\in [0,T]$ such that
\begin{equation}\label{assumpttwotimes}
    u(\cdot,t_1), u(\cdot,t_2)\in L^2(|x|^{2(a+1+\frac{d}{2})}\, dx)
\end{equation}
then
\begin{equation*}
    \int_{\mathbb{R}^d} u_0(x)\, dx=0.
\end{equation*}
\end{theorem}

When $k\geq 2$ is an even integer number in \eqref{HBO-ZK}, we can deduce some further asymptotic at infinity unique continuation principles, which yields the maximum $L^2$-spatial decay for solutions of \eqref{HBO-ZK}.

\begin{theorem}\label{theoremthreetimes}
Let $d\geq 1$ be integer, $k\geq 2$ be an even integer, $0<a<2$, $s\geq\{(\frac{d}{2}+1)^{+},a(a+2+\frac{d}{2})+1\}$ with $s\in (0,s_{d,k,2})\cup (s_{d,k,1},\infty)$, and $\nu\in\{1,-1\}$. Let $u\in C([0,T];H^s(\mathbb{R}^d)\cap L^2(|x|^{2(a+2+\frac{d}{2})^{-}}\,dx))$ be a solution of \eqref{HBO-ZK} with $\int_{\mathbb{R}^d} u_0(x)\, dx=0$.
\begin{itemize}
    \item[(i)] If there exist three different times $t_1, t_2, t_3\in [0,T]$ such that
\begin{equation}\label{assumpthreetimes}
    u(\cdot,t_1), u(\cdot,t_2), u(\cdot,t_3)\in L^2(|x|^{2(a+2+\frac{d}{2})}\, dx),
\end{equation}
then $u\equiv 0$.
\item[(ii)] If there exist two times such that
\begin{equation}\label{twotimes}
    u(\cdot,t_1), u(\cdot,t_2)\in L^2(|x|^{2(a+2+\frac{d}{2})}\, dx),
\end{equation}
and
\begin{equation}\label{extrahyph}
     \int_{\mathbb{R}^d} x_1u(x,t_1)\, dx=0, \hspace{1cm} \text{ or } \hspace{1cm} \int_{\mathbb{R}^d} x_1u(x,t_2)\, dx=0.  
    \end{equation}
Then $u\equiv 0$.
\end{itemize}
\end{theorem}

\begin{corollary}\label{sharpthree}
Let $d\geq 1$ be integer, $k\geq 2$ be integer, $0<a<2$, $s\geq\{(\frac{d}{2}+1)^{+},a(a+2+\frac{d}{2})+1\}$ with $s\in (0,s_{d,k,2})\cup (s_{d,k,1},\infty)$, and $\nu \in\{1,-1\}$. Let $u_0\in H^s(\mathbb{R}^d)\cap L^2(|x|^{2(a+2+\frac{d}{2})}\, dx)$, with $\int_{\mathbb{R}^d} u_0(x)\, dx=0$ and $\int_{\mathbb{R}^d} x_j u_0(x)\, dx=0$ for all $j=2,\dots,d$. If $k$ is even, we will also assume $\int_{\mathbb{R}^d}x_1 u_0(x)\neq 0$. Let $u\in C([0,T];H^s(\mathbb{R}^d)\cap L^2(|x|^{2(a+2+\frac{d}{2})^{-}}\,dx))$ be the solution of \eqref{HBO-ZK} with initial condition $u_0$.

Assume that there exists $t^{\ast}_k\in (0,T)$ such that
\begin{equation*}
\begin{aligned}
t^{\ast}_k\int_{\mathbb{R}^d}x_1 u_0(x)\, dx=-\frac{\nu }{k}\int_0^{t_k^{\ast}}\int_{\mathbb{R}^d}(t_k^{\ast}-\tau)(u(x,\tau))^k \, dx d\tau.
\end{aligned}  
\end{equation*}
In particular, when $k=2$, the $L^2$-conservation law implies 
\begin{equation*}
    t^{\ast}_{2}=\frac{-4}{\nu \|u_0\|_{L^2}^2}\int_{\mathbb{R}^d} x_1 u_0(x)\, dx.
\end{equation*}
Then $u(\cdot,t^{\ast}_{k})\in L^{2}(|x|^{2(a+2+\frac{d}{2})}\, dx)$.
\end{corollary}

When $k\geq 2$ is an odd integer number, we cannot use the arguments in the proof of Theorem \ref{theoremthreetimes} to find the maximum decay of solutions of \eqref{HBO-ZK}. Thus, the following result shows some persistence in weighted spaces when $k$ is an odd number. 

\begin{corollary}\label{oddcasetheorem}
Let $d\geq 1$ be integer, $k\geq 2$ be an odd integer, $0<a<2$, and $\nu\in\{1,-1\}$. Consider $m\geq 2$ be integer, $a+m+\frac{d}{2}\leq r<a+1+m+\frac{d}{2}$. Let $s\geq\{(\frac{d}{2}+1)^{+},ar+1\}$ with $s\in (0,s_{d,k,2})\cup (s_{d,k,1},\infty)$. Let $u_0\in H^s(\mathbb{R}^d)\cap L^{2}(|x|^{2r}\, dx)$ such that
\begin{equation}\label{inicondzeroprop}
    \int_{\mathbb{R}^d} x^{\beta}u_0(x)\, dx=0,
\end{equation}
for all multi-index $|\beta|\leq m-1$. Let $u\in C([0,T];H^s(\mathbb{R}^d))$ be the solution of \eqref{HBO-ZK} with initial condition $u_0$. If
    \begin{equation*}
    \int_{\mathbb{R}^d} x^{\beta}(u(x,t))^k\, dx=0
\end{equation*}
for all $t\in[0,T]$, and $|\beta|\leq m-2$, then it follows
\begin{equation*}
    u\in L^{\infty}([0,T];L^2(|x|^{2r}\,dx))).
\end{equation*}

\end{corollary}

We note that the unique continuation principles in Theorems \ref{theoremtwotimes} and \ref{theoremthreetimes} are strongly influenced by the linear part of the equation and the condition \eqref{linearEstimcomp1} in Lemma \ref{unicontlemma}, which ultimately determine the maximum polynomial decay allowed by the nonlinear equation fKdV. In this regard, our results in Lemma \ref{theoremtwotimes} establish that for arbitrary initial data such that $\int u_0(x)\, dx \neq 0$, the decay $|x|^{(\frac{d}{2}+1+a)^{-}}$ is the maximum possible for solutions of the IVP \eqref{HBO-ZK}. In other words, if $u_0\in H^s(\mathbb{R}^d)\cap L^2(|x|^{2(\frac{d}{2}+1+a)}\, dx)$ with $s>0$ large enough (as in Theorem \ref{LWPweights}), and $\int u_0(x)\, dx \neq 0$, then the corresponding solution of \eqref{HBO-ZK} generated from this initial condition satisfies $|x|^{(\frac{d}{2}+1+a )^{-}}u\in L^{\infty}([0,T];L^2(\mathbb{R}^d))$, but there does not exist a non-trivial solution $u$ with initial data $u_0$ such that $| x|^{(\frac{d}{2}+1+a)}u\in L^{\infty}([0,T_1];L^2(\mathbb{R}^d))$ for some $T_1$.

On the other hand, from Lemma \ref{unicontlemma}, we formally expect that to increase the order of the polynomial weight imposed on the initial condition, it would be required to further incorporate condition \eqref{linearEstimcomp1} in the study of the nonlinear equation. However, this is not the case when $k\geq 2$ is an even integer as certain symmetries (see Remark (c) below) of the nonlinear equation deduced in the proof of Theorem \ref{theoremthreetimes} show that $r=\big(\frac{d}{2}+2+a\big)^{-}$ is the largest possible $L^2$-polynomial spatial decay rate for solutions of \eqref{HBO-ZK}. More precisely, if $u_0\in H^s(\mathbb{R}^d)\cap L^2(|x|^{2(\frac{d}{2}+2+a)}\, dx)$, $u_0\neq 0$ with $s>0$ large enough, and $\int u_0(x)\, dx = 0$, then the corresponding solution $u$ of \eqref{HBO-ZK} with $k\geq 2$ even integer generated from $u_0$ satisfies $|x|^{(\frac{d}{2}+2+a )^{-}}u\in L^{\infty}([0,T];L^2(\mathbb{R}^d))$, but there does not exist a non-trivial solution $u$ with initial data $u_0$ such that $|x|^{(\frac{d}{2}+2+a)}u\in L^{\infty}([0,T_2];L^2(\mathbb{R}^d))$ for some $T_2$. However, our proof of Theorem \ref{theoremthreetimes} is not conclusive for the case $k\geq 2$ odd number as we can not confirm the maximum polynomial decay for solutions of \eqref{HBO-ZK} in this case. Nevertheless, we have deduced Corollary \ref{oddcasetheorem} in which we state some persistence properties in weighted spaces of arbitrary size $r\geq 0$ for solutions of \eqref{HBO-ZK} with $k\geq 2$ odd.

We also remark that Corollary \ref{sharpthree} shows that the three times condition in Theorem \ref{theoremthreetimes} is optimal and it cannot be reduced to two times, which is consistent with the result for the $d=1$ model in \cite[Theorem 2]{FLinaPonceWeBO}, \cite[Theorem 1.5]{FLinaPioncedGBO}, and \cite[Theorem 1.7]{Riano2021}.

Since our results mainly depend on the properties of the linear equation in Lemmas \ref{linearestilemma}, \ref{unicontlemma}, we can extend some of the previous theorems to the IVP \eqref{CombHBO-ZK}. 
\begin{corollary}\label{corollaryCombinedfKdV}
Let $d\geq 1$, $n\geq 2$ be integers.
\begin{itemize}[leftmargin=20pt]
    \item[(i)] Let $k_j \geq 2$ be integer, $\nu_j\in\{1,-1\}$, for all $j=1,\dots, n$. Then the Cauchy problem \eqref{CombHBO-ZK} is locally well-posed in the following spaces:
    \begin{itemize}
        \item[(i.a)] $H^s(\mathbb{R}^d)\cap L^2(|x|^{2r}\, dx)$, with $0<r<1$, $s>\frac{d}{2}+1$.
        \item[(i.b)] $H^s(\mathbb{R}^d)\cap L^2(|x|^{2r}\, dx)$ with $1\leq r<a+1+\frac{d}{2}$, $s\geq \{(\frac{d}{2}+1)^{+},ar+1\}$, $s>\max_{1\leq j \leq n}\{s_{d,k_j,1}\}$.       
        \item[(i.c)] $H^s(\mathbb{R}^d)\cap L^2(|x|^{2r}\, dx)\cap\{f\in H^s(\mathbb{R}^d): \int_{\mathbb{R}^d}f(x)\, dx=0\}$ with $a+1+\frac{d}{2}\leq r<a+2+\frac{d}{2}$, $s\geq \{(\frac{d}{2}+1)^{+},ar+1\}$, $s>\max_{1\leq j \leq n}\{s_{d,k_j,1}\}$ (with $s_{d,k_k,1}$ given by \eqref{defis1}).
    \end{itemize}
    \item[(ii)] Let $k_j \geq 2$ be integer, $\nu_j\in\{1,-1\}$, for all $j=1,\dots, n$. Let $s\geq\{(\frac{d}{2}+1)^{+},a(a+1+\frac{d}{2})+1\}$ with $s>\max_{1\leq j \leq n}\{s_{d,k_j,1}\}$. Then the two times unique continuation principle in Theorem \ref{theoremtwotimes} holds for the Cauchy problem \eqref{CombHBO-ZK}. 
    \item[(iii)] Assume that $k_j \geq 2$ is even for all $j=1,\dots, n$, and $\nu_j=\nu_{j'}$ for all $j,j'=1,\dots n$. Let $s\geq\{(\frac{d}{2}+1)^{+},a(a+2+\frac{d}{2})+1\}$ with $s>\max_{1\leq j \leq n}\{s_{d,k_j,1}\}$. Then the conclusions of Theorem \ref{theoremthreetimes} hold for the Cauchy problem \eqref{CombHBO-ZK}.  
    \item[(iv)] Let $k_j \geq 2$ be integer, $\nu_j\in\{1,-1\}$, $j=1,\dots n$. Let $s\geq\{(\frac{d}{2}+1)^{+},a(a+2+\frac{d}{2})+1\}$ with $s>\max_{1\leq j \leq n}\{s_{d,k_j,1}\}$.  Let $u_0\in H^s(\mathbb{R}^d)\cap L^2(|x|^{2(a+2+\frac{d}{2})}\, dx)$, with $\int_{\mathbb{R}^d} u_0(x)\, dx=0$ and $\int_{\mathbb{R}^d} x_j u_0(x)\, dx=0$ for all $j=2,\dots,d$. If all the $k_j$ are even and $\nu_j=\nu_{j'}$, for all $j,j'=1,\dots,n$, we will also assume $\int_{\mathbb{R}^d}x_1 u_0(x)\neq 0$. Then the same conclusion in Corollary \eqref{sharpthree} holds for the Cauchy problem \eqref{CombHBO-ZK}, but in this case $t^{\ast}\in(0,T)$ is defined by
    \begin{equation*}
\begin{aligned}
t^{\ast}\int_{\mathbb{R}^d}x_1 u_0(x)\, dx=-\sum_{j=1}^n\frac{\nu_j }{k_j}\int_0^{t^{\ast}}\int_{\mathbb{R}^d}(t^{\ast}-\tau)(u(x,\tau))^{k_j} \, dx d\tau.
\end{aligned}  
\end{equation*}
\item[(v)]Let $k_j\geq 2$, $j=1,\dots, n$ be different integers where either there exist some indices $k_j$ even and $k_{j'}$ odd, or $\nu_j\neq \nu_{j'}$ for some $j\neq j'$. Let $m\geq 2$ be integer, $a+m+\frac{d}{2}\leq r<a+1+m+\frac{d}{2}$. Let $s\geq\{(\frac{d}{2}+1)^{+},ar+1\}$ with $s>\max_{1\leq j \leq n}\{s_{d,k_j,1}\}$. Then the results of Corollary \ref{oddcasetheorem} are valid for the Cauchy problem \eqref{CombHBO-ZK} provided that $u_0\in H^s(\mathbb{R}^d)\cap L^2(|x|^{2r}\, dx)$ satisfies \eqref{inicondzeroprop}, and the solution $u$ of \eqref{CombHBO-ZK} with initial condition $u_0$ satisfies 
    \begin{equation*}
    \int_{\mathbb{R}^d} x^{\beta}(u(x,t))^{k_j}\, dx=0
\end{equation*}
for all $t\in[0,T]$, $|\beta|\leq m-2$ and $j=1,\dots,n$.
\end{itemize}    
\end{corollary}

{\bf Remarks.} (a) The flexibility of our results concerning the dimension allows us to recover known results in the literature while extending them naturally to several variable settings. For example, setting $d=1$ in Theorems \ref{LWPweights}, \ref{theoremtwotimes} and \ref{theoremthreetimes}, we recover the maximum spatial decay properties of the Benjamin-Ono equation ($d=1$, $a=1$ in \eqref{HBO-ZK}) established in \cite[Theorems 1, 2 and 3, Remark (e)]{FonPO}, the results for the dispersion generalized Benjamin-Ono equation ($d=1$, $1\leq a<2$ in \eqref{HBO-ZK}) in \cite[Theorems 1.1., 1.2, and 1.3, equation (1.18)]{FLinaPioncedGBO}, and the results for the $d=1$, $0<a<1$ fractional KdV equation \eqref{HBO-ZK} in \cite[Theorems 1.1, 1.3, 1.4, and 1.6, see also Remark 1.8 (c)]{Riano2021}. In higher dimensions $d\geq 2$, we recover the results for the Cauchy problem associated to \eqref{gBO} in \cite[Theorems 1.1, 1.2, and 1.3]{OscarWHBO}.
\\ \\
(b) This article aims to study the maximum decay $r$ propagated by solutions of the equation fKdV in terms of the dimension $d\geq 1$, and the dispersion $0<a<2$.  One question still unanswered is establishing the minimum regularity in the Sobolev scale $s$ required to propagate polynomial decay of order $0<r<\frac{d}{2}+a+2$. One may conjecture that the balance between decay and regularity for fKdV must satisfy $s\geq ar$, i.e., this should be the minimal condition to obtain local well-posedness in the space $H^s(\mathbb{R}^d)\cap L^2(|x|^{2r}\, dx)$. Notice that the condition $s\geq ar$ can be motivated from the requirements between $s$ and $r$ in the case of the linear equation of fKdV in Lemmas \ref{linearestilemma} and \ref{unicontlemma}. Even though our results in Theorems \ref{LWPweights}, \ref{theoremtwotimes} and \ref{theoremthreetimes} are not intend to obtain minimal regularity $s$, they are optimized to the techniques developed in this manuscript, and the condition $s>\frac{d}{2}+1$ can certainly be improved. We also notice that the assumption $s\in (0,s_{d,k,2})\cup(s_{d,k,1},\infty)$ is completely technical. Although, the condition $s>\frac{d}{2}+1$ eliminates the possibility that $s\in(0,s_{d,k,2})$, to show the extension of validity of our arguments towards a lower regularity local theory for solution of fKdV, we have kept the full range $s\in (0,s_{d,k,2})\cup(s_{d,k,1},\infty)$ in the hypothesis of our results. 

As a matter of fact, our results can be extended to regularities $s\leq \frac{d} {2}+1$ with $s\geq ar$ in the case of Theorem \ref{LWPweights} (i), and with the addition of the conditions $s\geq ar+1$ with $s\in (0,s_{d,k,2})\cup(s_{d,k,1},\infty)$ in Theorem \ref{LWPweights} (ii) and (iii), provided that there exists a local theory in $H^s(\mathbb{R}^d)$ for which the solutions $u$ of \eqref{HBO-ZK} also belong to the class 
\begin{equation}\label{wellposcond}
    u\in L^{k-1}((0,T);W^{1,\infty}(\mathbb{R}^d)),
\end{equation}
(see Remark \ref{RemarkNonlEstim} below). Examples of local well-posedness results where \eqref{wellposcond} holds true include: The Benjamin-Ono equation ($d=1$, $a=1$, $k=2,3$ in \eqref{HBO-ZK}), see \cite[Theorems 1.1. and 1.2]{KenigKo} (see also \cite{Ponce1991,KochT}), the dispersion generalized Benjamin-Ono equation ($d=1$, $1<a<2$, $k=2$ in \eqref{HBO-ZK}), see \cite[Theorem A]{Argenis2020}, the one dimensional fractional KdV with lower dispersion ($d=1$, $k=2$, $0<a<1$ in \eqref{HBO-ZK}), see \cite[Theorem 1.4]{LinaresPilodSaut2014}, and the Benjamin-Ono equation in higher dimensions ($d=2$, $k=2$, $a=1$), see \cite{HickmanLinaresRiano2019}.
\\ \\
(c) The proof of Theorem \ref{theoremthreetimes} depends on the following first momentum identity
\begin{equation*}
\frac{d}{dt}\int_{\mathbb{R}^d} x_1 u(x,t)\, dx=\frac{\nu}{k}\int_{\mathbb{R}^d} (u(x,t))^k\, dx. 
\end{equation*}
Under the hypothesis of Theorem \ref{theoremthreetimes}, the idea to deduce the three times unique continuation principle consists in proving that the first momentum of the solution vanishes somewhere in the time intervals $(t_1,t_2)$ and $(t_2,t_3)$. This in turn yields $\int_{\mathbb{R}^d} (u(x,\widetilde{t}))^k\, dx=0$ for some time $\widetilde{t}$. Thus, when $k\geq 2$ is even, it follows from the $L^2$-conservation law that $u\equiv 0$. This is not the case when $k\geq 2$ is odd. For this reason, Theorem \ref{theoremthreetimes} only deals with the case $k\geq 2$ even, and we have obtained some further decay properties when $k$ is odd in Corollary \ref{oddcasetheorem}. An open problem is to determine the maximum polynomial decay in $L^2$-spaces, if such exists, for solutions of \eqref{HBO-ZK} with $k\geq 2$ odd. To the best of our knowledge, this question has not been investigated for any dimension $d\geq 1$ and dispersion $0<a<2$. The solution to this problem might require different techniques than those proposed in this manuscript.
\\ \\
(d) Our studies on linear effects in weighted spaces allow us to obtain results for models with different nonlinearities. As an example, we consider the Cauchy problem \eqref{CombHBO-ZK} that corresponds to fKdV with combined nonlinearities. In this regard, Corollary \ref{corollaryCombinedfKdV} parts (i), (ii) show that $r=(\frac {d}{2}+1+a)^{-}$ is the maximum polynomial decay propagated for solutions of \eqref{CombHBO-ZK} with sufficiently regular initial condition $u_0$ such that $\int u_0(x)\, dx\neq 0$. In the case where all the nonlinearities $k_j\geq 2$ are even, and the terms $\nu_j$ are all equal, Corollary \ref{corollaryCombinedfKdV} (iii) shows that the maximum $L^2$-polynomial decay admitted by solutions of \eqref{CombHBO-ZK} is $r=(\frac{ d}{2}+2+a)^{-}$. Moreover, in this same case, the three times condition is optimal by the results of Corollary \ref{corollaryCombinedfKdV} (iv). However, when the initial condition is such that $\int u_0(x)\, dx=0$ with $u_0\neq 0$, given different integer powers $k_j\geq 2$ with $j=1,\dots,n$, $n\geq 2$, it remains to determinate the optimal $L^2$-spatial decay for solutions of \eqref{CombHBO-ZK}, where either the powers of the nonlinearities combine even and odd cases, or when $\nu_j\neq \nu_{j'}$ for some $j\neq j'$. This question is at present far from being solved as it seems to be more related to the competition between different nonlinearities, and their influence on the spatial dynamics of solutions. Nevertheless, we present some results towards the persistence of polynomial decay for arbitrary nonlinearities $k_j$ and coefficients $\nu_j$ in Corollary \ref{corollaryCombinedfKdV} (iv) and (v).
\\ \\
(e) When $d=1$, the asymptotic at infinity unique continuation principles deduced in Theorems \ref{theoremtwotimes}, \ref{theoremthreetimes}, and Corollary \ref{sharpthree} coincide with those obtained in the works of Fonseca, Linares and Ponce \cite{FLinaPioncedGBO}, and those of the second author in \cite{Riano2021}.  Since our results are valid for any dimension, in a way, they can be considered as a generalization of those previously established when $d=1$, moreover, they seem to be the first results for $d\geq 2$, see \cite{LinaresPonce2023}.  

On the other hand, Kenig, Pilod, Ponce, and Vega \cite{KenigPilodPonceVega2020,KenigPonceVega2020} recently deduced some local unique continuation principles for solution of fKdV with $d\geq 1$. For local one (see the introduction in \cite{LinaresPonce2023}), it means that if two suitable solutions $u$, $v$ of fKdV coincide in an open set $\Omega$ (space-time set), then $u\equiv v$ in the whole domain of definition of $u$ and $v$. In this regard, the results in \cite{KenigPilodPonceVega2020}  can be applied to the case of fKdV in higher dimensions, and the range of dispersions can also be extended to some negative values as well.
\begin{theorem}\label{unici}
Let $a\in (-1,\infty)-2\Z$, $k\geq 2$ integer, $\nu\in\{1,-1\}$, and $u,v$ be two real solutions of the IVP \eqref{HBO-ZK} such that
\begin{equation}\label{class}
u,v \in C([0,T];H^s(\R^d))\cap C^1([0,T];H^{s'}(\R^d)),
\end{equation}
with $s>\max\{a+1,\frac{d}{2}+1\}$ and $s'=\min\{s-(a+1),s-1\}$. Moreover, if $d=1$, $a\in (-1,-\frac{1}{2}]$ also assume that 
\begin{equation}\label{japonesuv}
|x|\p_{x}u(\cdot,t_0), |x| \p_{x}v(\cdot,t_0) \in L^2(\R).
\end{equation}
If there exists a non-empty open set $\Omega\subset \R^d$, and a time $t_0\in (0,T)$ such that 
\begin{equation}\label{omegaset}
u(x,t_0)=v(x,t_0) \quad \mbox{and} \quad \p_t u(x,t_0)=\p_t v(x,t_0) \quad \mbox{for any}  \quad x\in \Omega,
\end{equation}
then $u(x,t)=v(x,t)$, for all $(x,t)\in \R^d \times [0,T]$. 
\end{theorem}
When $d=1$, Theorem \ref{unici} was deduced in \cite[Theorem 1.1]{KenigPilodPonceVega2020}, we also notice that the case $d=1$, $a=-1$, i.e., the Burgers-Hilbert equation, similar local unique continuation principles were obtained in \cite[Remark 1.4]{KenigPonceVega2020}. The proof of Theorem \ref{unici} when $d\geq 2$ follows the same ideas as in case $d=1$ in \cite[Theorem 1.1]{KenigPilodPonceVega2020}, which depend on the unique continuation principles for the fractional Laplacian established by Ghosh, Salo and Uhlmann \cite{GhoshSaloUhlmann2020}. For this reason, we will not present the deduction of the above theorem here. However, we remark that when $d\geq 2$, one can conjecture that, under a possible addition of extra hypotheses on $u$, $v$ (such as an analogue of \eqref{japonesuv}), the conclusion of Theorem \ref{unici} also holds when $a\in(-d,-1]$. This conjecture leads to several difficulties, e.g., when $a\in (-d,-1)$, it would be required to show that $\partial_{x_1}D^{a}u\in H^s(\mathbb{R}^d)$ for some $s\in \mathbb{R}$, which is not so immediate to obtain only assuming that $u\in H^{(\frac{d}{2}+1)^{+}}(\mathbb{R}^d)$ as this condition may depend on well-posedness results in weights Sobolev spaces as well. On this subject, the study of persistence in weighted spaces and unique continuation principles for solutions of fKdV with $a\in(-d,-1]$, $d\geq 2$ seems to require different arguments and techniques than those presented in this manuscript, therefore, its study will not be pursued here. Finally, we notice that under the same assumption of regularity and decay \eqref{class}, \eqref{japonesuv}, the results in Theorem \ref{unici} holds true for the IVP \eqref{CombHBO-ZK} with $k_j \geq 2$ be integer, $\nu_j\in\{1,-1\}$, for all $j=1,\dots, n$.


\subsection*{Organization of the paper} Section \ref{notation} contains the notation and some fundamental commutator estimates to be used in the deduction of our main results. In this section, we also deduce Lemma \ref{lemmafracderivpolyno}, which deals with fractional derivatives estimates for a certain family of homogeneous polynomials.  Section \ref{linearSec} concerns the deduction of our results for the linear fKdV equation in weighted spaces, i.e., Lemmas \ref{linearestilemma} and \ref{unicontlemma} are deduced in this section.  The well-posedness results in Theorem \ref{LWPweights} are established in Section \ref{localweighted1}.  Section \ref{uniquep} is devoted to the deduction of the unique continuation principles stated in Theorems \ref{theoremtwotimes} and \ref{theoremthreetimes} as well as the persistence properties in Corollary \ref{sharpthree}. In this section, we also deduce Corollary \ref{corollaryCombinedfKdV}, which concerns spatial decay properties for solutions of the Cauchy problem \eqref{CombHBO-ZK}.


\section{Notation and Preliminaries}\label{notation}

We use the standard multi-index notation, $\alpha=(\alpha_1,\dots,\alpha_d) \in \mathbb{N}^d$, $\partial^{\alpha}=\partial^{\alpha_1}_{x_1}\cdots \partial_{x_d}^{\alpha_d}$, $|\alpha|=\sum_{j=1}^d \alpha_j$, $\alpha!=\alpha_1 ! \cdots \alpha_d !$ and $\alpha \leq \beta$ if $\alpha_j \leq \beta_j$ for all $j=1,\dots,d$. We say $a \lesssim b$ if there exists a constant $c>0$ such that $a\leq c b$. By $a\thicksim b$, we mean that $a\lesssim b$ and $b\lesssim a$. We write $a \les_{l} b$  to indicate that the implicit constant depends on the parameter $l$.  We denote the commutator between the operators $A$ and $B$ by $[A,B]=AB-BA$.

The Fourier transform is defined by $\widehat{f}(\xi)=\mathcal{F}(f)(\xi)=\int e^{-i\xi\cdot x}f(x)\, dx$. Given $s\in \mathbb{R}$, the Bessel operator $J^s$ is defined via the Fourier multiplier $\langle \xi \rangle^{s}=(1+|\xi|^2)^{\frac{s}{2}}$. As usual, the $L^2$-based Sobolev space $H^s(\mathbb{R}^d)$ is defined by the norm $\|f\|_{H^s}=\|\langle\xi \rangle^s\widehat{f}\|_{H^s}\sim (\|f\|_{L^2}+\|D^sf\|_{L^2})$, and the homogeneous Sobolev space $\dot{H}^s(\mathbb{R}^d)$ is defined by $\|f\|_{\dot{H}^s}=\|D^s f\|_{L^2}=\||\xi|^s \widehat{f}\|_{L^2}$. 
\\ \\
$\mathcal{R}_j=-\frac{\partial_{x_j}}{D}$ denotes the Riesz transform operator in the $j$-variable, and $\mathcal{H}$ stands for the Hilbert transform operator.

\subsection{Commutator estimates}

In this part, we introduce the commutator estimates required to control the dispersive and nonlinear terms in the equation in \eqref{HBO-ZK}.

We will use the following fractional Leibniz rule, see \cite[Theorem 1]{GrafakosOh2014} (see also, \cite{KatoPonce1988}).
\begin{lemma} Let $s>0$, $1<p<\infty$, then 
\begin{equation}\label{fLR}
\begin{aligned}
\|D^s(h g)\|_{L^p}\lesssim \|D^s h\|_{L^p}\|g\|_{L^{\infty}}+\|h\|_{L^{\infty}}\|D^s g\|_{L^p}, 
\end{aligned}
\end{equation}
\end{lemma}
We recall the following generalization of Calder\'on's first commutator estimate (see \cite{Calderon1965}) in the context of the Riesz transform. The proof of this result can be consulted in \cite[Proposition 1.2]{OscarWHBO}.
\begin{proposition}\label{propconmu}
Let $\mathcal{R}_l$ be the usual Riesz transform in the direction $l=1,\dots,d$, $d\geq 2$. For any $1<p<\infty$ and any multi-index $\alpha$ with $|\alpha|\geq 1$, there exists a constant $c$ depending on $\alpha$ and $p$ such that
\begin{equation}\label{conmuest}
\begin{aligned}
\Big\Vert \mathcal{R}_l(g\partial^{\alpha}f)-g \mathcal{R}_l\partial^{\alpha}f-\sum_{1\leq |\beta| < |\alpha|}\frac{1}{\beta!}\partial^{\beta}g D_{R_l}^{\beta}\partial^{\alpha}f \Big\Vert_{L^p} \leq c_{\alpha,p} \sum_{|\beta|=|\alpha|}\left\|\partial^{\beta}g\right\|_{L^{\infty}}\left\|f\right\|_{L^p}.
\end{aligned}
\end{equation}
The operator $D_{R_l}^{\beta}$ is defined via its Fourier transform as
\begin{equation}\label{diffeOperRie}
    \widehat{D_{R_l}^{\beta}f}(\xi)=i^{-|\beta|}\partial^{\beta}_{\xi}\left(\frac{-i\xi_l}{|\xi|}\right)\widehat{f}(\xi).
\end{equation}
\end{proposition}
Above, we will follow the standard convention for the empty summation (such as $\sum_{1\leq |\beta|<1}(\cdots)$) is defined as zero. Next, we present some estimates for the operators $D_{R_l}^{\beta}$ introduced above.
\begin{lemma}\label{lemmaRieszdeco} Let  $\alpha$ and $\beta$ be multi-indexes and  $f\in \dot{H}^{|\alpha|-|\beta|}(\mathbb{R}^d)$. Then there exist constants $c_{\sigma} \in \mathbb{R}$ such that
\begin{equation}
D_{R_1}^{\beta}(\partial^{\alpha}f)=\sum_{\sigma} c_{\sigma} \mathcal{R}_{\sigma}(D^{|\alpha|-|\beta|}f),
\end{equation}
where the sum runs over all index $\sigma=(\sigma_{1},\dots,\sigma_{|\alpha|+|\beta|+1})$ with integer components such that $1\leq \sigma_j \leq d$, $j=1,\dots,|\alpha|+|\beta|+1$ and we denote by
$$\mathcal{R}_{\sigma}=\mathcal{R}_{\sigma_1}\cdots \mathcal{R}_{\sigma_{|\alpha|+|\beta|+1}}.$$
\end{lemma}
The proof of the previous lemma can be consulted in \cite[Lemma 2.1]{OscarWHBO}. 
\\ \\
In the one-dimensional setting, the following version of Proposition \ref{propconmu} was deduced in \cite[Lemma 3.1]{DawsonMCPON}.
\begin{proposition}\label{CalderonComGU}
Let $1<p<\infty$ and $l,m \in \mathbb{Z}^{+}\cup \{0\}$, $l+m\geq 1$ then
\begin{equation}\label{Comwellprel1}
    \|\partial_x^l[\mathcal{H},g]\partial_x^{m}f\|_{L^p(\mathbb{R})} \lesssim_{p,l,m} \|\partial_x^{l+m} g\|_{L^{\infty}(\mathbb{R})}\|f\|_{L^p(\mathbb{R})}.
\end{equation}
\end{proposition}
We also require some estimates involving homogeneous fractional derivatives. 
\begin{proposition}\label{fractionalDeriv}
Let $s>0$ and $1<p<\infty$. Then for any $s_1,s_2\geq 0$, with $s_1+s_2=s$, it follows 
\begin{equation}
\|D^s(fg)-\sum_{|\beta|\leq s_1} \frac{1}{\beta!} \partial^{\beta}f D^{s,\beta} g-\sum_{|\beta|< s_2} \frac{1}{\beta!} \partial^{\beta}g D^{s,\beta} f\|_{L^p}\lesssim \|D^{s_1}f\|_{L^{\infty}}\|D^{s_2}g\|_{L^{p}}.
\end{equation}
The operator $D^{s,\beta}$ is defined via the Fourier transform as
\begin{equation*}
\widehat{D^{s,\beta}g}(\xi)=\widehat{D^{s,\beta}}(\xi)\widehat{g}(\xi), \, \, \text{ and } \, \,  \widehat{D^{s,\beta}}(\xi)=i^{-|\beta|}\partial_{\xi}^{\beta}(|\xi|^s).
\end{equation*}
\end{proposition}
Proposition \ref{fractionalDeriv} is a particular case of \cite[Theorem 1.2]{DonLi2019} (see also \cite{KPV1993}). The next proposition was also deduced in \cite{DonLi2019,KPV1993}.
\begin{proposition}\label{commutatorestim1}
Let $0<s< 1$, $1<p<\infty$,
\begin{equation}
    \|[D^s,g]f\|_{L^p}\lesssim \|D^sg\|_{L^{\infty}}\|f\|_{L^p}.
\end{equation}
\end{proposition}
We also require the following commutator estimate.
\begin{proposition}\label{propcomm1}
For any $0\leq \alpha<1$, $0<\beta\leq 1-\alpha$, $1<p<\infty$, we have
\begin{equation}
\|D^{\alpha}[D^{\beta},g]D^{1-(\alpha+\beta)}f\|_{L^p}\lesssim \|\nabla g\|_{L^{\infty}}\|f\|_{L^p}.
\end{equation}
\end{proposition}
Proposition \ref{propcomm1} is proved in \cite[Proposition 3.10]{DonLi2019}.


\subsection{Preliminaries fractional derivatives estimates}

In this part, we introduce some key estimates dealing with fractional derivatives.

$L^{p}_{s}$ denotes the Sobolev space determined by $L^{p}_{s}:=(1-\Delta)^{-\frac{s}{2}}L^{p}(\R^d)$. This space can be characterized by the Stein derivative according to the following theorem.
\begin{theorem}\label{stein}
	Let $b\in (0,1)$ and $\frac{2d}{(d+2b)}<p<\infty.$ Then $f\in L^{p}_{b}(\R^{d})$ if and only if
	\begin{itemize}
		\item [a)] $f\in L^{p}(\R^{d}),$
		\item [b)]
		$\mathcal{D}^{b}f(x):={\displaystyle \left (
			\int_{\R^{d}}\frac{|f(x)-f(y)|^{2}}{|x-y|^{d+2b}}dy\right)^{1/2}} \in
		L^{p}(\R^{d}),$ with
		\begin{equation}\label{equiv}
		\|f\|_{b,p}:=\|J^{b}f\|_{p}\thicksim \|f\|_{p}+\|D^{b}f\|_{p}\thicksim \|f\|_{p}+\|\mathcal{D}^{b}f\|_{p}.
		\end{equation}
	\end{itemize}
\end{theorem}
\begin{proof}
	See \cite[Theorem 1]{Stein1961}.
\end{proof}
Let us now present some consequences of the above theorem. When $p=2$, we have the following product estimate (see \cite[Proposition 1]{NahasPonce2009})
\begin{equation}\label{Leibaniso}
\|\mathcal{D}^{b}(fg)\|_{L^2} \leq \|f\mathcal{D}^{b}g\|_{L^2} + \|g\mathcal{D}^{b}f\|_{L^2},
\end{equation}
and the $L^{\infty}$ estimates
\begin{equation}\label{gradaniso}
\|\mathcal{D}^{b}h\|_{L^\infty} \les (\|h\|_{L^{\infty}}+\|\p_{x_j} h\|_{L^\infty}),
\end{equation}
We also have that
\begin{equation}\label{propereq1}
\|\mathcal{D}^bf\|_{L^2}\sim \|D^bf \|_{L^2}.    
\end{equation}
As a further consequence of Theorem \ref{stein}, we have the following interpolation result.
\begin{lemma}\label{interaniso}
Let $a,b>0$. Assume that $\langle x\rangle^b f \in L^{2}(\R^d)$ and $J^{a}f\in L^{2}(\R^d)$. Then for any
$\theta \in [0,1]$
\begin{equation}
\|J^{\theta a}(\langle {x}\rangle^{(1-\theta)b}f)\|_{L^2}\lesssim \|\langle x\rangle^b f\|_{L^2}^{1-\theta}\|J^{a}f\|_{L^2}^{\theta}.
\end{equation}

\end{lemma}
\begin{proof}
The proof can be consulted in \cite[Lemma 1]{FonPO}.
\end{proof}

\begin{lemma}\label{decaylinearexp}
Let $0<b<1$, $0<a<1$, then
\begin{equation*}
    \mathcal{D}^b\big(e^{itx_1|x|^a}\big)(\xi)\lesssim \langle t \rangle^{b}\langle\xi\rangle^{ab}.
\end{equation*}
\end{lemma}
\begin{proof}
The proof follows similar arguments in \cite[Proposition 2]{NahasPonce2009}.
\end{proof}

\begin{lemma}\label{lemmafracderivpolyno}
Let $k_1,k_2, l\geq 0$ be integers, $0<a<2$, and $P_{k_2+l}(\xi)$ be a homogeneous polynomial of order $k_2+l$.

\begin{itemize}
\item[i)] Assume   
\begin{equation}\label{propert1}
 -\frac{d}{2}<al+l+k_2-2k_1<0.   
\end{equation}
Then it follows
\begin{equation*}
\mathcal{D}^b\big(|\xi|^{-2k_1}|\xi|^{al}P_{k_2+l}(\xi)\big)(x)\lesssim |x|^{al+l+k_2-2k_1-b},    
\end{equation*}
for all $x\neq 0$.

\item[ii)] Let $\phi\in C^{\infty}_c(\mathbb{R}^d)$ with $\phi(\xi)=1$, whenever $|\xi|\leq 1$. Assume
\begin{equation}\label{propert2}
 al+l+k_2-2k_1\geq 0.   
\end{equation}
Then it follows
\begin{equation*}
\begin{aligned}
\mathcal{D}^b\big(|\xi|^{-2k_1}&|\xi|^{al} P_{k_2+l}(\xi)\phi(\xi)\big)(x)\\
&\lesssim 
\left\{\begin{aligned}
& 1+|x|^{b}+(-\log(|x|))^{\frac{1}{2}},\hspace{1cm} \text{ if } \, \, al+l+k_2-2k_1-b=0 \, \, \text{and } |x|\ll 1,\\
&1+|x|^{al+l+k_2-2k_1}+|x|^{al+l+k_2-2k_1-b},\hspace{0.3cm} \text{ otherwise},   
\end{aligned}\right.
\end{aligned}
\end{equation*}
for all $x\neq 0$.
\end{itemize}
\end{lemma}

\begin{proof}
For simplicity, we denote by $P(\xi):=|\xi|^{-2k_1}|\xi|^{al}P_{k_2+l}(\xi)$.  We use the definition of $\mathcal{D}^b$ to write
\begin{equation*}
\big(\mathcal{D}^b(P(\xi))(x)\big)^2=\int\frac{|P(x)-P(y)|^2}{|x-y|^{d+2b}}\, dy=\sum_{j=1}^4\int_{A_j}(\cdots),    
\end{equation*}
where $A_1=\{y\in \mathbb{R}^d: \, |y|\geq 2|x|\}$, $A_2=\{y\in \mathbb{R}^d: \, |y|\leq \frac{|x|}{4}\}$, $A_3=\{y\in \mathbb{R}^d: \, \frac{|x|}{4}< |y|< 2|x|, \, |x-y|\geq \frac{|x|}{4}\}$, and $A_4=\{y\in \mathbb{R}^d: \, \frac{|x|}{4}< |y|< 2|x|, \, |x-y|\leq \frac{|x|}{4}\}$.  We first deal with i). Since $al+l+k_2-2k_1<0$, and $|y|\gg |x|$ in $A_1$, it follows
\begin{equation*}
|P(y)|\lesssim |x|^{al+l+k_2-2k_1},    
\end{equation*}
thus we get
\begin{equation*}
\begin{aligned}
\int_{A_1}\frac{|P(x)-P(y)|^2}{|x-y|^{d+2b}}\, dy\lesssim & |x|^{2(al+l+k_2-2k_1)}\int_{|y|\geq 2|x|}\frac{1}{|y|^{d+2b}}\, dy\\
\lesssim & |x|^{2(al+l+k_2-2k_1-b)}.
\end{aligned}    
\end{equation*}
On the regions $A_j$, $j=2,3$, we have $|y|\lesssim |x|$ with $|x-y|\sim |x|$, thus $|P(x)|\lesssim |y|^{al+l+k_2-2k_1}$. This implies
\begin{equation*}
\begin{aligned}
\int_{A_j}\frac{|P(x)-P(y)|^2}{|x-y|^{d+2b}}\, dy\lesssim & |x|^{-d-2b}\int_{|y|\leq \frac{|x|}{4}}|y|^{2(al+l+k_2-2k_1)}\, dy\\
\lesssim & |x|^{2(al+l+k_2-2k_1-b)},
\end{aligned}    
\end{equation*}
where we have used that $al+l+k_2-2k_1+\frac{d}{2}>0$. Finally, on the region $A_4$, we have $|x|\sim |y|$ with $|x-y|\leq \frac{|x|}{4}$, then the mean value inequality with $y\in A_4$ yields
\begin{equation}\label{Pineq}
\begin{aligned}
|P(x)-P(y)|\lesssim |x|^{al+l+k_2-2k_1-1}|x-y|.
\end{aligned}    
\end{equation}
The previous inequality allows us to conclude
\begin{equation*}
\begin{aligned}
  \int_{A_4}\frac{|P(x)-P(y)|^2}{|x-y|^{d+2b}}\, dy\lesssim & |x|^{2(al+l+k_2-2k_1-1)} \int_{ |x-y|\leq \frac{|x|}{4}} |x-y|^{2-d-2b}\, dy\\
  \lesssim & |x|^{2(al+l+k_2-2k_1-b)}.
\end{aligned}
\end{equation*}
Gathering the previous estimates, we complete the proof of i).

Next, we deal with the deduction of ii). We assume $al+l+k_2-2k_1\geq 0$. We use the same decomposition as above given by the sets $A_1$, $A_2$, $A_3$ and $A_4$. By writing 
\begin{equation}\label{decomp1}
P(x)\phi(x)-P(y)\phi(y)=P(x)(\phi(x)-\phi(y))+(P(y)-P(x))\phi(y),
\end{equation}
and using that $|y|\gg |x|$ in $A_1$, we deduce
\begin{equation}
\begin{aligned}
\int_{A_1} &\frac{|P(x)\phi(x)-P(y)\phi(y)|^2}{|x-y|^{d+2b}}\, dy\\
&\lesssim  |x|^{2(al+l+k_2-2k_1)}\big(\mathcal{D}^b(\phi)(x)\big)^2+\int_{|y|\geq 2|x|}|y|^{2(al+l+k_2-2k_1-b)-d}|\phi(y)|^2\, dy.
\end{aligned}    
\end{equation}
Using that $\phi$ is compactly supported, it is seen that
\begin{equation*}
\int_{|y|\geq 2|x|}|y|^{2(al+l+k_2-2k_1-b)-d}|\phi(y)|^2\, dy \lesssim \left\{\begin{aligned}
&|x|^{2(al+l+k_2-2k_1-b)}, \hspace{0.5cm} al+l+k_2-2k_1-b<0,\\
&-\log(|x|), \hspace{0.7cm} al+l+k_2-2k_1-b=0,\, |x|\ll 1,\\
&1, \hspace{2cm} al+l+k_2-2k_1-b=0, \, |x|\gtrsim 1, \\
&1, \hspace{2cm} al+l+k_2-2k_1-b>0.
\end{aligned}\right.
\end{equation*}
On the region $A_j$, $j=2,3$, $|y|\lesssim |x|$ with $|x-y|\sim |x|$, it follows
\begin{equation*}
\begin{aligned}
\int_{A_j}\frac{|P(x)\phi(x)-P(y)\phi(y)|^2}{|x-y|^{d+2b}}\, dy\lesssim & |x|^{2(al+l+k_2-2k_1-b)-d}\int_{|y|\leq \frac{|x|}{4}}\big(1+ |\phi(y)|^2\big) \, dy\\
\lesssim & |x|^{2(al+l+k_2-2k_1-b)}.
\end{aligned}    
\end{equation*} 
Finally, on the region $A_4$, we use \eqref{Pineq} and \eqref{decomp1} to get
\begin{equation}
\begin{aligned}
\int_{A_4} &\frac{|P(x)\phi(x)-P(y)\phi(y)|^2}{|x-y|^{d+2b}}\, dy\\
&\lesssim  |x|^{2(al+l+k_2-2k_1)}\big(\mathcal{D}^b(\phi)(x)\big)^2+|x|^{2(al+l+k_2-2k_1)-2}\int_{|x-y|\leq \frac{|x|}{4}}|x-y|^{2-d-2b}\, dy\\
&\lesssim  |x|^{2(al+l+k_2-2k_1)}+|x|^{2(al+l+k_2-2k_1-b)}.
\end{aligned}    
\end{equation} 
This completes the deduction of ii). 
\end{proof}


\section{Proof of Lemmas \ref{linearestilemma} and \ref{unicontlemma}}\label{linearSec}

In this part, we first deduce Lemma \ref{linearestilemma}, which deals with spatial decay properties of the group of operators $U(t)$. After that, we verify the sharpness of condition \eqref{linearEstimcomp1} in Lemma \ref{unicontlemma}.

\subsection{Proof of Lemma \ref{linearestilemma}}

We first deduce part (i) of Lemma \ref{linearestilemma}.

\begin{proof}[Proof of Lemma \ref{linearestilemma} (i)] We will assume that $f$ is a Schwartz function as the general case follows by approximation to our estimates. We begin with the deduction of \eqref{linereq1}. We write $r=r_1+r_2$, with $r_1\in \mathbb{Z}^{+}\cup\{0\}$, $r_2\in [0,1)$. By using Plancherel's identity and the relation between decay and regularity connecting the Fourier and the spatial domains, we have
\begin{equation}\label{Planchereleq1}
\begin{aligned}
\|\langle x\rangle^r U(t)f\|_{L^2}=&\|J^{r}\big(e^{i\xi_1|\xi|^a t}\widehat{f}\big)\|_{L^2}\\
\lesssim & \|e^{i\xi_1|\xi|^a t}\widehat{f}\|_{L^2}+\sum_{\substack{|\beta_1|+|\beta_2|\leq r_1}}\|D^{r_2}_{\xi}\big(\partial^{\beta_1}(e^{i\xi_1|\xi|^a t})\partial^{\beta_2}\widehat{f}\big)\|_{L^2}.
\end{aligned}    
\end{equation}
To estimate the second term on the right-hand side of the above inequality, for a given $|\beta|\geq 1$, we use the following identity
\begin{equation}\label{lineqeq1}
\begin{aligned}
\partial^{\beta}(e^{i\xi_1|\xi|^a t})=&|\xi|^{-2|\beta|}\sum_{l=1}^{|\beta|}t^{l}|\xi|^{al}P_{|\beta|+l}(\xi)e^{i\xi_1|\xi|^a t}\\
=:&Q_{\beta}(a,t,\xi)e^{i\xi_1|\xi|^a t},
\end{aligned}
\end{equation}
$\xi \neq 0$, where $P_{|\beta|+l}(\xi)$ denotes a homogeneous polynomial of order $|\beta|+l$, with $P_{|\beta|+1}(\xi)\neq 0$. The deduction of \eqref{lineqeq1} follows by an inductive argument on the order of the multi-index $|\beta|\geq 1$, we omit its inference. To be consistent with the notation, we also denote by $Q_{0}(a,t,\xi)=1$.  Thus, setting $|\beta_1|+|\beta_2|\leq r_1$, we use properties \eqref{Leibaniso}, and \eqref{propereq1} to get
\begin{equation}\label{lineqeq1.0}
\begin{aligned}
\|D^{r_2}_{\xi}\big(\partial^{\beta_1}&(e^{i\xi_1|\xi|^a t})\partial^{\beta_2}\widehat{f}\big)\|_{L^2}\\
\lesssim & \|\mathcal{D}^{r_2}_{\xi}\big(Q_{\beta_1}(a,t,\xi)e^{i\xi_1|\xi|^a t}\partial^{\beta_2}\widehat{f}\big)\|_{L^2}\\
\lesssim & \|\mathcal{D}^{r_2}_{\xi}(e^{i\xi_1|\xi|^a t})Q_{\beta_1}(a,t,\xi)\partial^{\beta_2}\widehat{f}\|_{L^2}+  \|\mathcal{D}^{r_2}_{\xi}\big(Q_{\beta_1}(a,t,\xi)\partial^{\beta_2}\widehat{f}\big)\|_{L^2}\\
\lesssim & \langle t \rangle^{r_2}\|\langle\xi \rangle^{a r_2}Q_{\beta_1}(a,t,\xi)\partial^{\beta_2}\widehat{f}\|_{L^2}+  \|\mathcal{D}^{r_2}_{\xi}\big(Q_{\beta_1}(a,t,\xi)\partial^{\beta_2}\widehat{f}\big)\|_{L^2}\\
=&\langle t \rangle^{r_2}\mathcal{I}_1(t,\beta_1,\beta_2)+\mathcal{I}_2(t,\beta_1,\beta_2),
\end{aligned}
\end{equation}
where we have also used Lemma \ref{decaylinearexp}. We first estimate the factor $\mathcal{I}_2(t,\beta_1,\beta_2)$. Notice that when $\beta_1=0$, from \eqref{propereq1}, it is seen that
\begin{equation*}
\begin{aligned}
\mathcal{I}_2(t,\beta_1,\beta_2)\lesssim \|J^{r_1+r_2}\widehat{f}\|_{L^2}\sim \|\langle x \rangle^{r_1+r_2}f\|_{L^2}.
\end{aligned}
\end{equation*}
We will assume that $|\beta_1|\geq 1$. We denote by
\begin{equation*}
\begin{aligned}
Q_{\beta_1}(a,t,\xi)=\sum_{l=1}^{|\beta_1|}|\xi|^{-2|\beta_1|}t^l |\xi|^{al}P_{|\beta_1|+l}(\xi)=:\sum_{l=1}^{|\beta_1|}Q_{\beta_1,l}(a,t,\xi),
\end{aligned}    
\end{equation*}
thus by using \eqref{propereq1}, we have
\begin{equation*}
\begin{aligned}
\mathcal{I}_2(t,\beta_1,\beta_2)\lesssim \sum_{l=1}^{|\beta_1|}\|\mathcal{D}^{r_2}_{\xi}\big(Q_{\beta_1,l}(a,t,\xi)\partial^{\beta_2}\widehat{f}\big)\|_{L^2}=:\sum_{l=1}^{|\beta_1|}\mathcal{I}_{2,l}(t,\beta_1,\beta_2).
\end{aligned}    
\end{equation*}
We divide the estimate of $\mathcal{I}_{2,l}(t,\beta_1,\beta_2)$ into three cases:
\begin{itemize}
    \item[(a)] $al+l-|\beta_1|\geq 1$,
    \item[(b)] $0\leq al+l-|\beta_1|< 1$,
    \item[(c)] $al+l-|\beta_1|< 0$.
\end{itemize}

\underline{\bf Case (a): $al+l-|\beta_1|\geq 1$}. By \eqref{Leibaniso}, and \eqref{gradaniso}, we deduce
\begin{equation*}
\begin{aligned}
\mathcal{I}_{2,l}(t,\beta_1,\beta_2)=&\|\mathcal{D}_{\xi}^{r_2}\big(\langle \xi \rangle^{-(al+l-|\beta_1|)}Q_{\beta_1,l}(a,t,\xi))\langle \xi \rangle^{al+l-|\beta_1|}\partial^{\beta_2}\widehat{f}\big)\|_{L^2}\\
\lesssim & \big( \|\langle \xi \rangle^{-(al+l-|\beta_1|)}Q_{\beta_1,l}(a,t,\xi)\|_{L^{\infty}}+\|\nabla\big(\langle \xi \rangle^{-(al+l-|\beta_1|)}Q_{\beta_1,l}(a,t,\xi)\big)\|_{L^{\infty}}\big)\\
&\times \|\langle \xi \rangle^{al+l-|\beta_1|}\partial^{\beta_2}\widehat{f}\|_{L^2}+\|D^{r_2}\big(\langle \xi \rangle^{al+l-|\beta_1|}\partial^{\beta_2}\widehat{f}\big)\|_{L^2}\\
\lesssim & \langle t \rangle^l\big(\|\langle \xi \rangle^{al+l-|\beta_1|}\partial^{\beta_2}\widehat{f}\|_{L^2}+\|D^{r_2}\big(\langle \xi \rangle^{al+l-|\beta_1|}\partial^{\beta_2}\widehat{f}\big)\|_{L^2}\big).
\end{aligned}    
\end{equation*}
We emphasize that $al+l-|\beta_1|\geq 1$ guarantees that we can use \eqref{gradaniso} with $h(\xi)=\langle \xi \rangle^{-(al+l-|\beta_1|)}Q_{\beta_1,l}(a,t,\xi)$.  To complete the analysis of the above inequality, by distributing the derivative of order $\beta_2$, together with multiple applications of \eqref{Leibaniso}, and \eqref{gradaniso}, it is not difficult to deduce
\begin{equation}\label{lineeqe1.1}
\begin{aligned}
\|\langle \xi \rangle^{al+l-|\beta_1|}\partial^{\beta_2}\widehat{f}\|_{L^2}&+\|D^{r_2}\big(\langle \xi \rangle^{al+l-|\beta_1|}\partial^{\beta_2}\widehat{f}\big)\|_{L^2}\\
&\lesssim \sum_{\beta_{2,1}+\beta_{2,2}=\beta_2} \|J^{r_2}\partial^{\beta_{2,1}}\big(\langle \xi \rangle^{al+l-|\beta_1|-|\beta_{2,2}|}\widehat{f}\big)\|_{L^2}.
\end{aligned}    
\end{equation}
Now, if $al+l-|\beta_1|-|\beta_{2,2}|\leq 0$, we distribute the integer and the fractional derivatives, using \eqref{Leibaniso} and \eqref{gradaniso} in the fractional case to deduce
\begin{equation*}
\begin{aligned}
\|J^{r_2}\partial^{\beta_{2,1}}\big(\langle \xi \rangle^{al+l-|\beta_1|-|\beta_{2,2}|}\widehat{f}\big)\|_{L^2}\lesssim & \|J^{r_2+|\beta_{2,1}|}\widehat{f}\|_{L^2}\\
\lesssim &\|J^{r_1+r_2}\widehat{f}\|_{L^2}\sim \|
\langle x \rangle^{r_1+r_2}f\|_{L^2}.
\end{aligned}    
\end{equation*}
If $al+l-|\beta_1|-|\beta_{2,2}|> 0$, we use interpolation Lemma \ref{interaniso} to get
\begin{equation}\label{lineeqe1.1.1}
\begin{aligned}
\|J^{r_2}\partial^{\beta_{2,1}}\big(\langle \xi \rangle^{al+l-|\beta_1|-|\beta_{2,2}|}\widehat{f}\big)\|_{L^2} &\lesssim \|J^{r_2+|\beta_{2,1}|}\big(\langle \xi \rangle^{al+l-|\beta_1|-|\beta_{2,2}|}\widehat{f}\big)\|_{L^2}\\
&\lesssim \|J^{r_1+r_2}\widehat{f}\|_{L^2}+\|\langle \xi \rangle^{\frac{(r_1+r_2)(al+l-|\beta_1|-|\beta_{2,2}|)}{r_1-|\beta_{2,1}|}}\widehat{f}\|_{L^2}\\
&\lesssim \|J^{r_1+r_2}\widehat{f}\|_{L^2}+\|\langle \xi \rangle^{(r_1+r_2)a}\widehat{f}\|_{L^2}\\
&\sim \|\langle x \rangle^{r_1+r_2} f\|_{L^2}+\|J^{(r_1+r_2)a} f\|_{L^2},
\end{aligned}    
\end{equation}
where given that $|\beta_1|+|\beta_2|\leq r_1,$ $0\leq |\beta_{2,1}|,|\beta_{2,2}|\leq |\beta_2|$, $1\leq l\leq |\beta_1|$, we have used $al+l-|\beta_1|-|\beta_{2,2}|\leq a(r_1-|\beta_{2,1}|)$.  This completes the estimate of $\mathcal{I}_{2,l}(t,\beta_1,\beta_2)$ in the case $al+l-|\beta_1|\geq 1$.
\\ \\ 
\underline{\bf Case (b): $0\leq al+l-|\beta_1|< 1$}.   We consider $\phi \in C^{\infty}_c(\mathbb{R}^d)$ with $\phi(\xi)=1$, whenever $|\xi|\leq 1$. By \eqref{Leibaniso} and using the function $\phi$ to divide the domain of integration, we split our arguments as follows
\begin{equation}\label{lineeqe1.2}
\begin{aligned}
\mathcal{I}_{2,l}&(t,\beta_1,\beta_2)\\
\lesssim & \|\mathcal{D}_{\xi}^{r_2}\big(Q_{\beta_1,l}(a,t,\xi)\phi\partial^{\beta_2}\widehat{f}\big)\|_{L^2}+\|\mathcal{D}_{\xi}^{r_2}\big(Q_{\beta_1,l}(a,t,\xi)(1-\phi)\partial^{\beta_2}\widehat{f}\big)\|_{L^2}\\
\lesssim &\|\mathcal{D}_{\xi}^{r_2}\big(Q_{\beta_1,l}(a,t,\xi)\phi)\partial^{\beta_2}\widehat{f}\|_{L^2}+\|Q_{\beta_1,l}(a,t,\xi)\phi \mathcal{D}_{\xi}^{r_2}(\partial^{\beta_2}\widehat{f})\|_{L^2}\\
&+ \big( \|\langle \xi \rangle^{-(al+l-|\beta_1|)}Q_{\beta_1,l}(a,t,\xi)(1-\phi)\|_{L^{\infty}}+\|\nabla\big(\langle \xi \rangle^{-(al+l-|\beta_1|)}Q_{\beta_1,l}(a,t,\xi)(1-\phi)\big)\|_{L^{\infty}}\big)\\
&\times \|\langle \xi \rangle^{al+l-|\beta_1|}\partial^{\beta_2}\widehat{f}\|_{L^2}+\|D^{r_2}\big(\langle \xi \rangle^{al+l-|\beta_1|}\partial^{\beta_2}\widehat{f}\big)\|_{L^2}\\
\lesssim & \|\mathcal{D}_{\xi}^{r_2}\big(Q_{\beta_1,l}(a,t,\xi)\phi)\partial^{\beta_2}\widehat{f}\|_{L^2}\\
&+ \langle t \rangle^l\big(\|\mathcal{D}_{\xi}^{r_2}\partial^{\beta_2}\widehat{f}\|_{L^2}+\|\langle \xi \rangle^{al+l-|\beta_1|}\partial^{\beta_2}\widehat{f}\|_{L^2}+\|D^{r_2}\big(\langle \xi \rangle^{al+l-|\beta_1|}\partial^{\beta_2}\widehat{f}\big)\|_{L^2}\big).
\end{aligned}    
\end{equation}
By following the same strategy in the study of \eqref{lineeqe1.1}, i.e., using interpolation inequality, we get
\begin{equation*}
\begin{aligned}
\|\mathcal{D}_{\xi}^{r_2}\partial^{\beta_2}\widehat{f}\|_{L^2}&+\|\langle \xi \rangle^{al+l-|\beta_1|}\partial^{\beta_2}\widehat{f}\|_{L^2}+\|D^{r_2}\big(\langle \xi \rangle^{al+l-|\beta_1|}\partial^{\beta_2}\widehat{f}\big)\|_{L^2} \\
& \lesssim \|\langle x \rangle^{r_1+r_2}f\|_{L^2}+\|J^{(r_1+r_2)a}f\|_{L^2}.
\end{aligned}    
\end{equation*}
It only remains to estimate the first term on the right-hand side of \eqref{lineeqe1.2}. We consider two extra cases: $al+l-|\beta_1|-r_2\neq 0$ and $al+l-|\beta_1|-r_2= 0$.
\\ \\
\underline{Assume $al+l-|\beta_1|-r_2\neq 0$}. By ii) in Lemma \ref{lemmafracderivpolyno}, it follows 
\begin{equation*}
\begin{aligned}
\|\mathcal{D}_{\xi}^{r_2}&\big(Q_{\beta_1,l}(a,t,\xi)\phi)\partial^{\beta_2}\widehat{f}\|_{L^2}\\
&\lesssim \langle t \rangle^l\big(\|\partial^{\beta_2}\widehat{f}\|_{L^2}+\||\xi|^{al+l-|\beta_1|}\partial^{\beta_2}\widehat{f}\|_{L^2}+\||\xi|^{al+l-|\beta_1|-r_2}\partial^{\beta_2}\widehat{f}\|_{L^2} \big)\\
&\lesssim \langle t \rangle^l\big(\|\partial^{\beta_2}\widehat{f}\|_{L^2}+\|\langle \xi \rangle^{al+l-|\beta_1|}\partial^{\beta_2}\widehat{f}\|_{L^2}+\||\xi|^{al+l-|\beta_1|-r_2}\partial^{\beta_2}\widehat{f}\|_{L^2} \big).
\end{aligned}    
\end{equation*}
The first term on the right-hand side of the above expression is bounded as required. The second term can be estimated by distributing the derivative and using interpolation inequality as we did in \eqref{lineeqe1.1.1}. To estimate the last factor on the right-hand side of the above inequality, we first assume $al+l-|\beta_1|-r_2>0$. By distributing the derivative $\beta_2$, we deduce
\begin{equation*}
\begin{aligned}
\||\xi|^{al+l-|\beta_1|-r_2}\partial^{\beta_2}\widehat{f}\|_{L^2}\lesssim & \|\langle \xi\rangle^{al+l-|\beta_1|-r_2}\partial^{\beta_2}\widehat{f}\|_{L^2}\\
\lesssim & \sum_{\beta_{2,1}+\beta_{2,2}=\beta_2} \|\partial^{\beta_{2,1}}\big(\langle \xi \rangle^{al+l-|\beta_1|-r_2-|\beta_{2,2}|}\widehat{f}\big)\|_{L^2}\\
\lesssim & \|J^{|\beta_2|}\widehat{f}\|_{L^2}+\sum_{\substack{|\beta_{2,1}|+|\beta_{2,2}|=|\beta_2|\\ al+l-|\beta_1|-r_2-|\beta_{2,2}|>0}} \|J^{|\beta_{2,1}|}\big(\langle \xi\rangle^{al+l-|\beta_1|-r_2-|\beta_{2,2}|}\widehat{f}\big)\|_{L^2}.
\end{aligned}    
\end{equation*}
By interpolation Lemma \ref{interaniso} and the fact that $al+l-|\beta_1|-r_2-|\beta_{2,2}|\leq a(r_1+r_2-|\beta_{2,1}|)$, we find
\begin{equation}\label{lineeqe1.2.1}
\begin{aligned}
\|J^{|\beta_{2,1}|}\big(\langle \xi\rangle^{al+l-|\beta_1|-r_2-|\beta_{2,2}|}\widehat{f}\big)\|_{L^2}\lesssim & \|J^{r_1+r_2}\widehat{f}\|_{L^2}+\|\langle \xi \rangle^{\frac{(r_1+r_2)(al+l-|\beta_1|-r_2-|\beta_{2,2}|)}{r_1+r_2-|\beta_{2,1}|}}\widehat{f}\|_{L^2}\\
\lesssim & \|J^{r_1+r_2}\widehat{f}\|_{L^2}+\|\langle \xi \rangle^{a(r_1+r_2)}\widehat{f}\|_{L^2}\\
\sim & \|\langle x \rangle^{r_1+r_2}f\|_{L^2}+\|J^{a(r_1+r_2)}f\|_{L^2}.
\end{aligned}    
\end{equation}
Now, we assume $al+l-|\beta_1|-r_2<0$. We use H\"older's inequality and Hardy-Littlewood-Sobolev inequality to get
\begin{equation}\label{lineeqe1.3}
\begin{aligned}
\||\xi|^{al+l-|\beta_1|-r_2}\partial^{\beta_2}\widehat{f}\|_{L^2}\lesssim & \||\xi|^{al+l-|\beta_1|-r_2}\phi \partial^{\beta_2}\widehat{f}\|_{L^2}+\||\xi|^{al+l-|\beta_1|-r_2}(1-\phi)\partial^{\beta_2}\widehat{f}\|_{L^2}\\
\lesssim & \||\xi|^{al+l-|\beta_1|-r_2}\phi\|_{L^{p_1}}\| \partial^{\beta_2}\widehat{f}\|_{L^{p_2}}+\|\partial^{\beta_2}\widehat{f}\|_{L^2}\\
\lesssim & \||\xi|^{al+l-|\beta_1|-r_2}\phi\|_{L^{p_1}}\|D^{s_1}\partial^{\beta_2}\widehat{f}\|_{L^{2}}+\|\partial^{\beta_2}\widehat{f}\|_{L^2},\\
\lesssim & \|J^{s_1+|\beta_2|}\widehat{f}\|_{L^2}\lesssim\|\langle x \rangle^{r_1+r_2}f\|_{L^2}, 
\end{aligned}    
\end{equation}
where the above inequality is justified if there exists $1<p_1,p_2<\infty$, and $s_1>0$ such that
\begin{equation*}
\left\{\begin{aligned}
&\frac{1}{2}=\frac{1}{p_1}+\frac{1}{p_2}, \hspace{1cm} \frac{1}{p_2}=\frac{1}{2}-\frac{s_1}{d}, \hspace{1cm} s_1<\frac{d}{2}, \\
& al+l-|\beta_1|-r_2+\frac{d}{p_1}>0, \hspace{1cm} s_1+|\beta_{2}|\leq r_1+r_2.
\end{aligned} \right.   
\end{equation*}
However, the above conditions are reduced to show that there exists $s_1>0$ such that $-al-l+|\beta_1|+r_2<s_1<\min\{r_1+r_2-|\beta_2|,\frac{d}{2}\}$, which holds true provided that $1 \leq l \leq |\beta_1|$, $|\beta_1|+|\beta_2|\leq r_1+r_2$, and $r_1+r_2<a+1+\frac{d}{2}$.

\underline{Assume $al+l-|\beta_1|-r_2= 0$}. Let $\widetilde{\phi}\in C^{\infty}_c(\mathbb{R}^d)$ supported on $|\xi|\leq \frac{1}{4}$, and $\phi(\xi)=1$ for $\xi$ in a small neighborhood of the origin. Then, we use Lemma \ref{lemmafracderivpolyno} to deduce 
\begin{equation*}
\begin{aligned}
\|\mathcal{D}_{\xi}^{r_2}&\big(Q_{\beta_1,l}(a,t,\xi)\phi)\partial^{\beta_2}\widehat{f}\|_{L^2}\\
&\lesssim \langle t \rangle^l\big(\|\partial^{\beta_2}\widehat{f}\|_{L^2}+\||\xi|^{r_2}\partial^{\beta_2}\widehat{f}\|_{L^2}+\|(-\log(|\xi|))^{\frac{1}{2}}\widetilde{\phi}\partial^{\beta_2}\widehat{f}\|_{L^2} \big)\\
&\lesssim \langle t \rangle^l\big(\|\langle x \rangle^{r_1+r_2}f\|_{L^2}+\|\langle \xi\rangle^{r_2}\partial^{\beta_2}\widehat{f}\|_{L^2}+\|(-\log(|\xi|))^{\frac{1}{2}}\widetilde{\phi}\partial^{\beta_2}\widehat{f}\|_{L^2} \big).
\end{aligned}    
\end{equation*}
The estimate for the second term on the right-hand side of the above expression follows similar arguments as in \eqref{lineeqe1.1.1}. The analysis of the third term follows the ideas in \eqref{lineeqe1.3}. To see this, let $0<s_2<\min\{\frac{d}{2},r_1+r_2-|\beta_2|\}$, we apply H\"older's inequality and Hardy-Littlewood-Sobolev inequality to deduce 
\begin{equation*}
\begin{aligned}
\|(-\log(|\xi|))^{\frac{1}{2}}\widetilde{\phi}\partial^{\beta_2}\widehat{f}\|_{L^2}\lesssim & \|(-\log(|\xi|))^{\frac{1}{2}}\widetilde{\phi}\|_{L^{\frac{d}{s_2}}}\|\partial^{\beta_2}\widehat{f}\|_{L^{\frac{2d}{d-2s_2}}}\\
\lesssim & \|(-\log(|\xi|))^{\frac{1}{2}}\widetilde{\phi}\|_{L^{\frac{d}{s_2}}}\|D^{s_2}\partial^{\beta_2}\widehat{f}\|_{L^{2}}\\
\lesssim & \|\langle x \rangle^{r_1+r_2}f\|_{L^2}.
\end{aligned}    
\end{equation*}

\underline{\bf Case (c): $al+l-|\beta_1|< 0$}. By using \eqref{Leibaniso} and Lemma \ref{lemmafracderivpolyno}, it is seen that
\begin{equation*}
\begin{aligned}
\mathcal{I}_{2,l}(t,\beta_1,\beta_2)\lesssim & \|\mathcal{D_{\xi}}^{r_2}(Q_{\beta_1,l}(a,t,\xi))\partial^{\beta_2}\widehat{f}\|_{L^2}+\langle t \rangle^l\||\xi|^{al+l-|\beta_1|}\mathcal{D}_{\xi}^{r_2}\partial^{\beta_2}\widehat{f}\|_{L^2}\\
\lesssim & \langle t \rangle^l\big(\||\xi|^{al+l-|\beta_1|-r_2}\phi \partial^{\beta_2} \widehat{f}\|_{L^2}+\||\xi|^{al+l-|\beta_1|}\phi \mathcal{D}_{\xi}^{r_2}\partial^{\beta_2}\widehat{f}\|_{L^2}\\
&+\|J^{|\beta_2|+r_2}\widehat{f}\|_{L^2}\big).
\end{aligned}    
\end{equation*}
The estimate of the first two terms follow by similar arguments in \eqref{lineeqe1.3}. However, there is a difference due to the presence of the fractional derivative $\mathcal{D}^{r_2}_{\xi}$. More precisely, we let $s_1,s_2>0$ such that
\begin{equation*}
-al-l+|\beta_1|+r_2<s_1<\min\{\frac{d}{2},r_1+r_2-|\beta_2|\},  
\end{equation*}
and
\begin{equation*}
-al-l+|\beta_1|<s_2<\min\{\frac{d}{2},r_1-|\beta_2|\}. 
\end{equation*}
We apply  H\"older's inequality and Hardy-Littlewood-Sobolev inequality to deduce
\begin{equation*}
\begin{aligned}
\||\xi|^{al+l-|\beta_1|-r_2}\phi \partial^{\beta_2}\widehat{f}\|_{L^2}&+\||\xi|^{al+l-|\beta_1|}\phi \mathcal{D}_{\xi}^{r_2}\partial^{\beta_2}\widehat{f}\|_{L^2}\\
\lesssim & \||\xi|^{al+l-|\beta_1|-r_2}\phi\|_{L^{\frac{d}{s_1}}}\|\partial^{\beta_2}\widehat{f}\|_{L^{\frac{2d}{d-2s_1}}}\\
&+\||\xi|^{al+l-|\beta_1|}\phi\|_{L^{\frac{d}{s_2}}}\big(\|\partial^{\beta_2}\widehat{f}\|_{L^{\frac{2d}{d-2s_2}}}+\|D^{r_2}\partial^{\beta_2}\widehat{f}\|_{L^{\frac{2d}{d-2s_2}}}\big)\\
\lesssim & \|J^{s_1+|\beta_2|}\widehat{f}\|_{L^2}+\|J^{s_2+|\beta_2|+r_2}\widehat{f}\|_{L^2}\\
\lesssim & \|J^{r_1+r_2}\widehat{f}\|_{L^2} \sim \|\langle x \rangle^{r_1+r_2}f\|_{L^2},
\end{aligned}    
\end{equation*}
we have also used Theorem \ref{stein} at the $L^{\frac{2s_2}{s-2s_2}}$-level to exchange the derivative $\mathcal{D}_{\xi}^{r_2}$ by $D^{r_2}$. This completes the estimate of $I_{2}(t,\beta_1,\beta_2)$ within the assumptions of case (c).
\\ \\
Finally, we collect the estimates of the previous cases $(a)$, $(b)$ and $(c)$ to get
\begin{equation}\label{finaleq1}
\begin{aligned}
\mathcal{I}_{2}(t,\beta_1,\beta_2)\lesssim \langle t \rangle^{r_1}\big(\|\langle x \rangle^{r_1+r_2}f\|_{L^2}+\|J^{a(r_1+r_2)}f\|_{L^2}\big).
\end{aligned}    
\end{equation}
Next, we deal with $\mathcal{I}_1(t,\beta_1,\beta_2)$. When $\beta_1=0$, by similar arguments as above using Lemma \ref{interaniso}, we get
\begin{equation*}
\begin{aligned}
\mathcal{I}_1(t,0,\beta_2)\lesssim \langle t \rangle^{r_1}\|\langle\xi \rangle^{a r_2}\partial^{\beta_2}\widehat{f}\|_{L^2}\lesssim  \langle t \rangle^{r_1}\Big(\|\langle x \rangle^{r_1+r_2}f\|_{L^2}+\|J^{a(r_1+r_2)}f\|_{L^2}\Big).
\end{aligned}    
\end{equation*}
We  will assume $|\beta_1|\geq 1$. If $al+l-|\beta_1|\geq 0$, by using interpolation Lemma \ref{interaniso}, and similar ideas as in \eqref{lineeqe1.2.1}, we have
\begin{equation*}
\begin{aligned}
\mathcal{I}_1(t,\beta_1,\beta_2)|\lesssim & \langle t \rangle^{r_1} \sum_{l=1}^{|\beta_1|} \|\langle \xi \rangle^{ar_2+al+l-|\beta_1|}\partial^{\beta_2}\widehat{f}\|_{L^2}\\
\lesssim & \langle t \rangle^{r_1}\big( \|J^{r_1+r_2}\widehat{f}\|_{L^2}+\sum_{l=1}^{|\beta_1|}\sum_{\substack{\beta_{2,1}+\beta_{2,2}=\beta_2 \\ a r_2+al+l-|\beta_1|-|\beta_{2,2}|>0}} \|J^{|\beta_{2,1}|}(\langle \xi \rangle^{a r_2+al+l-|\beta_1|-|\beta_{2,2}|}\widehat{f})\|_{L^2}\big)\\
\lesssim & \langle t \rangle^{r_1}(\|J^{r_1+r_2}\widehat{f}\|_{L^2}+\|\langle \xi \rangle^{a(r_1+r_2)}\widehat{f}\|_{L^2})\\
\sim & \langle t \rangle^{r_1}(\|\langle x \rangle^{r_1+r_2}f\|_{L^2}+\|J^{a(r_1+r_2)}f\|_{L^2}).
\end{aligned}    
\end{equation*}
If $al+l-|\beta_1|< 0$, we use the function $\phi$ introduced before to divide our arguments as follows
\begin{equation*}
\begin{aligned}
\mathcal{I}_1(t,\beta_1,\beta_2)|\lesssim & \langle t \rangle^{r_1}\sum_{l=1}^{|\beta_1|}( \|\langle \xi \rangle^{a r_2}|\xi|^{al+l-|\beta_1|}\phi\partial^{\beta_2}\widehat{f}\|_{L^2}+\|\langle \xi \rangle^{a r_2}|\xi|^{al+l-|\beta_1|}(1-\phi)\partial^{\beta_2}\widehat{f}\|_{L^2}).
\end{aligned}    
\end{equation*}
Since $al+l-|\beta_1|<0$, the estimate for the second term on the right-hand side of the above inequality follows by similar considerations in \eqref{lineeqe1.2.1}, which mostly depend on Lemma \ref{interaniso}. Hence, we deduce
\begin{equation*}
\begin{aligned}
\|\langle \xi \rangle^{a r_2}|\xi|^{al+l-|\beta_1|}(1-\phi)\partial^{\beta_2}\phi\widehat{f}\|_{L^2}\lesssim & \|\langle \xi \rangle^{a r_2}\partial^{\beta_2}\phi\widehat{f}\|_{L^2}\\
\lesssim & \|J^{r_1+r_2}\widehat{f}\|_{L^2}+\|\langle \xi \rangle^{a(r_1+r_2)}\widehat{f}\|_{L^2}\\
\sim & \|\langle x \rangle^{r_1+r_2}f\|_{L^2}+\|J^{a(r_1+r_2)}f\|_{L^2}.
\end{aligned}    
\end{equation*}
By the argument in \eqref{lineeqe1.3} and familiar estimates, it is readily seen that
\begin{equation*}
\begin{aligned}
\|\langle \xi \rangle^{a r_2}|\xi|^{al+l-|\beta_1|}\phi\partial^{\beta_2}\widehat{f}\|_{L^2}\lesssim & \||\xi|^{al+l-|\beta_1|}\phi\partial^{\beta_2}\widehat{f}\|_{L^2}+\|\langle \xi \rangle^{a r_2}\partial^{\beta_2}\widehat{f}\|_{L^2}\\
\lesssim & \|J^{r_1+r_2}\widehat{f}\|_{L^2}+\|\langle \xi \rangle^{a(r_1+r_2)}\widehat{f}\|_{L^2}\\
\sim & \|\langle x \rangle^{r_1+r_2}f\|_{L^2}+\|J^{a(r_1+r_2)}f\|_{L^2}.
\end{aligned}    
\end{equation*}
Consequently, the previous considerations ultimately imply
\begin{equation}\label{finaleq2}
\begin{aligned}
\mathcal{I}_{1}(t,\beta_1,\beta_2)\lesssim \langle t \rangle^{r_1}\big(\|\langle x \rangle^{r_1+r_2}f\|_{L^2}+\|J^{a(r_1+r_2)}f\|_{L^2}\big).
\end{aligned}    
\end{equation}
Finally, \eqref{finaleq1} and \eqref{finaleq2} yield the desired result \eqref{linereq1}. 

\end{proof}

Next we deduce the second results in Lemma \ref{linearestilemma}.
\begin{remark}
When $a+m+\frac{d}{2}\leq r<a+1+m+\frac{d}{2}$ for some $m\in \mathbb{Z}^{+}$, and $f\in H^{a r}(\mathbb{R}^d)\cap L^2(|x|^{2r}\, dx)$, the integral in the condition \eqref{linearEstimcomp1} is well-defined. This is a consequence of Cauchy-Schwarz inequality as follows
\begin{equation*}
\begin{aligned}
\|x^{\beta}f\|_{L^1}\lesssim & \|\langle x\rangle^{-(\frac{d}{2}^{+})}\|_{L^2}\|\langle x\rangle^{\frac{d}{2}^{+}+m-1}f\|_{L^2}\\
\lesssim &\|\langle x\rangle^{r}f\|_{L^2}.
\end{aligned}    
\end{equation*}
\end{remark}

\begin{proof}[Proof of Lemma \ref{linearestilemma} (ii)]

We assume $a+m+\frac{d}{2}\leq r<a+1+m+\frac{d}{2}$ for some $m\in \mathbb{Z}^{+}$, $f\in H^{a r}(\mathbb{R}^d)\cap L^2(|x|^{2r}\, dx)$, where $f$ satisfies \eqref{linearEstimcomp1}. We let $P_{\phi}$ be the operator defined by the Fourier multiplier by the function $\phi$, i.e., $\widehat{P_{\phi}g}=\phi\widehat{g}$. It follows 
\begin{equation}\label{eqdecay2}
\|\langle x \rangle^r U(t)f\|_{L^2}\lesssim \|\langle x \rangle^r U(t)P_{\phi}f\|_{L^2} +\|\langle x \rangle^r U(t)(I-P_{\phi})f\|_{L^2}.
\end{equation}
As before, we will write $r=r_1+r_2$ with $r_1\in \mathbb{Z}^{+}$, $r_2\in [0,1)$. Arguing as in \eqref{Planchereleq1}, we have $\langle x \rangle^r U(t)P_{\phi}f\in L^{2}(\mathbb{R}^d)$ if and only if 
\begin{equation}\label{eqdecay3}
\begin{aligned}
\partial^{\beta_1}(e^{i\xi _1|\xi|^a t}\phi)\partial^{\beta_2}\widehat{f}\in H^{r_2}(\mathbb{R}^d),
\end{aligned}    
\end{equation}
and $\langle x \rangle^r U(t)(I-P_{\phi})f\in L^{2}(\mathbb{R}^d)$ if and only if 
\begin{equation}\label{eqdecay4}
\begin{aligned}
\partial^{\beta_1}(e^{i\xi _1|\xi|^a t}(1-\phi))\partial^{\beta_2}\widehat{f}\in H^{r_2}(\mathbb{R}^d),
\end{aligned}    
\end{equation}
for all multi-index $|\beta_1|+|\beta_2|\leq r_1$. We will show that \eqref{eqdecay3} and \eqref{eqdecay4} hold true and the right-hand side of \eqref{linereq1} is also valid for each function in these statements.

When $\beta_1=0$, \eqref{eqdecay3} and \eqref{eqdecay4}  follow from Lemma \ref{decaylinearexp}, and properties \eqref{Leibaniso} and \eqref{propereq1} (this same idea was used in \eqref{lineqeq1.0}, and the subsequent cases $\beta_1=0$). Thus, we will assume $|\beta_1|\geq 1$. Let us first show \eqref{eqdecay4}.

By using Leibniz's rule, we write
\begin{equation}\label{eqdecay4.1}
\begin{aligned}
\partial^{\beta_1}(e^{i\xi _1|\xi|^a t}(1-\phi))\partial^{\beta_2}\widehat{f}=&\partial^{\beta_1}(e^{i\xi_1|\xi|^a t})(1-\phi)\partial^{\beta_2}\widehat{f}\\
&+\sum_{\substack{\beta_{1,1}+\beta_{1,2}=\beta_1 \\ |\beta_{1,2}|\geq 1}}c_{\beta_{1,1},\beta_{1,2}} \partial^{\beta_{1,1}}(e^{i\xi_1|\xi|^a t})\partial^{\beta_{1,2}}(1-\phi)\partial^{\beta_2}\widehat{f}.
\end{aligned}    
\end{equation}
Since $\partial^{\beta_{1,2}}(1-\phi)$ is compactly supported outside of the origin whenever $|\beta_{1,2}|\geq 1$, by combining \eqref{Leibaniso}, \eqref{gradaniso}, and \eqref{propereq1}, we deduce
\begin{equation*}
\begin{aligned}
\|D^{r_2}&\big(\frac{\partial^{\beta_{1,1}}(e^{i\xi_1|\xi|^a t})\partial^{\beta_{1,2}}(1-\phi)}{\langle \xi \rangle^{a|\beta_1|}}\langle \xi \rangle^{a|\beta_1|} \partial^{\beta_2}\widehat{f}\big)\|_{L^2}\\
\lesssim & \big(\|\frac{\partial^{\beta_{1,1}}(e^{i\xi_1|\xi|^a t})\partial^{\beta_{1,2}}(1-\phi)}{\langle \xi \rangle^{a|\beta_1|}}\|_{L^{\infty}}+\|\nabla\big(\frac{\partial^{\beta_{1,1}}(e^{i\xi_1|\xi|^a t})\partial^{\beta_{1,2}}(1-\phi)}{\langle \xi \rangle^{a|\beta_1|}}\big)\|_{L^{\infty}}\big)\\
& \times \big(\|D^{r_2}(\langle \xi \rangle^{a|\beta_1|} \partial^{\beta_2}\widehat{f})\|_{L^2}+\|\langle \xi \rangle^{a|\beta_1|} \partial^{\beta_2}\widehat{f}\|_{L^2}\big)\\
\lesssim & \langle t \rangle^{|\beta_1|}\big( \|D^{r_2}(\langle \xi \rangle^{a|\beta_1|} \partial^{\beta_2}\widehat{f})\|_{L^2}+\|\langle \xi \rangle^{a|\beta_1|} \partial^{\beta_2}\widehat{f}\|_{L^2}\big).
\end{aligned}    
\end{equation*}
By distributing the derivative of order $|\beta_2|$, together with properties \eqref{Leibaniso}-\eqref{propereq1}, we deduce
\begin{equation}\label{eqdecay4.2}
\begin{aligned}
   \|D^{r_2}(\langle \xi \rangle^{a|\beta_1|} \partial^{\beta_2}\widehat{f})\|_{L^2}&+\|\langle \xi \rangle^{a|\beta_1|} \partial^{\beta_2}\widehat{f}\|_{L^2}\\
   \lesssim & \|J^{r_2+|\beta_2|}\widehat{f}\|_{L^2}+ \sum_{\substack{\beta_{2,1}+\beta_{2,2}=\beta_2\\ a|\beta_1|-|\beta_{2,2}|>0}}\|J^{r_2+|\beta_{2,1}|}(\langle \xi \rangle^{a|\beta_1|-|\beta_{2,2}|}\widehat{f})\|_{L^2}\\
   \lesssim &\|J^{r}\widehat{f}\|_{L^2}+\|\langle \xi \rangle^{ar}\widehat{f}\|_{L^2}\\
   \sim & \|\langle x\rangle^{r}f\|_{L^2}+\|J^{ar}f\|_{L^2},
    \end{aligned}
\end{equation}
with $r=r_1+r_2$, and where we have also used interpolation Lemma \ref{interaniso}. Thus, going back to \eqref{eqdecay4.1}, we have that \eqref{eqdecay4} holds true if and only if $\partial^{\beta_1}(e^{i\xi_1|\xi|^a t})(1-\phi)\partial^{\beta_2}\widehat{f}\in H^{r_2}(\mathbb{R}^d)$. Now,  to check that $\partial^{\beta_1}(e^{i\xi_1|\xi|^a t})(1-\phi)\partial^{\beta_2}\widehat{f}\in H^{r_2}(\mathbb{R}^d)$, we use \eqref{lineqeq1} to write
\begin{equation*}
\begin{aligned}
\partial^{\beta_1}(e^{i\xi_1|\xi|^a t})(1-\phi)\partial^{\beta_2}\widehat{f}=e^{i\xi_1|\xi|^a t}\frac{|\xi|^{-2|\beta_1|}\sum_{l=1}^{|\beta_{1}|}t^l|\xi|^{al}P_{|\beta_1|+l}(\xi)(1-\phi)}{\langle\xi \rangle^{a|\beta_1|}}\langle\xi \rangle^{a|\beta_1|}\partial^{\beta_2}\widehat{f}.
\end{aligned}    
\end{equation*}
Once again an application of \eqref{Leibaniso}, \eqref{gradaniso}, \eqref{propereq1}, Lemma \ref{decaylinearexp}, and the previous identity yield
\begin{equation*}
\begin{aligned}
\|D^{r_2}\big(\partial^{\beta_1}(e^{i\xi_1|\xi|^a t})&(1-\phi)\partial^{\beta_2}\widehat{f}\big)\|_{L^2}\\
\lesssim & \langle t \rangle^{r_2}\Big(\|\frac{|\xi|^{-2|\beta_1|}\sum_{l=1}^{|\beta_{1}|}t^l|\xi|^{al}P_{|\beta_1|+l}(\xi)(1-\phi)}{\langle\xi \rangle^{a|\beta_1|}}\|_{L^{\infty}}\\
&+\|\nabla\big(\frac{|\xi|^{-2|\beta_1|}\sum_{l=1}^{|\beta_{1}|}t^l|\xi|^{al}P_{|\beta_1|+l}(\xi)(1-\phi)}{\langle\xi \rangle^{a|\beta_1|}}\big)\|_{L^{\infty}}\Big)\\
& \times \big(\|D^{r_2}(\langle \xi \rangle^{a|\beta_1|} \partial^{\beta_2}\widehat{f})\|_{L^2}+\|\langle \xi \rangle^{a|\beta_1|+ar_2} \partial^{\beta_2}\widehat{f}\|_{L^2}\big)\\
\lesssim & \langle t \rangle^{|\beta_1|+r_2}\big( \|D^{r_2}(\langle \xi \rangle^{a|\beta_1|} \partial^{\beta_2}\widehat{f})\|_{L^2}+\|\langle \xi \rangle^{a|\beta_1|+ar_2} \partial^{\beta_2}\widehat{f}\|_{L^2}\big).
\end{aligned}    
\end{equation*}
Hence, by \eqref{eqdecay4.2} and familiar arguments, we get
\begin{equation*}
    \begin{aligned}
\|D^{r_2}(\langle \xi \rangle^{a|\beta_1|} \partial^{\beta_2}\widehat{f})\|_{L^2}+\|\langle \xi \rangle^{a|\beta_1|+ar_2} \partial^{\beta_2}\widehat{f}\|_{L^2}
   \lesssim & \|\langle x\rangle^{r}f\|_{L^2}+\|J^{ar}f\|_{L^2}.
    \end{aligned}
\end{equation*}
Consequently, the above inequality shows the validity of \eqref{eqdecay4}.

Next, we establish \eqref{eqdecay3}. By Leibniz's rule,  \eqref{eqdecay3} is equivalent to 
\begin{equation}\label{eqdecay5}
\begin{aligned}
(\partial^{\beta_1}(e^{i\xi _1|\xi|^a t}))(\partial^{\beta_3}\phi)\partial^{\beta_2}\widehat{f}\in H^{r_2}(\mathbb{R}^d),
\end{aligned}    
\end{equation}
with $|\beta_1|+|\beta_2|+|\beta_3|\leq r_1$. If $\beta_{3}\neq 0$, following the notation in \eqref{lineqeq1}, i.e, $\partial^{\beta_1}(e^{it\xi_1|\xi|^a})=Q_{\beta_1}(a,t,\xi)e^{it\xi_1|\xi|^a}$, we have  $Q_{\beta_1}(a,t,\xi)\partial^{\beta_{3}}\phi\in C^{\infty}_c(\mathbb{R}^d)$. Then, by Lemma \ref{decaylinearexp}, properties \eqref{Leibaniso}, \eqref{gradaniso} and \eqref{propereq1}, we deduce
\begin{equation}\label{eqdecay5.1}
\begin{aligned}
\|D^{r_2}\big(\partial^{\beta_{1}}(e^{it\xi_1|\xi|^a}))(\partial^{\beta_3}\phi)&(\partial^{\beta_{2}}\widehat{f})\big)\|_{L^2}\\
\lesssim & \langle t\rangle^{r_2}\big( \|\langle\xi \rangle^{a r_2}Q_{\beta_1}(a,t,\xi)(\partial^{\beta_{3}}\phi)\|_{L^{\infty}}\\
&+ \|\mathcal{D}^{r_2}\big(Q_{\beta_1}(a,t,\xi)(\partial^{\beta_{3}}\phi)\big)\|_{L^{\infty}}\big)\|\partial^{\beta_{2}}\widehat{f}\|_{H^{r_2}}\\
\lesssim & \langle t\rangle^{r_1+r_2} \|\langle x \rangle^{r}f\|_{L^2}.
\end{aligned}    
\end{equation}

Consequently, we are reduced to deduce \eqref{eqdecay5} with $\beta_3=0$. We will consider two further cases $|\beta_2|\geq m$, and $|\beta_2|< m$. 

\underline{Assume $|\beta_2|\geq m$, and $\beta_3=0$ in \eqref{eqdecay5}}. 
We distribute the derivative of order $\beta_{1}$ to write
\begin{equation*}
\begin{aligned}
(\partial^{\beta_1}(e^{it\xi_1|\xi|^a}))(\partial^{\beta_{2}}\widehat{f})\phi=\sum_{\beta_{1,1}+\beta_{1,2}=\beta_{1}} c_{\beta_{1,1},\beta_{1,2}}\partial^{\beta_{1,1}}\big(e^{it \xi_1|\xi|^a}\partial^{\beta_{2}+\beta_{1,2}}\widehat{f}\big)\phi,
\end{aligned}    
\end{equation*}
for some constants $c_{\beta_{1,1},\beta_{1,2}}$. The fact that $|\beta_{2}|\geq m$ implies
\begin{equation*}
    0\leq r_2+|\beta_{1,1}|\leq r_1+r_2-|\beta_2|-|\beta_{1,2}|<a+1+\frac{d}{2}.
\end{equation*}
Thus Plancherel's identity and the first part of Lemma \ref{linearestilemma} yield
\begin{equation}\label{eqdecay5.2}
\begin{aligned}
\|D^{r_2}\big((\partial^{\beta_1}(e^{it\xi_1|\xi|^a}))&(\partial^{\beta_{2}}\widehat{f})\phi\big)\|_{L^2}\\
\lesssim & \sum_{\beta_{1,1}+\beta_{1,2}=\beta_1}\|\langle x\rangle^{r_2+|\beta_{1,1}|}U(t)(x^{\beta_{2}+\beta_{1,2}}f)\|_{L^2}\\
\lesssim & \sum_{\beta_{1,1}+\beta_{1,2}=\beta_{1}}\|\langle x \rangle^{r_2+|\beta_{1,1}|}(x^{\beta_2+\beta_{1,2}}f)\|_{L^{2}}+\|J^{a(r_2+|\beta_{1,1}|)}\big(x^{\beta_{2}+\beta_{1,2}}f\big)\|_{L^2}.
\end{aligned}
\end{equation}
Writing $x^{\beta_{2}+\beta_{1,2}}=\big(\frac{x^{\beta_{2}+\beta_{1,2}}}{\langle x \rangle^{|\beta_{2}|+|\beta_{1,2}|}}\big) \langle x \rangle^{|\beta_{2}|+|\beta_{1,2}|}$, and decomposing the derivative of order $a(r_2+|\beta_{1,1}|)$ into its integer and fractional parts, we can use properties \eqref{Leibaniso}-\eqref{propereq1} to deduce 
\begin{equation*}
\begin{aligned}
\|J^{a(r_2+|\beta_{1,1}|)}\big(x^{\beta_{2}+\beta_{1,2}}f\big)\|_{L^2}\lesssim & \|J^{a(r_2+|\beta_{1,1}|)}\big(\langle x\rangle^{|\beta_{2}|+|\beta_{1,2}|}f\big)\|_{L^2}\\
\lesssim &\|J^{ar}f\|_{L^2}+\|\langle x \rangle^{r}f\|_{L^2},
\end{aligned}    
\end{equation*}
where in the last line we used interpolation Lemma \ref{interaniso}. This completes the study of case $|\beta_2|\geq m$. 

\underline{Assume $|\beta_2|< m$, and  $\beta_3=0$ in \eqref{eqdecay5}}. We first notice that \eqref{linearEstimcomp1} is equivalent to
\begin{equation}\label{meancond}
\partial^{\beta}\widehat{f}(0)=0,    
\end{equation}
for all $|\beta|\leq m-1$. We use the Taylor's formula and \eqref{meancond} to write
\begin{equation}\label{eqdecay5.3}
\begin{aligned}
\partial^{\beta_2}\widehat{f}(\xi)&=\sum_{|\beta|\leq m-1-|\beta_2|}\frac{1}{\beta!}\partial^{\beta+\beta_2}\widehat{f}(0)\xi^{\beta}+\sum_{|\beta|=m-|\beta_2|}\frac{|\beta|}{\beta!}\Big(\int_0^1 (1-\sigma)^{|\beta|-1}(\partial^{\beta+\beta_2}\widehat{f})(\sigma \xi)\, d\sigma\Big) \xi^{\beta}\\
&=\sum_{|\beta|=m-|\beta_2|}\frac{|\beta|}{\beta!}\Big(\int_0^1 (1-\sigma)^{|\beta|-1}(\partial^{\beta+\beta_2}\widehat{f})(\sigma \xi)\, d\sigma\Big) \xi^{\beta}\\
&=:\sum_{|\beta|=m-|\beta_2|}G_{\beta}(\partial^{\beta+\beta_2}\widehat{f},\xi) \xi^{\beta}.
\end{aligned}    
\end{equation}
Thus, by \eqref{lineqeq1}, \eqref{eqdecay5}, and the cases studied above, it only remains to show that for all $|\beta|=m-|\beta_2|$,
\begin{equation}\label{eqdecay6}
\begin{aligned}
\partial^{\beta_1}(e^{it\xi_1|\xi|^a})&\xi^{\beta}\phi(\xi) G_{\beta}(\partial^{\beta+\beta_2}\widehat{f},\xi) \\
=&|\xi|^{-2|\beta_1|}\sum_{l=1}^{|\beta_1|}t^{l}|\xi|^{al}P_{|\beta_1|+l}(\xi)e^{i\xi_1|\xi|^a t}\xi^{\beta}\phi(\xi)G_{\beta}(\partial^{\beta+\beta_2}\widehat{f},\xi)\\
=:&\sum_{l=1}^{|\beta_1|}Q_{|\beta_1|,l}(a,t,\xi)e^{i\xi_1|\xi|^a t}\xi^{\beta}\phi(\xi)G_{\beta}(\partial^{\beta+\beta_2}\widehat{f},\xi)\in H^{r_2}(\mathbb{R}^d),
\end{aligned}    
\end{equation}
$\xi \neq 0$, where $P_{|\beta_1|+l}(\xi)$ denotes a homogeneous polynomial of order $|\beta_1|+l$, with $P_{|\beta_1|+1}(\xi)\neq 0$. The idea is that the extra weight $\xi^{\beta}$ compensates the lack of regularity of the function $e^{i\xi _1|\xi|^a t}$ at the origin. Setting $1\leq l \leq |\beta_1|$, to prove \eqref{eqdecay6}, we use the properties \eqref{Leibaniso}-\eqref{propereq1} to get 
\begin{equation*}
\begin{aligned}
\|D^{r_2}\big(Q_{|\beta_1|,l}(a,t,\xi)e^{i\xi_1|\xi|^a t}\xi^{\beta}&\phi(\xi)G_{\beta}(\partial^{\beta+\beta_2}\widehat{f},\xi)\big)\|_{L^2}\\
\lesssim & 
\|\mathcal{D}^{r_2}(e^{i\xi_1|\xi|^a t})Q_{|\beta_1|,l}(a,t,\xi)\xi^{\beta}\phi G_{\beta}(\partial^{\beta+\beta_2}\widehat{f},\xi) \|_{L^2}\\
&+
\|D^{r_2}\big(Q_{|\beta_1|,l}(a,t,\xi)\xi^{\beta}\phi\big)\|_{L^2}\|G_{\beta}(\partial^{\beta+\beta_2}\widehat{f},\xi)\big)\|_{L^{\infty}}\\
&+\|Q_{|\beta_1|,l}(a,t,\xi)\xi^{\beta}\phi\mathcal{D}^{r_2}\big(G_{\beta}(\partial^{\beta+\beta_2}\widehat{f},\xi)\big)\|_{L^2}\\
=:& \mathcal{B}_1+\mathcal{B}_2+\mathcal{B}_3.
\end{aligned}    
\end{equation*}
Let us estimate each term $\mathcal{B}_1$, $\mathcal{B}_2$ and $\mathcal{B}_3$. For all multi-index $|\widetilde{\beta}|$, We observe that
\begin{equation}\label{eqdecay7}
    |Q_{|\beta_1|,l}(a,t,\xi)\xi^{\widetilde{\beta}}|\lesssim \langle t \rangle^{|\beta_1|}|\xi|^{al+l+|\widetilde{\beta}|-|\beta_1|}, \quad \xi \neq 0.
\end{equation}
Given that $|\beta|=m-|\beta_2|$, and the restriction on $r$, we have $al+l+|\beta|-|\beta_1|-r_2+\frac{d}{2}>0$, thus the above estimate and Lemma \ref{lemmafracderivpolyno} imply $\mathcal{D}^{r_2}\big(Q_{|\beta_1|,l}(a,t,\xi)\xi^{\beta}\phi(\xi)\big)\in L^2(\mathbb{R}^d)$, i.e., $Q_{|\beta_1|,l}(a,t,\xi)\xi^{\beta}\phi(\xi)\in H^{r_2}(\mathbb{R}^2)$. Moreover, $|\xi|^{ar_2}Q_{|\beta_1|,l}(a,t,\xi)\xi^{\beta}\phi(\xi)\in L^2(\mathbb{R}^d)$. These facts, Lemma \ref{decaylinearexp} ,and Sobolev embedding $H^{\frac{d}{2}^{+}}(\mathbb{R}^d)\hookrightarrow L^{\infty}(\mathbb{R}^d)$ allow us to conclude
\begin{equation*}
\begin{aligned}
\mathcal{B}_1+\mathcal{B}_2\lesssim \langle  t \rangle^{r} \int_0^1 \|\partial^{\beta+\beta_2}\widehat{f}(\xi)\|_{L^{\infty}}\, d\sigma\lesssim  \langle  t \rangle^{r} \|J^{(m+\frac{d}{2})^{+}}\widehat{f}\|_{L^{2}}\lesssim  \langle  t \rangle^{r} \|\langle x\rangle^{r} f\|_{L^{2}}.
\end{aligned}
\end{equation*}
Where we have also used that $a>0$, $r\geq a+m+\frac{d}{2}$, and thus $H^r(\mathbb{R}^d)\hookrightarrow H^{(m+\frac{d}{2})^{+}}(\mathbb{R}^d)$. The estimate of $\mathcal{B}_3$ is similar to that in \eqref{lineeqe1.3}. However, we will present its deduction as we have the integral term $G_{\beta}(\partial^{\beta+\beta_2}\widehat{f},\xi)$. We let $s_1>0$ fixed such that
\begin{equation*}
    \max\{0,\frac{d}{2}-r_2-1,|\beta_1|+|\beta_2|-al-l-m\}<s_1<\min\{r-r_2-m,\frac{d}{2}\}.
\end{equation*}
Such $s_1>0$ exists by our conditions on $r_1$, $m\geq 1$,  $\beta_1$, and $\beta_2$. Recalling that $|\beta|=m-|\beta_2|$, from H\"older's inequality, a change of variables, and Hardy-Littlewood-Sobolev inequality, we infer
\begin{equation*}
\begin{aligned}
\mathcal{B}_3\lesssim & \langle  t \rangle^{|\beta_1|}\||\xi|^{al+l+|\beta|-|\beta_1|}\phi\|_{L^{\frac{d}{s_1}}}\int_0^1\|\mathcal{D}^{r_2}\big((\partial^{\beta+\beta_2}\widehat{f})(\sigma \xi)\big)\|_{L^{\frac{2d}{d-2s_1}}}\, d\sigma \\
\lesssim & \langle  t \rangle^{|\beta_1|}\||\xi|^{al+l+|\beta|-|\beta_1|}\phi\|_{L^{\frac{d}{s_1}}}\Big(\int_0^1\sigma^{r_2-\frac{(d-2s_1)}{2}}\, d\sigma \Big) \|J^{r_2}\partial^{\beta+\beta_2}\widehat{f}\|_{L^{\frac{2d}{d-2s_1}}} \\
\lesssim & \langle  t \rangle^{|\beta_1|} \|J^{r}\widehat{f}\|_{L^{2}}\sim \langle  t \rangle^{|\beta_1|} \|\langle x \rangle^{r}f\|_{L^{2}}. 
\end{aligned}    
\end{equation*}
This concludes the estimate of the last case $|\beta_2|<m$. Consequently, the proof of Lemma \ref{linearestilemma} is complete.

\end{proof}

\subsection{Proof of Lemma \ref{unicontlemma}}

This part concerns the analysis of condition \eqref{linearEstimcomp1} in the propagation of fractional weights. 

\begin{proof}[Proof of Lemma \ref{unicontlemma}]
We let $P_{\phi}$ be the operator defined by the Fourier multiplier by the function $\phi$, i.e., $\widehat{P_{\phi}g}(\xi)=(\phi(\xi) \widehat{g}(\xi))$. We write
\begin{equation*}
\begin{aligned}
U(t)f=U(t)P_{\phi}f+U(t)(I-P_{\phi})f.
\end{aligned}    
\end{equation*}
Since $(I-P_{\phi})f$ satisfies the hypothesis of Lemma \ref{linearestilemma}, in particular, $\mathcal{F}\big(\partial^{\beta}(I-P_{\phi})f\big)(0)=0$ for all $|\beta|\leq m-1$, we infer $U(t)(I-P_{\phi})f\in L^2(|x|^{2(a+m+\frac{d}{2})}\, dx)$. Thus the hypothesis \eqref{uniquecontlemma1} implies
\begin{equation}\label{equnicont1}
\begin{aligned}
U(t)P_{\phi}f\in L^{2}(|x|^{2(a+m+\frac{d}{2})}\, dx).
\end{aligned}    
\end{equation}
Let $a+m+\frac{d}{2}=r_1+r_2$ with $r_1\in \mathbb{Z}^{+}$, $r_2\in [0,1)$, and $|\beta|\leq r$, by taking the Fourier transform, we have that \eqref{equnicont1} is equivalent to
\begin{equation}\label{equnicont2}
\begin{aligned}
\mathcal{F}(x^{\beta} U(t)P_{\phi}f)(\xi)=&\partial^{\beta}(e^{it\xi_1|\xi|^a}\phi(\xi)\widehat{f}(\xi))\\
=&\sum_{\beta_{1}+\beta_{2}+\beta_{3}=\beta}c_{\beta_{1},\beta_{2},\beta_{3}}(\partial^{\beta_{1}}(e^{it\xi_1|\xi|^a}))(\partial^{\beta_{2}}\widehat{f})(\partial^{\beta_{3}}\phi)\in H^{r_2}(\mathbb{R}^d),
\end{aligned}    
\end{equation}
which must follows for all $|\beta|= r_1$. The arguments in the proof of Lemma \ref{linearestilemma} (ii) show that the hypothesis $f\in H^{a(a+m+\frac{d}{2})}(\mathbb{R}^d)\cap L^2 (|x|^{2({a+m+\frac{d}{2}})}\, dx)$ is enough to assure 
\begin{equation*}
(\partial^{\beta_{1}}(e^{it\xi_1|\xi|^a}))(\partial^{\beta_{2}}\widehat{f})(\partial^{\beta_{3}}\phi)\in H^{r_2}(\mathbb{R}^d),    
\end{equation*}
if either $|\beta_3|\neq 0$, or $|\beta_2|\geq m$ (see the ideas around \eqref{eqdecay5.1} for the former condition, and \eqref{eqdecay5.2} for the latter). Thus \eqref{equnicont2} implies  
\begin{equation}\label{equnicont3}
\begin{aligned}
(\partial^{\beta_{1}}(e^{it\xi_1|\xi|^a}))(\partial^{\beta_{2}}\widehat{f})\phi\in H^{r_2}(\mathbb{R}^d),
\end{aligned}    
\end{equation}
for all $|\beta_2|\leq m-1$, where $|\beta_1|+|\beta_2|= r_1$. We emphasize that in the proof of Lemma \ref{linearestilemma} (ii) , we show that \eqref{equnicont3} is valid provided that $\partial^{\beta}\widehat{f}(0)=0$ for all $|\beta|\leq m-1$ (see \eqref{eqdecay5.3}). Thus, here we will show that \eqref{equnicont3} with $t\neq 0$ forces that  $\partial^{\beta}\widehat{f}(0)=0$ for all $|\beta|\leq m-1$. 
\\ \\
In what follows, we will use an inductive argument on the size of $m\geq 1$ to prove that $\partial^{\beta}\widehat{f}(0)=0$ for all $|\beta|\leq m-1$.
\\ \\ 
\underline{Case $m=1$}.  By \eqref{lineqeq1}, given that $|\beta_1|=r_1 $, we have that \eqref{equnicont3} is equivalent to
\begin{equation}\label{equnicont4}
\begin{aligned}
\partial^{\beta_1}(e^{it\xi_1|\xi|^a})\phi(\xi)\widehat{f}(\xi)=&|\xi|^{-2|\beta_1|}\sum_{l=1}^{|\beta_1|}t^{l}|\xi|^{al}P_{|\beta_1|+l}(\xi)e^{i\xi_1|\xi|^a t}\phi(\xi)\widehat{f}(\xi)\\
=:&\sum_{l=1}^{|\beta_1|}Q_{|\beta_1|,l}(a,t,\xi)e^{i\xi_1|\xi|^a t}\phi(\xi)\widehat{f}(\xi)\in H^{r_2}(\mathbb{R}^d),
\end{aligned}    
\end{equation}
for some homogeneous polynomial $P_{|\beta_1|+l}(\xi)$ of order $|\beta_1|+l$, with $P_{|\beta_1|+1}(\xi)\neq 0$. We observe that
\begin{equation*}
    |Q_{|\beta_1|,l}(a,t,\xi)|\lesssim \langle t \rangle^{r_1}|\xi|^{al+l-r_1}, \quad \xi \neq 0.
\end{equation*}
When $l\geq 2$, given that in this case $r_1+r_2=a+1+\frac{d}{2}<al+l+\frac{d}{2}$, the above estimate and Lemma \ref{lemmafracderivpolyno} imply $\mathcal{D}^{r_2}\big(Q_{|\beta_1|,l}(a,t,\xi)\phi(\xi)\big)\in L^2_{loc}(\mathbb{R}^d)$. This in turn allows us to argue exactly as in the cases (a), (b), and (c) in the proof of Lemma \ref{linearestilemma} to deduce
\begin{equation}\label{equnicont5}
\sum_{l=2}^{r_1}Q_{|\beta_1|,l}(a,t,\xi)e^{i\xi_1|\xi|^a t}\phi(\xi)\widehat{f}(\xi)\in H^{r_2}(\mathbb{R}^d).    
\end{equation}
Hence by \eqref{equnicont3} and \eqref{equnicont5}, it must follow
\begin{equation*}
Q_{r_1,1}(a,t,\xi)e^{i\xi_1|\xi|^a t}\phi(\xi)\widehat{f}(\xi)\in H^{r_2}(\mathbb{R}^d).
\end{equation*}
The above statement is equivalent to
\begin{equation*}
\begin{aligned}
Q_{r_1,1}(a,t,\xi)e^{i\xi_1|\xi|^a t}\phi(\xi)\widehat{f}(\xi)=&Q_{r_1,1}(a,t,\xi)e^{i\xi_1|\xi|^a t}\phi(\xi)(\widehat{f}(\xi)-\widehat{f}(0))\\
&+Q_{r_1,1}(a,t,\xi)\big(e^{i\xi_1|\xi|^a t}-1\big)\phi(\xi)\widehat{f}(0)+Q_{r_1,1}(a,t,\xi)\phi(\xi)\widehat{f}(0)\\
=:&\mathcal{C}_1+\mathcal{C}_2+\mathcal{C}_3\in H^{r_2}(\mathbb{R}^d).
\end{aligned}
\end{equation*}
To study $\mathcal{C}_1$, we write
\begin{equation*}
\mathcal{C}_1= Q_{r_1,1}(a,t,\xi)e^{i\xi_1|\xi|^a t}\phi(\xi)\sum_{j=1}^d\int_0^1 \xi_j \partial_{\xi_j}\widehat{f}(\sigma \xi)\, d\xi.    
\end{equation*}
Consequently, the fact that $\mathcal{C}_1\in H^{r_2}(\mathbb{R}^d)$ follows from the same arguments in the study of \eqref{eqdecay5.3} and \eqref{eqdecay6} in the proof of Lemma \ref{linearestilemma} (ii).  On the other hand, since
\begin{equation*}
\begin{aligned}
|Q_{r_1,1}(a,t,\xi)\big(e^{i\xi_1|\xi|^a t}-1\big)|\lesssim |t||\xi|^{2a+2-r_1},
\end{aligned}
\end{equation*}
it is not hard to deduce $\mathcal{C}_2\in H^{r_2}(\mathbb{R}^d)$. Summarizing, we conclude that \eqref{equnicont1} is valid if and only if $\mathcal{C}_3\in H^{r_2}(\mathbb{R}^d)$, in other words, by noticing that $Q_{r_1,1}(a,t,\xi)\phi\in L^2(\mathbb{R}^d)$ if $r_2>0$, it must be the case
\begin{equation}\label{equnicont6}
\begin{aligned}
Q_{r_1,1}(a,t,\xi)\phi \widehat{f}(0)=t |\xi|^{-2r_1}|\xi|^a P_{r_1+1}(\xi)\phi \widehat{f}(0)\in \dot{H}^{r_2}(\mathbb{R}^d),
\end{aligned}    
\end{equation}
where \eqref{equnicont6} holds true for all homogeneous polynomial $P_{r_1+1}(\xi)$ or order $r_1+1$, if and only of \eqref{equnicont1} is valid. We will show that if $P_{r_1+1}(\xi)\neq 0$,
\begin{equation}\label{equnicont7}
|\xi|^{-2r_1}|\xi|^a P_{r_1+1}(\xi)\phi\notin \dot{H}^{r_2}(\mathbb{R}^d). 
\end{equation}
Once we have established \eqref{equnicont7}, given that $t\neq 0$, \eqref{equnicont6} can only be true if $\widehat{f}(0)=0$ as required. Consequently, to complete the present case, it only remains to show \eqref{equnicont7}.  Notice that when $r_2=0$, the fact that $r_1=a+1+\frac{d}{2}$ shows $|\xi|^{-2r_1}|\xi|^a P_{r_1+1}(\xi)\phi\notin L^2(\mathbb{R}^d)$. On the other hand, when $0<r_2<1$, \eqref{equnicont7} is a consequence of the following claim.

\begin{claim}\label{claimL2integr}
Let $l_1$, $0<b<1$, $l_2\geq 1$ be an integer number, and $P_{l_2}(\xi)\neq 0$ be an homogeneous polynomial of order $l_2$. Assume 
\begin{equation}\label{claimL2integ0}
l_2-l_1-b\leq -\frac{d}{2}.    
\end{equation}
When $l_2=l_1$, since $0<b<1$, \eqref{claimL2integ0} is equivalent to $d=1$ and  $\frac{1}{2}\leq b<1$, in which case we will further assume that $l_2$ is an odd number. Then, under either of the previous assumptions it follows
\begin{equation}\label{claimL2integ}
    |\xi|^{-l_1}P_{l_2}(\xi)\phi\notin \dot{H}^{b}(\mathbb{R}^d).
\end{equation}
\end{claim}

\begin{proof}[Proof of Claim \ref{claimL2integr}]
We will show that there exists a set $\widetilde{\mathcal{A}}$ such that
\begin{equation}\label{claimL2integ2}
    \mathcal{D}^{b}\big(|\xi|^{-l_1}P_{l_2}(\xi)\phi\big)(x)\chi_{\widetilde{\mathcal{A}}}\gtrsim |x|^{l_2-l_1-b}\chi_{\widetilde{\mathcal{A}}},
\end{equation}
and $|x|^{l_2-l_1-b}\chi_{\widetilde{\mathcal{A}}}\notin L^2(\mathbb{R}^d)$, where $\chi_{\widetilde{\mathcal{A}}}$ stands for the indicator function over the set $\widetilde{\mathcal{A}}$. Once we have established \eqref{claimL2integ2}, our conditions on $l_1,l_2$ and $b$ imply $\mathcal{D}^{b}\big(|\xi|^{-l_1}P_{l_2}(\xi)\phi\big)(x)\notin L^2(\mathbb{R}^d)$, thus Theorem \ref{stein} and \eqref{propereq1} yield \eqref{claimL2integ}. We divide the deduction of \eqref{claimL2integ2} in the cases $l_2-l_1=0$, and $l_2-l_1\neq 0$. 
\\ \\ 
\underline{Assume $l_2-l_1=0$}. Here $d=1$, $\frac{1}{2}\leq b<1$, with $l_1$ be an odd number. From the fact that $\phi(x)=1$, when $|x|\leq 1$ , and using that $P_{l_2}(\xi)$ is an odd function, we get
\begin{equation*}
\begin{aligned}
   \big(\mathcal{D}^{b}&\big(|\xi|^{-l_1}P_{l_2}(\xi)\phi\big)(x)\big)^2\chi_{\{0<x<1\}}(x)\\
   &\gtrsim \Big( \int_{\frac{x}{4}<y<0}\,\frac{||x|^{-l_1}P_{l_2}(x)\phi(x)-|y|^{-l_1}P_{l_2}(y)\phi(y)|^2}{|x-y|^{1+2b}} dy\Big)\chi_{\{0<x<1\}}(x)\\
   &\gtrsim  |x|^{-1-2b}\Big(\int_{\{\frac{x}{4}<y<0\}}\, dy\Big)\chi_{\{0<x<1\}}(x)\\
    &\gtrsim  |x|^{-2b}\chi_{\{0<x<1\}}(x).
\end{aligned}
\end{equation*}
Given that $\frac{1}{2}\leq b <1$, we have $|x|^{-b}\chi_{\{0<x<1\}}(x)\notin L^{2}(\mathbb{R})$. The proof of this case is complete.
\\ \\ 
\underline{Assume $l_2-l_1\neq 0$}. Let $0<\delta<1$ to be chosen later. When $d\geq 2$, since $P_{l_2}(\xi)\neq 0$, there exist $c_0>0$, and a subset $A_d\subset \mathbb{S}^{d-1}$ (where $\mathbb{S}^{d-1}$ denotes the $d$-dimensional sphere) with positive surface area, such that $c_0\leq |P_{l_2}(y)|\leq 2c_0$ for all $y\in A_d$. We set
\begin{equation}\label{Aset}
    \mathcal{A}_{d,\delta}=\{x\in \mathbb{R}^d: \, \frac{x}{|x|}\in A_d, 0<|x|\leq \delta \}.
\end{equation}
When $d=1$, we set $\mathcal{A}_{1,\delta}=\{x\in \mathbb{R}^{+}: |x|\leq \delta \}$, i.e., $\mathcal{A}_{1,\delta}$ is defined by \eqref{Aset} with $A_1=\{1\}$,  and we can take $c_0=|P_{l_2}(1)|$. Now, we consider two further cases: $l_2-l_1<0$, and $l_2-l_1>0$. 
\\ \\
\underline{Assume $l_2-l_1< 0$}. Let $\delta=\frac{1}{2}$, $x\in \mathcal{A}_{d,\delta}$ be fixed,  we define $\mathcal{B}_{d,l_1,l_2,x}=\{y\in \mathbb{R}^d:\,  \frac{y}{|y|}\in A_d, \, 0<|y|\leq \min\{3^{\frac{-1}{l_1-l_2}}|x|,\delta\}\}$. Thus, for $x\in \mathcal{A}_{d,\delta}$, and $y\in \mathcal{B}_{d,l_1,l_2,x}$, 
\begin{equation*}
\begin{aligned}
||x|^{-l_1}P_{l_2}(x)\phi(x)|=|x|^{-l_1+l_2}|P_{l_2}(\frac{x}{|x|})|\leq 2c_0 |x|^{-l_1+l_2},
\end{aligned}    
\end{equation*}
and
\begin{equation*}
\begin{aligned}
||y|^{-l_1}P_{l_2}(y)\phi(y)|=|y|^{-l_1+l_2}|P_{l_2}(\frac{y}{|y|})|\geq c_0\Big(3^{\frac{-1}{l_1-l_2}}|x|\Big)^{-l_1+l_2}.
\end{aligned}    
\end{equation*}
Hence, we have
\begin{equation*}
\begin{aligned}
||x|^{-l_1}P_{l_2}(x)\phi(x)-|y|^{-l_1}P_{l_2}(y)\phi(y)|\geq c_0|x|^{-l_1+l_2},
\end{aligned}    
\end{equation*}
The previous estimate implies
\begin{equation*}
\begin{aligned}
   \big(\mathcal{D}^{b}\big(|\xi|^{-l_1}P_{l_2}(\xi)\phi\big)(x)\big)^2\chi_{\mathcal{A}_{d,\delta}}(x)\gtrsim &\Big( |x|^{-2l_1+2l_2-d-2b}\int_{\mathcal{B}_{d,l_1,l_2,x}}\, dy\Big)\chi_{\mathcal{A}_{d,\delta}}(x)\\
     \gtrsim & |x|^{-2l_1+2l_2-2b}\chi_{\mathcal{A}_{d,\delta}}(x),
\end{aligned}
\end{equation*}
where we have also used that $|y|\ll |x|$, if $y\in \mathcal{B}_{d,l_1,l_2,x}$. Now, it is not difficult to see, $|x|^{-2l_1+2l_2-2b}\chi_{\mathcal{A}_{d,\delta}}(x)\notin L^{2}(\mathbb{R}^d)$. This in turn completes the case $l_2-l_1<0$.
\\ \\
\underline{Assume $l_2-l_1> 0$}. Let $0<\delta<\min\{1,2^{\frac{1}{-l_1+l_2}}\}$ fixed.  Given $x\in \mathcal{A}_{d,\delta}$, we define  $\mathcal{B}_{d,l_1,l_2,x}=\{y\in \mathbb{R}^d:\,  \frac{y}{|y|}\in A_d, \, 2^{\frac{1}{-l_1+l_2}+1}|x|\leq |y| \leq \delta\}$. 
Thus, given $x\in \mathcal{A}_{d,\delta}$, and $y\in \mathcal{B}_{d,l_1,l_2,x}$, we have
\begin{equation*}
\begin{aligned}
||x|^{-l_1}P_{l_2}(x)\phi(x)|=|x|^{-l_1+l_2}|P_{l_2}(\frac{x}{|x|})|\leq 2c_0\big( 2^{\frac{-1}{-l_1+l_2}-1}|y|\big)^{-l_1+l_2},
\end{aligned}    
\end{equation*}
and
\begin{equation*}
\begin{aligned}
||y|^{-l_1}P_{l_2}(y)\phi(y)|=|y|^{-l_1+l_2}|P_{l_2}(\frac{y}{|y|})|\geq c_0|y|^{-l_1+l_2}.
\end{aligned}    
\end{equation*}
It follows,
\begin{equation*}
\begin{aligned}
||x|^{-l_1}P_{l_2}(x)\phi(x)-|y|^{-l_1+l_2}P_{l_2}(y)\phi(y)|\gtrsim |y|^{-l_1+l_2}.
\end{aligned}    
\end{equation*}
Using that $|x-y|\sim |y|$, whenever $x\in \mathcal{A}_{d,\delta}$, $y\in \mathcal{B}_{d,l_1,l_2,x}$, we deduce
\begin{equation*}
\begin{aligned}
   \big(\mathcal{D}^{b}\big(|\xi|^{-l_1}P_{l_2}(\xi)\phi\big)(x)\big)^2\chi_{\mathcal{A}_{d,\delta}}(x)\gtrsim &\Big(\int_{\mathcal{B}_{d,l_1,l_2,x}} |y|^{-2l_1+2l_2-d-2b}\, dy\Big)\chi_{\mathcal{A}_{d,\delta}}(x)\\
     \gtrsim & |x|^{-2l_1+2l_2-2b}\chi_{\mathcal{A}_{d,\delta}}(x).
\end{aligned}
\end{equation*}
The previous estimate completes the proof of the case $l_2-l_1>0$, and in consequence the proof of Claim \ref{claimL2integr} is complete. 
\end{proof}
Next, we deal with the case $m\geq 2$ in \eqref{equnicont3}. 
\\ \\
\underline{\bf Case $m\geq 2$}. We consider $f\in H^{a(a+m+\frac{d}{2})}(\mathbb{R}^d)\cap L^2 (|x|^{2({a+m+\frac{d}{2}})}\, dx)$, which satisfies \eqref{uniquecontlemma1}. By induction hypothesis, we have
\begin{equation}\label{inducthypho1}
    \partial^{\beta}\widehat{f}(0)=0,
\end{equation}
for all $|\beta|\leq m-2$. To complete the inductive step, we will show that \eqref{inducthypho1} is true when $|\beta|=m-1$. By the previous discussion (see \eqref{equnicont3}), and the inductive hypothesis \eqref{inducthypho1} (i.e., using Lemma \ref{linearestilemma} (ii)), we have that \eqref{uniquecontlemma1} is equivalent to \eqref{equnicont3} with $|\beta_2|=m-1$, i.e., by \eqref{lineqeq1}, 
\begin{equation}
\begin{aligned}
\partial^{\beta_1}(e^{it\xi_1|\xi|^a})\phi(\xi)\partial^{\beta_2}\widehat{f}(\xi)=:&\sum_{l=1}^{|\beta_1|}Q_{|\beta_1|,l}(a,t,\xi)e^{i\xi_1|\xi|^a t}\phi(\xi)\partial^{\beta_2}\widehat{f}(\xi)\in H^{r_2}(\mathbb{R}^d),    
\end{aligned}
\end{equation}
where $a+m+\frac{d}{2}=r_1+r_2$, with $r_1\in \mathbb{Z}^{+}$, $r_2\in [0,1)$, and   $|\beta_1|=r_1-|\beta_2|\leq a+1+\frac{d}{2}$. However, the restrictions on $|\beta_1|, |\beta_2|$ shows that the same arguments  developed for $m=1$ above are applicable in the current case with $\partial^{\beta_2}\widehat{f}$ instead of $\widehat{f}$. Thus, it must be the case  that $\partial^{\beta_2}\widehat{f}(0)=0$, and since $|\beta_2|=m-1$ is arbitrary, \eqref{inducthypho1} follows for any $|\beta|=m-1$. This completes the inductive step and the proof of Lemma \ref{unicontlemma}.
\end{proof}


\section{Well-posedness in weighted spaces}\label{localweighted1}

This part aims to deduce Theorem \ref{LWPweights}. We divide the proof of this theorem into two parts. The first one deals with the deduction of (i) in  Theorem \ref{LWPweights}, which is based on energy estimates. We follow by proving parts (ii) and (iii) as a consequence of (i) and Lemma \ref{linearestilemma}. 

\subsection{Proof of Theorem \ref{LWPweights} (i)}

Here we establish LWP in the space $H^{s}(\mathbb{R}^d)\cap L^{2}(|x|^{2r}\, dx)$, $0<r<1$. This result is the first step towards the study of higher-order weights. The proof presented here is based on weighted energy estimates, which depend on some approximation of the function $\langle x \rangle$ in the spirit of the work in \cite{FonPO}.
\\ \\
Let $N\in \mathbb{Z}^{+}$, we introduce the truncated weights $\tilde{w}_N : \mathbb{R} \rightarrow \mathbb{R}$ satisfying 
 \begin{equation}
 \tilde{w}_{N}(x)=\left\{\begin{aligned} 
 &\langle x \rangle, \text{ if } |x|\leq N, \\
 &2N, \text{ if } |x|\geq 3N
 \end{aligned}\right.
 \end{equation}
in such a way that $\tilde{w}_N(x)$ is smooth and non-decreasing in $|x|$ with $\tilde{w}'_N(x) \leq 1$ for all $x>0$  and there exists a constant $c$ independent of $N$ from which $|\tilde{w}^{(j)}_N(x)|   \leq |\partial_x^{j}(\langle x \rangle)|$, $j=2,3$. We then define the $d$-dimensional weights 
\begin{equation}\label{intro2}
w_N(x)=\tilde{w}_N(|x|), \text{ where } |x|=\sqrt{x_1^2+\dots+x_d^2}.
\end{equation}
Consequently, for fixed $0<r<1$, the definition of the $\omega_N$ yields $|\partial^{\alpha}( w_{N}^{r}) | \lesssim 1$, for all multi-index $1\leq |\alpha|\leq 3$, where the implicit constant is independent of $N$.
\\ \\
Next, we introduce the local existence theory in Sobolev spaces $H^s(\mathbb{R}^d)$ to be used in our arguments.
\begin{lemma}\label{comwellp}
Let $s>\frac{d}{2}+1$, $k\geq 2$ be an integer number, and $\nu\in \{1,-1\}$. Then for any $u_0 \in H^s(\mathbb{R}^d)$, there exist $T=T(\left\|u_0\right\|_{H^s})>0$ and a unique solution $u\in C([0,T]; H^s(\mathbb{R}^d))$ of the IVP \eqref{HBO-ZK}. In addition, the flow-map $u_0 \mapsto u(t)$ is continuous in the $H^s$-norm. Moreover, there exists a function $\rho\in C([0,T];[0,\infty))$ such that
\begin{equation*}
\|u(t)\|_{H^s}\leq \rho(t), \hspace{0.5cm} t\in [0,T].
\end{equation*}
\end{lemma}
The proof of Lemma \ref{comwellp} follows by standard parabolic regularization argument. Broadly speaking, an additional factor $-\mu \Delta u$ is added to the equation, after which the limit $\mu \to 0$ is taken. The argument applying this technique for the IVP \eqref{HBO-ZK} follow by similar considerations in \cite{CunhaPastor2014,Iorio1986}, so we omit its proof. Notice that in Remark (a) in the introduction (see the discussion below \eqref{wellposcond}), we also mention some other alternatives for Lemma \ref{comwellp} where our results hold with $s\leq \frac{d}{2}+1$.
\\ \\
Now we turn to the proof of Theorem \ref{LWPweights}. Given $0<r<1$, $u_0\in H^s(\mathbb{R}^d)\cap L^2(|x|^{2r} \, dx)$, $s>\frac{d}{2}+1$, by Lemma \ref{comwellp} there exists a time $T>0$, and $u\in C([0,T]; H^s(\mathbb{R}^d))$ solution of the IVP \eqref{HBO-ZK} with initial condition $u_0$. Moreover,
\begin{equation}\label{eqadw0}
\sup_{t\in [0,T]}\|u(t)\|_{H^s}\leq c(\|u_0\|_{H^s})<\infty.
\end{equation}
Accordingly, to prove (i) in Theorem \ref{LWPweights} with $0<r<1$, we only have to establish the persistence property and the continuous dependence in the space $L^2(|x|^{2r}, dx)$. In what follows, we will show
\begin{equation}\label{eqpersi1}
u\in L^{\infty}([0,T];L^2(|x|^{2r} \, dx)).
\end{equation}
Once this has been established, the fact that $u\in C([0,T];L^2(|x|^{2r} \, dx))$, and the continuous dependence follow by a simple adaptation of the arguments in \cite[Theorem 1.3]{CunhaPastor2014} (see also \cite[Proposition 1.1]{OscarWHBO}). Moreover, by the continuous dependence of the flow-map data solution provided by Lemma \ref{comwellp}, and taking the limit in our estimates below, we will assume that the solution $u$ of \eqref{HBO-ZK} is sufficiently regular to perform all the arguments developed in this section.

Let $\omega_{N}$ given by \eqref{intro2}. By multiplying the equation in \eqref{HBO-ZK} by $\omega_{N}^{2r} u$, and then integrating on the spatial variable, we arrive at
\begin{equation}\label{awdiffeeq1}
\begin{aligned}
    \frac{1}{2} \frac{d}{dt} \int (\omega_{N}^{2r} u)^2\, dx-\underbrace{\int \big(\omega_{N}^{r} \partial_{x_1}D^{a} u\big)( \omega_{N}^r u) \, dx}_{\mathcal{A}_1}+\underbrace{\nu\int (\omega_{N}^r u^{k-1} \partial_{x_1}u)  (\omega_{N}^r u)\, dx}_{\mathcal{A}_2}=0.
\end{aligned}
\end{equation}
At once we find
\begin{equation}\label{eqadw0.1}
\begin{aligned}
\big|\mathcal{A}_2 \big| \leq \|u\|_{L^{\infty}}^{k-2}\|\partial_{x_1}u\|_{L^{\infty}}\| \omega_{N}^ru\|_{L^2}^2.
\end{aligned}
\end{equation}
Next, we divide the estimate of $\mathcal{A}_1$ in two cases determined by the values of the dispersion $a$, namely, $1\leq a<2$, and $0<a<1$. But before that, we will set some notation. We recall that the Riesz transform is denoted by $\mathcal{R}_j=-\partial_{x_j}D^{-1}$. Since the Hilbert and the Riesz transforms share the same $L^2$-multiplier restricted to one-dimension, for simplicity, we will also denote by $\mathcal{R}_1$ the Hilbert transform when $d=1$.
\\ \\
\underline{Assume $1\leq a <2$}.  When $a=1$, we follow the convection $D^{a-1}=I$ to be the identity operator. By writing $\partial_{x_1}=-\mathcal{R}_1D$, and $D^2=-\Delta$, we divide our considerations as follows
\begin{equation}\label{eqadw1}
\begin{aligned}
\omega_{N}^{r} \partial_{x_1}D^{a} u=&\omega_{N}^{r}\mathcal{R}_1\Delta D^{a-1}u\\
=&[\omega_{N}^{r},\mathcal{R}_1]\Delta D^{a-1}u+\mathcal{R}_1\big([\omega_{N}^{r},D^{a-1}]\Delta u\big)+\mathcal{R}_1 D^{a-1}([\omega_{N}^{r}, \Delta]u)\\
&+\mathcal{R}_1 D^{a-1}\Delta(\omega_{N}^{r} u)\\
=:& \mathcal{A}_{1,1}+\mathcal{A}_{1,2}+\mathcal{A}_{1,3}+\mathcal{R}_1 D^{a-1}\Delta(\omega_{\lambda} u),
\end{aligned}
\end{equation}
and when $a=1$, the above identity reduces to
\begin{equation}\label{eqadw1.1}
\begin{aligned}
\omega_{N}^{r} \partial_{x_1}Du=&[\omega_{N}^{r},\mathcal{R}_1]\Delta u+\mathcal{R}_1 ([\omega_{N}^{r}, \Delta]u)+\mathcal{R}_1 \Delta(\omega_{N}^{r} u)\\
=& \mathcal{A}_{1,1}+\mathcal{A}_{1,3}+\mathcal{R}_1 D^{a-1}\Delta(\omega_{\lambda} u).
\end{aligned}
\end{equation}
By going back to the integral defining $\mathcal{A}_{1}$, and rewriting $\mathcal{R}_1D^{a-1}\Delta=\partial_{x_1}D^a$ (which is a skew-symmetric operator), we deduce that the last terms on the right-hand side of \eqref{eqadw1} and \eqref{eqadw1.1} yield null contribution to the estimate. Now, when $d=1$, the estimate of the $L^2$-norm of $\mathcal{A}_{1,1}$ is a consequence of Proposition \ref{CalderonComGU}, and  the fact that $\partial_x^2(\omega_{N}^r)$ is uniformly bounded by constant independent of $N\geq 1$. Thus, we will assume $d\geq 2$ in  $\mathcal{A}_{1,1}$. We write
\begin{equation*}
\begin{aligned}
\mathcal{A}_{1,1}=&\sum_{j=1}^d\Big( [\omega_{N}^{r},\mathcal{R}_1]\partial_{x_j}^2D^{a-1}u+\sum_{|\beta|=1}\frac{1}{\beta!}\partial^{\beta}(\omega_{N}^{r}) D_{R_1}^{\beta}\partial_{x_j}^2D^{a-1}u \\
&-\sum_{|\beta|=1}\frac{1}{\beta!}\partial^{\beta}(\omega_{N}^{r}) D_{R_1}^{\beta}\partial_{x_j}^2D^{a-1}u \Big).
\end{aligned}
\end{equation*}
An application of Proposition \ref{propconmu} and Lemma \ref{lemmaRieszdeco} allows us to conclude
\begin{equation}\label{eqadw2}
\begin{aligned}
\|\mathcal{A}_{1,1}\|_{L^2}\lesssim & \sum_{|\alpha|=2} \|\partial^{\alpha}(\omega_{N}^{r})\|_{L^{\infty}}\|D^{a-1}u\|_{L^2}+\sum_{j=1}^d \sum_{|\beta|=1} \|\partial^{\beta}(\omega_{N}^{r})\|_{L^{\infty}}\|D_{R_1}^{\beta}\partial_{x_j}^2 D^{a-1} u\|_{L^2} \\
\lesssim & \sum_{1\leq |\alpha| \leq 2} \|\partial^{\alpha}(\omega_{N}^{r})\|_{L^{\infty}}\|u\|_{H^a}.
\end{aligned}
\end{equation}
Since $0<r<1$, the definition of $\omega_N$ implies that all the derivatives $1\leq |\beta|\leq 3$ of $\partial^{\alpha}(\omega_{N}^{r})$ are bounded by a constant independent of $N$.  Next, when $1<a<2$, since $\mathcal{R}_1$ determines a bounded operator in $L^2(\mathbb{R}^d)$, by Proposition \ref{propcomm1}, we have
\begin{equation}\label{eqadw3}
\begin{aligned}
\|\mathcal{A}_{1,2}\|_{L^2} \lesssim \|[\omega_{N}^r,D^{a-1}]D^2u\|_{L^2}&=\|[\omega_{N}^r,D^{a-1}]D^{2-a}D^a u\|_{L^2}\\
&\lesssim \|\nabla (\omega_{N}^r)\|_{L^{\infty}}\|u\|_{H^{a}}.
\end{aligned}
\end{equation}
On the other hand,
\begin{equation*}
\begin{aligned}
\mathcal{A}_{1,3}=&-\mathcal{R}_1 D^{a-1}\big((\Delta (\omega_{N}^{r}) u+2(\nabla (\omega_{N}^{r})) \cdot \nabla u\big)\\
=&-\mathcal{R}_1 D^{a-1}\big((\Delta (\omega_{N}^{r})\big) u)-2\sum_{j=1}^d\mathcal{R}_1 D^{a-1}\big((\partial_{x_j}(\omega_{N}^{r})) \cdot \partial_{x_j} u\big).
\end{aligned}
\end{equation*}
Now, when $a=1$, the previous decomposition yields
\begin{equation*}
\begin{aligned}
\|\mathcal{A}_{1,3}\|_{L^2}\lesssim \|\Delta(\omega_N^r)\|_{L^{\infty}}\|u\|_{L^2}+\sum_{j=1}^{d}\|\partial_{x_j}(\omega_N^r)\|_{L^{\infty}}\|\partial_{x_j}u\|_{L^2}\lesssim \|u\|_{H^1},
\end{aligned}    
\end{equation*}
where the implicit constant is independent of $N\geq 1$. To proceed with the estimate of $\mathcal{A}_{1,3}$ when $1<a<2$, we first require the following result.
\begin{claim}\label{claim1}
Let $0<r<1$, and $r<s< 3$. It follows
\begin{equation}\label{eqdaw3.1}
\|D^{s}(\omega_N^r)\|_{L^{\infty}}\lesssim 1,
\end{equation}
\end{claim}
where the implicit constant is independent of $N\geq 1$. 

\begin{proof}[Proof of Claim \ref{claim1}]
The proof follows the arguments in the one-dimensional setting presented in \cite[Proposition 2.14]{Riano2021}. Let us consider first the case $0<s<1$. We let $0<\delta<\min\{\frac{s-r}{2},1-s\}$, and $q>1$ large enough such that $\frac{1}{q}<\frac{\delta}{d}$, thus Sobolev embedding yields 
\begin{equation*}
\begin{aligned}
\begin{aligned}
\|D^{s}(\omega^{r}_N)\|_{L^{\infty}}\lesssim & \|D^{s}(\omega^{r}_N)\|_{L^{q}}+\|D^{s+\delta}(\omega^{r}_N)\|_{L^{q}}\\
\lesssim & \|D^{-(1-s)}D(\omega^{r}_N)\|_{L^{q}}+\|D^{-(1-(s+\delta))}D(\omega^{r}_N)\|_{L^{q}}.
\end{aligned}
\end{aligned}    
\end{equation*}
Let $l=0,1$, since $D=\sum_{j=1}^d\mathcal{R}_j\partial_{x_j}$, the continuity of the Riesz transform and Hardy-Littlewood-Sobolev inequality show
\begin{equation*}
\begin{aligned}
\|D^{-(1-(s+l\delta))}D(\omega^{r}_N)\|_{L^{q}}\lesssim & \sum_{j=1}^d \|D^{-(1-(s+l\delta))}\partial_{x_j}(\omega^{r}_N)\|_{L^{q}}\\
\lesssim & \sum_{j=1}^d \|\frac{1}{|x|^{d-(1-(s+l\delta))}}\ast \big(\partial_{x_j}(\omega^{r}_N)\big)\|_{L^{q}}\\
\lesssim &\sum_{j=1}^d\|\partial_{x_j}(\omega^{r}_N)\|_{L^{p_l}},
\end{aligned}    
\end{equation*}
where
\begin{equation*}
\frac{1}{p_l}=\frac{1}{q}+\frac{(1-(s+l\delta))}{d}.    
\end{equation*}
Hence, our choice of $\delta>0$ and $q$ imply $0<\frac{1}{p_l}<\frac{1-r}{d}$, thus $\|\partial_{x_j}(\omega^{r}_N)\|_{L^{p_l}}\lesssim 1$ with constant independent of $N$. 

On the other hand, when $1\leq s< 2$ by writing $s=1+\theta$ with $\theta\in [0,1)$, we have when $\theta>0$ that
\begin{equation*}
\begin{aligned}
D^{s}(\omega_N^r)=\sum_{j=1}\mathcal{R}_jD^{\theta}(\partial_{x_j}(\omega_N^r))=\sum_{j,m=1}\mathcal{R}_j\mathcal{R}_m D^{\theta-1}(\partial_{x_m}\partial_{x_j}(\omega_N^r)).
\end{aligned}    
\end{equation*}
If $\theta=0$, we just write $D^{s}(\omega_N^r)=\sum_{j=1}\mathcal{R}_j(\partial_{x_j}(\omega_N^r))$. Thus, the above decomposition allows us to argue exactly as in the case $0<s<1$ to obtain the desired result. Notice that when $s=2$, we have $D^2=-\Delta$, which directly implies \eqref{eqdaw3.1}. When $2<s<3$, writing $D^{s}(\omega_N^r)=-D^{s-2}(\Delta \omega_N^r)$, again the consideration for the first case $0<s<1$ yield the desired result.
\end{proof}

We continue with the study of $\mathcal{A}_{1,3}$ when $1<a<2$. By Claim \ref{claim1}, Proposition \ref{commutatorestim1} , it is seen that
\begin{equation}\label{eqadw4}
\begin{aligned}
\|\mathcal{A}_{1,3}\|_{L^2}\lesssim &\|[D^{a-1},\Delta(\omega_N^r)]u\|_{L^2}+\|\Delta(\omega_N^r)D^{a-1}u\|_{L^2}\\
 &+ \sum_{j=1}^d\big(\|[D^{a-1},\partial_{x_j}(\omega_N^r)]\partial_{x_j}u\|_{L^2}+\|\partial_{x_j}(\omega_N^r)\partial_{x_j}D^{a-1}u\|_{L^2}\big)\\
\lesssim &\big(\|D^{a-1}\big(\Delta(\omega_N^r)\big)\|_{L^{\infty}}+\sum_{j=1}^d \|D^{a-1}\big(\partial_{x_j}(\omega_N^r)\big)\|_{L^{\infty}}\big)\|u\|_{H^{a}}\\
\lesssim & \|u\|_{H^a}.
\end{aligned}
\end{equation}
This completes the analysis of the case $1\leq a<2$. 
\\ \\
\underline{Assume $0< a <1$}. Once again, since $\partial_{x_1}=-\mathcal{R}_1D$, we write
\begin{equation*}
\begin{aligned}
\omega_N^r\partial_{x_1}D^{a}u=&-\omega_N^r\mathcal{R}_1D^{1+a}u\\
=&-[\omega_N^r,\mathcal{R}_1]D^{1+a}u-\mathcal{R}_1([\omega_N^r,D^{1+a}]u)-\mathcal{R}_1D^{1+a}(\omega_N^r u)\\
=&\widetilde{\mathcal{A}}_{1,1}+\widetilde{\mathcal{A}}_{1,2}-\mathcal{R}_1D^{1+a}(\omega_N^r u).
\end{aligned}    
\end{equation*}
Going back to the integral in $\mathcal{A}_1$, we have that the last term of the above identity gives a null contribution to the estimate. Now, since $D=\sum_{j=1}^d \mathcal{R}_j \partial_{x_j}$, we use Proposition \ref{CalderonComGU} when $d=1$, and Proposition \ref{propconmu} when $d\geq 2$ to deduce  
\begin{equation*}
\begin{aligned}
\|\widetilde{\mathcal{A}}_{1,1}\|_{L^2}\lesssim & \sum_{j=1}^d\|[\omega_N^r,\mathcal{R}_1]\partial_{x_j}\mathcal{R}_jD^{a}u\|_{L^2}\\
\lesssim & \|\nabla(\omega_N^r)\|_{L^{\infty}}\|\mathcal{R}_jD^{a}u\|_{L^2}\lesssim  \|u\|_{H^a}.
\end{aligned}    
\end{equation*}
To estimate $\widetilde{\mathcal{A}}_{1,2}$, we write
\begin{equation*}
\begin{aligned}
\left[D^{1+a},\omega_N^r\right]u=:G(\omega_N^{r},u,a)+\sum_{|\beta|=1}\partial^{\beta}(\omega_N^r)D^{1+a,\beta}u,
\end{aligned}    
\end{equation*}
where by Proposition \ref{fractionalDeriv} and Claim \ref{claim1},
\begin{equation}\label{eqadw4.1}
\begin{aligned}
\|G(\omega_N^{r},u,a)\|_{L^2}\lesssim & \|D^{1+a}(\omega_N^r)\|_{L^{\infty}}\|u\|_{L^{2}}\lesssim \|u\|_{L^2}. 
\end{aligned}    
\end{equation}

By using Plancherel's identity and the definition of the operator $D^{1+a,\beta}$ (see Proposition \ref{fractionalDeriv}), it is not difficult to see
\begin{equation}\label{eqadw4.2}
\begin{aligned}
\sum_{|\beta|=1}\|\partial^{\beta}(\omega_N^r)D^{1+a,\beta}u\|_{L^2}\lesssim & \sum_{|\beta|=1}\|\partial^{\beta}(\omega_N^r)\|_{L^{\infty}}\|D^{1+a,\beta}u\|_{L^2}\lesssim \|u\|_{H^{a}}.
\end{aligned}    
\end{equation}
Thus, \eqref{eqadw4.1} and \eqref{eqadw4.2} complete the estimate of $\widetilde{\mathcal{A}}_{1,2}$, and in turn the present case $0<a<1$. Collecting the estimates deduced above, it follows
\begin{equation}\label{desperestim1}
|\mathcal{A}_1|\lesssim \|u\|_{H^a}\|\omega_N^r u\|_{L^2}.    
\end{equation}
Consequently, by \eqref{eqadw0.1} and \eqref{desperestim1} there exists $c_0>0$ such that
\begin{equation}\label{eqadw5}
\begin{aligned}
\frac{d}{dt}\|\omega_{N}^r u\|_{L^2}^2 \leq c_0\|\omega_{N}^r u\|_{L^2}+2\|u\|_{L^{\infty}}^{k-2}\|\partial_{x_1}u\|_{L^{\infty}}\|\omega_{N}^r u\|_{L^2}^2,
\end{aligned}
\end{equation}
where, by \eqref{eqadw0} and Lemma \ref{comwellp}, we have that the above constant $c_0$ depends on $\|u_0\|_{H^s}$ and it is independent of $N$. Now, Sobolev embedding and \eqref{eqadw0} yield $\|\partial_{x_1}u(t)\|_{L^{\infty}} \lesssim \sup_{t\in [0,T]}\|u(t)\|_{H^s}\leq c(\|u_0\|_{H^s})$. Thus, this remark and Gronwall's inequality applied to \eqref{eqadw5} imply that there exist two positive constants $c_1,c_2$ depending on $\|u_0\|_{H^s}$ and independent of $N$ such that
\begin{equation}
    \begin{aligned}
         \left\|\omega_{N}^r u(t)\right\|_{L^2}&\leq (\left\|\omega_{N}^r u(0)\right\|_{L^2}+c_1 t)e^{c_2 t} \\
&\leq (\left\|\langle x \rangle^r u(0)\right\|_{L^2}+c_1 T)e^{c_2 T}, \hspace{0.5cm} 0\leq t \leq T.
    \end{aligned}
\end{equation}
Taking the limit $N \to 0$ in the above inequality establishes \eqref{eqpersi1}. This in turn completes the proof of Theorem \ref{LWPweights}, whenever $0<r<1$.

\begin{remark}\label{RemarkNonlEstim}
The results in \eqref{eqadw0.1},  \eqref{desperestim1} and \eqref{eqadw5} show that Theorem \ref{LWPweights} (i) can be extended to regularity $s\geq a$ provided that there exist a solution $u$ of \eqref{HBO-ZK} in $C([0,T];H^{a}(\mathbb{R}^d))$ such that
\begin{equation*}
\|u\|_{L^{\infty}}^{k-2}\|\partial_{x_1} u\|_{L^{\infty}}\in L^{1}((0,T)).    
\end{equation*}
For example, H\"older's inequality on the time variable shows that the above condition is granted if
\begin{equation*}
u\in L^{k-1}((0,T);W^{1,\infty}(\mathbb{R}^d)).    
\end{equation*}
\end{remark}


\subsection{Proof of Theorem \ref{LWPweights} (ii) and (iii)}

Before we prove Theorem \ref{LWPweights}, we first present the following lemma in which we relate the decay of a function $f$ with that of the derivative of its power function $\nabla(f^k)$ for some integer $k\geq 2$.

\begin{lemma}\label{lemmadecaynonlinear}
Let $k\geq 2$ be an integer, $r>0$. Let $s>\max\{1,\frac{d}{2}-\frac{d}{2k}\}$, and 
\begin{equation}\label{extradecay}
r_1:=\frac{(s-1)(2ks-d(k-1))}{2s^2}r.   
\end{equation}
Assume $f\in H^{s}(\mathbb{R}^d)\cap L^{\infty}(\mathbb{R}^d)\cap L^{2}(|x|^{2r}\, dx)$, then $\nabla(f^k)\in L^2(|x|^{2r_1}\, dx)$.
\end{lemma}

\begin{proof}
We use interpolation Lemma \ref{interaniso} to deduce
\begin{equation*}
\begin{aligned}
\|\langle x \rangle^{r_1}\nabla(f^k)\|_{L^2}\lesssim & \|\langle x \rangle^{r_1-1}f^k\|_{L^2}+\|J(\langle x \rangle^{r_1}f^k)\|_{L^2}\\
\lesssim & \|J(\langle x \rangle^{r_1}f^k)\|_{L^2}\\
\lesssim & \|\langle x \rangle^{\frac{s r_1}{s-1}}f^k\|_{L^2}+\|J^{s}(f^k)\|_{L^2}.
\end{aligned}    
\end{equation*}
To estimate $J^s(f^k)$, given that $\|J^s(f^k)\|_{L^2}\sim (\|f^k\|_{L^2}+\|D^s(f^k)\|_{L^2})$, and that $k\geq 2$ is integer, by multiple applications of \eqref{fLR}, we deduce
\begin{equation*}
\begin{aligned}
\|J^s(f^k)\|_{L^2}\lesssim \|f\|_{L^{\infty}}^{k-1}\|J^sf \|_{L^2}.
\end{aligned}    
\end{equation*}
On the other hand, we use interpolation, the definition of $r_1$, and Hardy-Littlewood-Sobolev inequality to deduce
\begin{equation*}
\begin{aligned}
\|\langle x \rangle^{\frac{sr_1}{s-1}}f^{k}\|_{L^2}= \|\langle x \rangle^{\frac{s r_1}{k(s-1)}}f\|_{L^{2k}}^{k}\lesssim & \|J^{\frac{d(k-1)}{2k}}(\langle x \rangle^{\frac{s r_1}{k(s-1)}}f)\|_{L^{2}}^k \\
\lesssim & \|\langle x \rangle^{r}f\|_{L^{2}}^{(1-\theta)k}\|J^{s}f\|_{L^{2}}^{\theta k},
\end{aligned}    
\end{equation*}
where $\theta=\frac{d(k-1)}{2ks}$. This completes the proof of the lemma.

\end{proof}

Under the hypothesis of Lemma \ref{lemmadecaynonlinear}, we have that $r_1>r$ if either $s>s_{d,k,1}$ or $0<s<s_{d,k,2}$, where $s_{d,k,1}$ and $s_{d,k,2}$ are defined by \eqref{defis1} and \eqref{defis2}, respectively. This remark and the proof of Lemma \ref{lemmadecaynonlinear} allow us to deduce the following space-time estimate.

\begin{corollary}\label{corollarynonlin}
Let $T>0$, $r>0$, and $k\geq 2$ integer. Let $s>\max\{1,\frac{d}{2}-\frac{d}{2k}\}$ with $s\in (0,s_{d,k,2})\cup (s_{d,k,1},\infty)$. Consider
\begin{equation*}
g\in L^{\infty}([0,T];H^{s}(\mathbb{R}^d))\cap L^{k-1}((0,T);L^{\infty}(\mathbb{R}^d))\cap L^{\infty}((0,T);L^2(|x|^{2r}\, dx)).
\end{equation*}
Then
\begin{equation*}
    \nabla (g^k) \in L^{1}((0,T);L^2(|x|^{2r_1}\, dx)),
\end{equation*}
where $r_1>r$ is defined by \eqref{extradecay}.
\end{corollary}

We are in the condition to deduce the proof of Theorem \ref{LWPweights} (ii) and (iii).

\begin{proof}[Proof of Theorem \ref{LWPweights} (ii) and (iii)]

Given $1\leq r<\frac{d}{2}+1+a$, $s\geq \{(\frac{d}{2}+1)^{+},ar+1\}$, with $s\in (0,s_{d,k,2})\cup (s_{d,k,1},\infty)$, we consider $u_0\in H^s(\mathbb{R}^d)\cap L^2(|x|^{2r}\, dx)$. Let  $u\in C([0,T];H^{s}(\mathbb{R}^d))$ be the solution of \eqref{HBO-ZK} with initial condition $u_0$ provided by Lemma \ref{comwellp}. Following the same strategy stated in the proof of (i) in Theorem \ref{LWPweights}, we will only show that \eqref{eqpersi1} holds true. Since we have already established some lower decay cases, we know
\begin{equation}\label{previouscasecond}
    u\in C([0,T];L^2(|x|^{2r_1}\, dx)),
\end{equation}
for all $0<r_1<1$. We first notice that by Lemma \ref{linearestilemma} and the hypothesis on the function $u_0$, it follows
\begin{equation}\label{integralcond0}
U(t)u_0\in L^{\infty}([0,T];L^2(|x|^{2\widetilde{r}}\, dx)).    
\end{equation}
for all $0<\widetilde{r}\leq r$. Thus, since $u$ solves the integral equation 
\begin{equation}\label{integraleq}
    u=U(t)u_0-\nu\int_0^t U(t-\tau)(u^{k-1}\partial_{x_1} u)(s)\, ds,
\end{equation}
to complete the Proof of Theorem \ref{LWPweights}, it is enough to show
\begin{equation}\label{integralcond1}
\int_0^t U(t-\tau)(u^{k-1}\partial_{x_1} u)(\tau)\, d\tau\in  L^{\infty}([0,T];L^2(|x|^{2\widetilde{r}}\, dx)).    
\end{equation}
By Lemma \ref{linearestilemma},  \eqref{integralcond1} is valid provided that
\begin{equation}\label{integralcond1.1}
    \partial_{x_1}(u^k)\in L^{1}((0,T);H^{a\widetilde{r}}(\mathbb{R}^d))\cap L^{1}((0,T);L^{2}(|x|^{2\widetilde{r}}\, dx)).
\end{equation}
In what follows, we will perform some iterations on the size of the weight $\widetilde{r}$ to show that  \eqref{integralcond1.1} holds true. 
\\ \\
Given that $u\in C([0,T];H^s(\mathbb{R}^d))$, $s\geq \max\{\big(\frac{d}{2}+1\big)^{+},ar+1\}$, by Sobolev embedding,  we have $u,\nabla u\in L^{k-1}((0,T);L^{\infty}(\mathbb{R}^d))$. Let $0<r_1<1$ fixed, by \eqref{previouscasecond} and Corollary \ref{corollarynonlin}, it follows
\begin{equation}\label{integralcond2}
\begin{aligned}
\partial_{x_1}(u^k)\in L^{1}((0,T);L^{2}(|x|^{2r_2}\, dx)),
\end{aligned}    
\end{equation}
where $r_2= \theta_{s,k}r_1$ with $\theta_{s,k}:=\frac{(s-1)(2ks-d(k-1))}{2s^2}$. Given that $\theta_{s,k}>1$, \eqref{integralcond2} implies that \eqref{integralcond1} follows with $\widetilde{r}=r_2$, then \eqref{integralcond0} yields
\begin{equation}
u\in L^{\infty}([0,T];L^2(|x|^{2r_2}\, dx)).
\end{equation}
Consequently, setting $r_{j}=\theta_{s,k}r_{j-1}$, $j\geq 2$, we can iterate the argument just describe a finite number of times, up to the least integer $j$ where $r_j\geq r$, if $r_j\geq a+1+\frac{d}{2}$, we let $r_j=r$. The idea is that in each step $l=1,\dots,j$, we show the validity of \eqref{integralcond1.1} with $\widetilde{r}=r_l$, thus by \eqref{integralcond1}, \eqref{integralcond0}, and the integral equation, we have that \eqref{eqpersi1} holds true. This in turn completes the proof of part (ii) of Theorem \ref{LWPweights}.
\\ \\ 
Assuming the hypothesis of Theorem \ref{LWPweights} (iii), the fact that $\widehat{u}_0(0)=0$, and Lemma \ref{linearestilemma} show that \eqref{integralcond0} follows for $0<\widetilde{r}\leq r<a+2+\frac{d}{2}$, also notice that $\mathcal{F}(u^{k-1}\partial_{x_1}u)(0)=0$, then we can use the same iterative argument described above to conclude the proof of Theorem \ref{LWPweights} (iii). 

\end{proof}


\section{Unique continuation principles and further results} \label{uniquep}

This part concerns the deduction of the unique continuation principles stated in Theorems \ref{theoremtwotimes} and \ref{theoremthreetimes}. We also study the sharpness of the three times condition in Corollary \ref{sharpthree}. Finally, we deduce Corollary \ref{oddcasetheorem} in which we get some decay properties of \eqref{HBO-ZK} when $k\geq 2$ is odd. We conclude by showing our results for the Cauchy problem \eqref{CombHBO-ZK}, that is, we deduce Corollary \ref{corollaryCombinedfKdV}.

\begin{proof}[Proof of Theorem \ref{theoremtwotimes}]
Without loss of generality, we will assume that $t_1=0$, $t_2\neq 0$. Since $u\in C([0,T];H^s(\mathbb{R}^d)\cap L^2(|x|^{2(a+1+\frac{d}{2})^{-}}\, dx))$, Corollary \ref{corollarynonlin}, the fact that $\widehat{u^{k-1}\partial_{x_1} u}(0)=0$, and Lemma \ref{linearestilemma} imply
\begin{equation}\label{unicontwoeq1}
    \int_0^t U(t-\tau)(u^{k-1}\partial_{x_1}u)(\tau)\, d\tau\in L^{\infty}([0,T];L^2(|x|^{2(a+1+\frac{d}{2})}\, dx)).
\end{equation}
Given that $u_0\in L^2(|x|^{2(a+1+\frac{d}{2})}\, dx)$, using the integral equation \eqref{integraleq}, the assumption at time $t_2$ together with \eqref{unicontwoeq1}, we obtain
\begin{equation*}
\begin{aligned}
U(t_2)u_0\in L^2(|x|^{2(a+1+\frac{d}{2})}\, dx).
\end{aligned}    
\end{equation*}
Hence, Lemma \ref{unicontlemma} implies that $\widehat{u}_0(0)=0$. The proof is complete. 
\end{proof}

\begin{proof}[Proof of Theorem \ref{theoremthreetimes}]

We first deduce part (i). Let $k\geq 2$ be an even number, and $t_1<t_2<t_3$. The equation in \eqref{HBO-ZK} yields the following identity
\begin{equation}\label{unicontident1}
\frac{d}{dt}\int_{\mathbb{R}^d} x_j u(x,t)\, dx=\frac{\delta_{j,1} \nu}{k}\int_{\mathbb{R}^d} (u(x,t))^k\, dx,  
\end{equation}
$j=1,\dots, d$, where $\delta_{j,1}$ denotes Kronecker delta function.  We emphasize that the above identity is valid for all $0<a<2$ provided that $u\in C([0,T];H^s(\mathbb{R}^d))\cap L^{2}(|x|^{2(a+1+\frac{d}{2})^{+}})$, with $\widehat{u}(0,t)=0$ for all $t\in [0,T]$. We will show that there exists $\widetilde{t}_1\in (t_1,t_2)$ and $\widetilde{t}_2\in (t_2,t_3)$ such that 
\begin{equation}\label{eqthreetimes1}
    \int_{\mathbb{R}^d}x_1 u(x,\widetilde{t}_l)\, dx=0, \qquad l=1,2. 
\end{equation}
Once we have established \eqref{eqthreetimes1}, by integrating \eqref{unicontident1} with $j=1$ over $[\widetilde{t}_1,\widetilde{t}_2]$, we deduce
\begin{equation*}
\begin{aligned}
0=\int_{\mathbb{R}^d}x_1 u(x,\widetilde{t}_2)\, dx-\int_{\mathbb{R}^d}x_1 u(x,\widetilde{t}_1)\, dx=\frac{\nu}{k}\int_{\widetilde{t}_1}^{\widetilde{t}_2}\int_{\mathbb{R}^d} (u(x,t))^k\, dx.
\end{aligned}    
\end{equation*}
Since $k\geq 2$ is an even number, by continuity, the above identity shows that $u(\cdot,t)\equiv 0$ for all $t\in[\widetilde{t}_1,\widetilde{t}_2]$. Thus, the $L^2$ conservation law implies $u\equiv 0$. We remark that this previous strategy has been applied in \cite{FLinaPioncedGBO,FonPO}. 
\\ \\
Next, we deduce \eqref{eqthreetimes1}. Without loss of generality, we assume that $t_1=0$. We will only show \eqref{eqthreetimes1} for $\widetilde{t}_1\in (0,t_2)$ as the considerations for $\widetilde{t}_2\in (t_2,t_3)$ follow by the exact same arguments. Writing $a+2+\frac{d}{2}=r_1+r_2$ with $r_1\in \mathbb{Z}^{+}$, $r_2\in [0,1)$, by the proof of Lemma \ref{unicontlemma}, see \eqref{equnicont7} and the arguments around \eqref{equnicont6}, and the fact that $u_0\in H^{2(a+2+\frac{d}{2})}(\mathbb{R}^d)\cap L^2(|x|^{a+2+\frac{d}{2}}\, dx)$ with $\widehat{u}_0(0)=0$, we obtain
\begin{equation*}
\begin{aligned}
U(t_2)u_0\in L^{2}(|x|^{2(a+2+\frac{d}{2})}\, dx),
\end{aligned}    
\end{equation*}
if and only if for all homogeneous polynomial $P_{r_1}(\xi)$ of order $r_1$, it follows
\begin{equation*}
t_2|\xi|^{-2(r_1-1)}|\xi|^{a}P_{r_1}(\xi)\phi(\xi) \partial_{\xi_j}\widehat{u}_0(0)\in \dot{H}^{r_2}(\mathbb{R}^d),    
\end{equation*}
which by Claim \ref{claimL2integ} is valid if and only if $t\partial_{\xi_j}\widehat{u}_0(0)=0$, for all $j=1,\dots,d$. On the other hand, by Corollary \ref{corollarynonlin}, we know $u^{k-1}\partial_{x_1}u\in L^{1}((0,T);L^{2}(|x|^{2(a+2+\frac{d}{2})}\, dx))$, then by a similar discussion as before, based on the proof of Lemma \ref{unicontlemma}, we have
\begin{equation*}
\begin{aligned}
\int_0^{t_2} U(t_2-\tau)u^{k-1}\partial_{x_1}u(\tau)\, d\tau \in L^{2}(|x|^{2(a+2+\frac{d}{2})}\, dx),
\end{aligned}    
\end{equation*}
if and only if
\begin{equation*}
\begin{aligned}
\int_0^{t_2} (t_2-\tau)\partial_{\xi_j}\widehat{(u^{k-1}\partial_{x_1}u)}(0,\tau)\, d\tau=0, \quad \text{ for all }  j=1,\dots,d.
\end{aligned}    
\end{equation*}
Although the previous remark was given individually, when studying the integral formulation of \eqref{HBO-ZK}, by the same reasoning, it must follow that $u(\cdot,t_2)\in L^2(|x|^{2(a+2+\frac{d}{2})}\, dx)$, if and only if
\begin{equation*}
\begin{aligned}
t_2\partial_{\xi_j}\widehat{u}_0(0)-\nu \int_0^{t_2} (t_2-\tau)\partial_{\xi_j}\widehat{(u^{k-1}\partial_{x_1}u)}(0,\tau)\, d\tau=0.
\end{aligned}
\end{equation*}
When $j=2,\dots, d$, the above identity reduces to $t_2\partial_{\xi_j}\widehat{u}_0(0)=0$. When $j=1$, the above identity is equivalent to 
\begin{equation*}
\begin{aligned}
0=&-i t_2 \int x_1 u_0(x)\, dx+\frac{i \nu}{k}\int_0^{t_2} (t_2-\tau)\int_{\mathbb{R}^d} x_1\partial_{x_1}(u^k)(x,\tau)\, dx\, d\tau\\
=&-i t_2 \int x_1 u_0(x)\, dx-\frac{i\nu}{k}\int_0^{t_2} (t_2-\tau)\int_{\mathbb{R}^d} (u(x,\tau))^k\, dx\, d\tau,
\end{aligned}
\end{equation*}
which by \eqref{unicontident1} and integration by parts allow us to conclude
\begin{equation*}
\begin{aligned}
0=&-it_2\int x_1 u_0(x)\, dx-i\int_0^{t_2} (t_2-\tau)\frac{d}{d\tau}\Big(\int_{\mathbb{R}^d} x_1 u(x,\tau)\, dx\Big)\, d\tau\\
=&-i\int_0^{t_2}\int_{\mathbb{R}^d} x_1 u(x,\tau)\, dx\, d\tau.
\end{aligned}
\end{equation*}
Consequently, by continuity, there exists $\widetilde{t}_1$ such that \eqref{eqthreetimes1} holds true. This completes the proof of Theorem \ref{theoremthreetimes} (i).
\\ \\
Let us deduce (ii). Without loss of generality, we assume $t_1=0$. By the arguments developed in the previous case, we know that there exists $\widetilde{t}_1\in(0,t_2)$ such that \eqref{eqthreetimes1} holds true. If $\int_{\mathbb{R}^d}x_1 u(x,0)\,dx =0$, the proof of the present case follows by integrating over $[0,\widetilde{t}_1]$ the identity \eqref{unicontident1} with $j=1$. Likewise, if $\int_{\mathbb{R}^d}x_1 u(x,t_2)\,dx =0$, the desired result is obtained by integrating over $[\widetilde{t}_1,t_2]$ identity \eqref{unicontident1}. The proof is complete.
\end{proof}

\begin{proof}[Proof of Corollary \ref{sharpthree}]
By the proof of Theorem \ref{theoremthreetimes}, we have that
$u(\cdot,t)\in L^{2}(|x|^{2(a+2+\frac{d}{2})}\, dx)$ if and only if,
\begin{equation*}
    \int_0^t \int_{\mathbb{R}^d} x_1 u(x,\tau)\,dx d\tau=0,
\end{equation*}
and
\begin{equation*}
    \int_{\mathbb{R}^d} x_j u_0(x)\,dx=0,
\end{equation*}
for all $j=2,\dots, d$. Then, the previous observations and similar ideas in the proof of Theorem \ref{theoremthreetimes} yield the desired result.
\end{proof}

Finally, we deduce Corollary \ref{oddcasetheorem}.

\begin{proof}[Proof of Corollary \ref{oddcasetheorem}] Here we set $m\geq 2$ be integer, $a+m+\frac{d}{2}\leq r<a+1+m+\frac{d}{2}$. When $k\geq 2$ is odd, by using the integral formulation of \eqref{HBO-ZK} and Lemma \ref{linearestilemma}, the desired persistence result, i.e., $u\in L^{\infty}([0,T];L^2(|x|^{2r}\,dx)))$ follows by similar arguments in the proof of Theorem \ref{LWPweights}. However, to apply Lemma \ref{linearestilemma}, we must verify that \eqref{linearEstimcomp1} holds true for the nonlinear term in \eqref{HBO-ZK}. In other words, integration by parts yields that \eqref{linearEstimcomp1} is valid provided that
\begin{equation}\label{oddcorol}
    0=\int_{\mathbb{R}^d}x^{\beta}u^{k-1}\partial_{x_1}u(x,t)\, dx=-\frac{1}{k}\int_{\mathbb{R}^d}\partial_{x_1}(x^{\beta})(u(x,t))^k \, dx, 
\end{equation}
for all multi-index $|\beta|\leq m-1$. Since $\partial_{x_1}(x^{\beta})=cx^{\widetilde{\beta}}$ for some multi-index $\widetilde{\beta}$ with $|\widetilde{\beta}|\leq m-2$, the identity \eqref{oddcorol} is guarantee by one of the hypothesis in Corollary \ref{oddcasetheorem}. This remark completes the proof.
\end{proof}

\begin{proof}[Proof of Corollary \ref {corollaryCombinedfKdV}]
We begin by noticing that the well-posedness result in $H^s(\mathbb{R}^d)$ provided by Lemma \ref{comwellp} holds true for the Cauchy problem \eqref{CombHBO-ZK}. Thus, the persistence in weighted spaces is obtained by following the same arguments in the proof of Theorem \ref{LWPweights}, which mostly depend on the linear estimates for $\mathcal{A}_1$ in \eqref{awdiffeeq1}, and Lemma \ref{linearestilemma}. It is worth mentioning that the estimates for the nonlinear terms in \eqref{CombHBO-ZK} follow in a similar fashion applying arguments as in Corollary \ref{corollarynonlin} to each of the nonlinearities individually. On the other hand, the two times unique continuation principles follow by the same reasoning in the proof of Theorem \ref{theoremtwotimes}, again a key observation is that we can apply all of our estimates individually to each of the nonlinearities involved.  This discussion encloses the proof of Corollary \eqref{corollaryCombinedfKdV} (i) and (ii).

The proof of Corollary \ref{corollaryCombinedfKdV}  (iii), (iv) and (v) are a direct adaptation of the proof of Theorem \ref{theoremthreetimes} and Corollaries \eqref{sharpthree} and \eqref{oddcasetheorem}, respectively. A difference is that the first momentum identities for the Cauchy problem \eqref{CombHBO-ZK} are given by
\begin{equation*}
\frac{d}{dt}\int_{\mathbb{R}^d} x_j u(x,t)\, dx=\sum_{j=1}^n\frac{\delta_{j,1} \nu_j }{k_j}\int_{\mathbb{R}^d} (u(x,t))^{k_j}\, dx,
\end{equation*}
for all $j=1,\dots, n$.
\end{proof}


\section*{Acknowledgments}

O. R. received support from Universidad Nacional de Colombia-Bogot\'a.


\bibliographystyle{acm}
\bibliography{bibli}

\begin{thebibliography}{10}

\bibitem{AbBOnaFellSaut1989}
{\sc Abdelouhab, L., Bona, J., Felland, M., and Saut, J.-C.}
\newblock {Nonlocal models for nonlinear, dispersive waves}.
\newblock {\em Physica D 40}, 3 (1989), 360--392.

\bibitem{A}
{\sc Abramyan, L.~A., Stepanyants, Y.~A., and Shrira, V.~I.}
\newblock {Multidimensional solitons in shear flows of the boundary-layer
  type}.
\newblock {\em Sov. Phys. Dokl 37}, 12 (1992), 575--578.

\bibitem{Albert1992}
{\sc Albert, J.~P.}
\newblock Positivity properties and stability of solitary-wave solutions of
  model equations for long waves.
\newblock {\em Comm. Partial Differential Equations 17}, 1-2 (1992), 1--22.

\bibitem{AnguloBonaLinaresScialom2002}
{\sc Angulo, J., Bona, J.~L., Linares, F., and Scialom, M.}
\newblock Scaling, stability and singularities for nonlinear, dispersive wave
  equations: the critical case.
\newblock {\em Nonlinearity 15}, 3 (2002), 759--786.

\bibitem{Benjamin1967}
{\sc Benjamin, T.~B.}
\newblock Internal waves of permanent form in fluids of great depth.
\newblock {\em J. Fluid Mech. 29}, 3 (1967), 559--592.

\bibitem{BonaColinLannes2005}
{\sc Bona, J., Colin, T., and Lannes, D.}
\newblock {Long wave approximations for water waves}.
\newblock {\em Arch. Rational Mech. Anal. 178}, 3 (2005), 373–410.

\bibitem{BonaKalisch2004}
{\sc Bona, J., and Kalisch, H.}
\newblock Singularity formation in the generalized {B}enjamin-{O}no equation.
\newblock {\em Discrete Contin. Dynam. Syst. 11\/} (2004), 27--46.

\bibitem{Bona1981}
{\sc Bona, J.~L.}
\newblock {On solitary waves and their role in the evolution of long waves}.
\newblock {\em Applications of nonlinear analysis in the physical sciences\/}
  (1981), 183--205.

\bibitem{Calderon1965}
{\sc Calder\'on, A.-P.}
\newblock Commutators of singular integral operators.
\newblock {\em Proc. Nat. Acad. Sci. U.S.A. 53\/} (1965), 1092--1099.

\bibitem{CossetiFanelliLinares2019}
{\sc Cossetti, L., Fanelli, L., and Linares, F.}
\newblock Uniqueness results for {Z}akharov-{K}uznetsov equation.
\newblock {\em Comm. Partial Differential Equations 44}, 6 (2019), 504--544.

\bibitem{Cunha2022}
{\sc Cunha, A.}
\newblock On decay of the solutions for the dispersion
  generalized-{B}enjamin-{O}no and {B}enjamin-{O}no equations.
\newblock {\em Adv. Differential Equations 27}, 11-12 (2022), 781--822.

\bibitem{Cunha2023}
{\sc Cunha, A.}
\newblock On uniqueness results for solutions of the {B}enjamin equation.
\newblock {\em J. Math. Anal. Appl. 526}, 2 (2023), Paper No. 127256, 26.

\bibitem{CunhaPastor2014}
{\sc Cunha, A., and Pastor, A.}
\newblock The {IVP} for the {B}enjamin-{O}no-{Z}akharov-{K}uznetsov equation in
  weighted {S}obolev spaces.
\newblock {\em J. Math. Anal. Appl. 417}, 2 (2014), 660--693.

\bibitem{CunhaPastor2016}
{\sc Cunha, A., and Pastor, A.}
\newblock The {IVP} for the {B}enjamin-{O}no-{Z}akharov-{K}uznetsov equation in
  low regularity {S}obolev spaces.
\newblock {\em J. Differential Equations 261}, 3 (2016), 2041--2067.

\bibitem{CunhaPastor2021}
{\sc Cunha, A., and Pastor, A.}
\newblock Persistence properties for the dispersion generalized {BO}-{ZK}
  equation in weighted anisotropic {S}obolev spaces.
\newblock {\em J. Differential Equations 274\/} (2021), 1067--1114.

\bibitem{DawsonMCPON}
{\sc Dawson, L., McGahagan, H., and Ponce, G.}
\newblock {On the Decay Properties of Solutions to a Class of Schr\"odinger
  Equations}.
\newblock {\em Proceedings of the American Mathematical Society 136}, 6 (2008),
  2081--2090.

\bibitem{EscaKenigPonVega2007}
{\sc Escauriaza, L., Kenig, C.~E., Ponce, G., and Vega, L.}
\newblock On uniqueness properties of solutions of the {$k$}-generalized
  {K}d{V} equations.
\newblock {\em J. Funct. Anal. 244}, 2 (2007), 504--535.

\bibitem{Eychenne2023}
{\sc Eychenne, A.}
\newblock Asymptotic {$N$}-soliton-like solutions of the fractional
  {K}orteweg--de {V}ries equation.
\newblock {\em Rev. Mat. Iberoam. 39}, 5 (2023), 1813--1862.

\bibitem{EychenneValet2023}
{\sc Eychenne, A., and Valet, F.}
\newblock Strongly interacting solitary waves for the fractional modified
  {K}orteweg--de {V}ries equation.
\newblock {\em J. Funct. Anal. 285}, 11 (2023), Paper No. 110145, 71.

\bibitem{FLinaPonceWeBO}
{\sc Fonseca, G., Linares, F., and Ponce, G.}
\newblock The {IVP} for the {B}enjamin-{O}no equation in weighted {S}obolev
  spaces {II}.
\newblock {\em J. Funct. Anal. 262}, 5 (2012), 2031--2049.

\bibitem{FLinaPioncedGBO}
{\sc Fonseca, G., Linares, F., and Ponce, G.}
\newblock The {IVP} for the dispersion generalized {B}enjamin-{O}no equation in
  weighted {S}obolev spaces.
\newblock {\em Ann. Inst. H. Poincar\'{e} C Anal. Non Lin\'{e}aire 30}, 5
  (2013), 763--790.

\bibitem{FonPO}
{\sc Fonseca, G., and Ponce, G.}
\newblock The {IVP} for the {B}enjamin-{O}no equation in weighted {S}obolev
  spaces.
\newblock {\em J. Funct. Anal. 260}, 2 (2011), 436--459.

\bibitem{FranLenz2013}
{\sc Frank, R.~L., and Lenzmann, E.}
\newblock {Uniqueness of non-linear ground states for fractional Laplacians in
  $\mathbb{R}$}.
\newblock {\em Acta Math. 210}, 2 (2013), 261--318.

\bibitem{FranLenzSilve2016}
{\sc Frank, R.~L., Lenzmann, E., and Silvestre, L.}
\newblock {Uniqueness of Radial Solutions for the Fractional Laplacian}.
\newblock {\em Comm. Pure Appl. Math. 69}, 9 (2016), 1671--1726.

\bibitem{GhoshSaloUhlmann2020}
{\sc Ghosh, T., Salo, M., and Uhlmann, G.}
\newblock The {C}alder\'{o}n problem for the fractional {S}chr\"{o}dinger
  equation.
\newblock {\em Anal. PDE 13}, 2 (2020), 455--475.

\bibitem{GrafakosOh2014}
{\sc Grafakos, L., and Oh, S.}
\newblock The {K}ato-{P}once inequality.
\newblock {\em Comm. Partial Differential Equations 39}, 6 (2014), 1128--1157.

\bibitem{Guo2012}
{\sc Guo, Z.}
\newblock Local well-posedness for dispersion generalized {B}enjamin-{O}no
  equations in {S}obolev spaces.
\newblock {\em J. Differential Equations 252}, 3 (2012), 2053--2084.

\bibitem{HerrIonesKenKoch2010}
{\sc Herr, S., Ionescu, A.~D., Kenig, C.~E., and Koch, H.}
\newblock A para-differential renormalization technique for nonlinear
  dispersive equations.
\newblock {\em Comm. Partial Differential Equations 35}, 10 (2010), 1827--1875.

\bibitem{HickmanLinaresRiano2019}
{\sc Hickman, J., Linares, F., Ria\~no, O.~G., Rogers, K.~M., and Wright, J.}
\newblock On a higher dimensional version of the {B}enjamin-{O}no equation.
\newblock {\em SIAM J. Math. Anal. 51}, 6 (2019), 4544--4569.

\bibitem{Hur2017}
{\sc Hur, V.~M.}
\newblock Wave breaking in the {W}hitham equation.
\newblock {\em Adv. Math. 317\/} (2017), 410--437.

\bibitem{VeraTaoL2014}
{\sc Hur, V.~M., and Tao, L.}
\newblock Wave breaking for the {W}hitham equation with fractional dispersion.
\newblock {\em Nonlinearity 27}, 12 (2014), 2937--2949.

\bibitem{IfrimTata2019}
{\sc Ifrim, M., and Tataru, D.}
\newblock Well-posedness and dispersive decay of small data solutions for the
  {B}enjamin-{O}no equation.
\newblock {\em Ann. Sci. \'Ec. Norm. Sup\'er. (4) 52}, 2 (2019), 297--335.

\bibitem{KenigIonescu2007}
{\sc Ionescu, A.~D., and Kenig, C.~E.}
\newblock Global well-posedness of the {B}enjamin-{O}no equation in
  low-regularity spaces.
\newblock {\em J. Amer. Math. Soc. 20}, 3 (2007), 753--798.

\bibitem{Iorio1986}
{\sc I\'{o}rio, Jr., R.~J.}
\newblock On the {C}auchy problem for the {B}enjamin-{O}no equation.
\newblock {\em Comm. Partial Differential Equations 11}, 10 (1986), 1031--1081.

\bibitem{Iorio2003}
{\sc Iorio, Jr., R.~J.}
\newblock Unique continuation principles for the {B}enjamin-{O}no equation.
\newblock {\em Differential Integral Equations 16}, 11 (2003), 1281--1291.

\bibitem{IsazaLinaresPonce2013}
{\sc Isaza, P., Linares, F., and Ponce, G.}
\newblock On the propagation of regularity and decay of solutions to the
  {$k$}-generalized {K}orteweg-de {V}ries equation.
\newblock {\em Comm. Partial Differential Equations 40}, 7 (2015), 1336--1364.

\bibitem{JeffreyKakutani1972}
{\sc Jeffrey, A., and Kakutani, T.}
\newblock Weak nonlinear dispersive waves: {A} discussion centered around the
  {K}orteweg-de {V}ries equation.
\newblock {\em SIAM Rev. 14\/} (1972), 582--643.

\bibitem{Kato1983}
{\sc Kato, T.}
\newblock On the {C}auchy problem for the (generalized) {K}orteweg-de {V}ries
  equation.
\newblock In {\em Studies in applied mathematics}, vol.~8 of {\em Adv. Math.
  Suppl. Stud.} Academic Press, New York, 1983, pp.~93--128.

\bibitem{KatoPonce1988}
{\sc Kato, T., and Ponce, G.}
\newblock Commutator estimates and the {E}uler and {N}avier-{S}tokes equations.
\newblock {\em Comm. Pure Appl. Math. 41}, 7 (1988), 891--907.

\bibitem{KenigKo}
{\sc Kenig, C.~E., and Koenig, K.~D.}
\newblock On the local well-posedness of the {B}enjamin-{O}no and modified
  {B}enjamin-{O}no equations.
\newblock {\em Math. Res. Lett. 10}, 5-6 (2003), 879--895.

\bibitem{KenigPilodPonceVega2020}
{\sc Kenig, C.~E., Pilod, D., Ponce, G., and Vega, L.}
\newblock On the unique continuation of solutions to non-local non-linear
  dispersive equations.
\newblock {\em Comm. Partial Differential Equations 45}, 8 (2020), 872--886.

\bibitem{KenigPonceVega1991}
{\sc Kenig, C.~E., Ponce, G., and Vega, L.}
\newblock Well-posedness of the initial value problem for the {K}orteweg-de
  {V}ries equation.
\newblock {\em J. Amer. Math. Soc. 4}, 2 (1991), 323--347.

\bibitem{KPV1993}
{\sc Kenig, C.~E., Ponce, G., and Vega, L.}
\newblock {Well-posedness and scattering results for the generalized
  Korteweg-de Vries equation via the contraction principle}.
\newblock {\em Comm. Pure Appl. Math. 46}, 4 (1993), 527--620.

\bibitem{KenigPonceVega2020}
{\sc Kenig, C.~E., Ponce, G., and Vega, L.}
\newblock Uniqueness properties of solutions to the {B}enjamin-{O}no equation
  and related models.
\newblock {\em J. Funct. Anal. 278}, 5 (2020), 108396, 14.

\bibitem{KenigTakaoka2006}
{\sc Kenig, C.~E., and Takaoka, H.}
\newblock Global wellposedness of the modified {B}enjamin-{O}no equation with
  initial data in {$H^{1/2}$}.
\newblock {\em Int. Math. Res. Not.\/} (2006), Art. ID 95702, 44.

\bibitem{KleinLinaresPilodSaut2018}
{\sc Klein, C., Linares, F., Pilod, D., and Saut, J.-C.}
\newblock On {W}hitham and related equations.
\newblock {\em Stud. Appl. Math. 140}, 2 (2018), 133--177.

\bibitem{KleinSaut2015}
{\sc Klein, C., and Saut, J.-C.}
\newblock A numerical approach to blow-up issues for dispersive perturbations
  of {B}urgers' equation.
\newblock {\em Phys. D 295/296\/} (2015), 46--65.

\bibitem{KleinSautWang2022}
{\sc Klein, C., Saut, J.-C., and Wang, Y.}
\newblock On the modified fractional {K}orteweg--de {V}ries and related
  equations.
\newblock {\em Nonlinearity 35}, 3 (2022), 1170--1212.

\bibitem{KochT}
{\sc Koch, H., and Tzvetkov, N.}
\newblock On the local well-posedness of the {B}enjamin-{O}no equation in
  {$H^s(\mathbb{R})$}.
\newblock {\em Int. Math. Res. Not.}, 26 (2003), 1449--1464.

\bibitem{DonLi2019}
{\sc Li, D.}
\newblock On {K}ato-{P}once and fractional {L}eibniz.
\newblock {\em Rev. Mat. Iberoam. 35}, 1 (2019), 23--100.

\bibitem{LinaresMendezPonce2021}
{\sc Linares, F., Mendez, A., and Ponce, G.}
\newblock Asymptotic behavior of solutions of the dispersion generalized
  {B}enjamin-{O}no equation.
\newblock {\em J. Dynam. Differential Equations 33}, 2 (2021), 971--984.

\bibitem{LinaresPilodSaut2014}
{\sc Linares, F., Pilod, D., and Saut, J.-C.}
\newblock Dispersive perturbations of {B}urgers and hyperbolic equations {I}:
  {L}ocal theory.
\newblock {\em SIAM J. Math. Anal. 46}, 2 (2014), 1505--1537.

\bibitem{LinaresPilodSaut2015}
{\sc Linares, F., Pilod, D., and Saut, J.-C.}
\newblock Remarks on the orbital stability of ground state solutions of
  f{K}d{V} and related equations.
\newblock {\em Adv. Differential Equations 20}, 9-10 (2015), 835--858.

\bibitem{LinaresPonce2023}
{\sc Linares, F., and Ponce, G.}
\newblock On unique continuation for non-local dispersive models.
\newblock {\em Vietnam J. Math. 51}, 4 (2023), 771--797.

\bibitem{Argenis2020}
{\sc Mendez, A.~J.}
\newblock On the propagation of regularity for solutions of the dispersion
  generalized {B}enjamin-{O}no equation.
\newblock {\em Anal. PDE 13}, 8 (2020), 2399--2440.

\bibitem{MendezA2020}
{\sc Mendez, A.~J.}
\newblock On the propagation of regularity for solutions of the fractional
  {K}orteweg--de {V}ries equation.
\newblock {\em J. Differential Equations 269}, 11 (2020), 9051--9089.

\bibitem{Argenis2023}
{\sc Mendez, A.~J.}
\newblock On {K}ato's smoothing effect for a fractional version of the
  {Z}akharov-{K}uznetsov equation.
\newblock {\em Discrete Contin. Dyn. Syst. 43}, 5 (2023), 2047--2101.

\bibitem{MolinetPilod2012}
{\sc Molinet, L., and Pilod, D.}
\newblock The {C}auchy problem for the {B}enjamin-{O}no equation in {$L^2$}
  revisited.
\newblock {\em Anal. PDE 5}, 2 (2012), 365--395.

\bibitem{MolinetPilodVentp2018}
{\sc Molinet, L., Pilod, D., and Vento, S.}
\newblock {On well-posedness for some dispersive perturbations of Burgers'
  equation}.
\newblock {\em Ann. Inst. H. Poincar\'e Anal. Non Lin\'eaire 35}, 7 (2018),
  1719 -- 1756.

\bibitem{MolinetRibaud2004}
{\sc Molinet, L., and Ribaud, F.}
\newblock Well-posedness results for the generalized {B}enjamin-{O}no equation
  with arbitrary large initial data.
\newblock {\em Int. Math. Res. Not.}, 70 (2004), 3757--3795.

\bibitem{NahasPonce2009}
{\sc Nahas, J., and Ponce, G.}
\newblock On the persistent properties of solutions to semi-linear
  {S}chr\"{o}dinger equation.
\newblock {\em Comm. Partial Differential Equations 34}, 10-12 (2009),
  1208--1227.

\bibitem{NaumkinShishmar1994}
{\sc Naumkin, P.~I., and Shishmar\"ev, I.~A.}
\newblock {\em Nonlinear nonlocal equations in the theory of waves}, vol.~133
  of {\em Translations of Mathematical Monographs}.
\newblock American Mathematical Society, Providence, RI, 1994.
\newblock Translated from the Russian manuscript by Boris Gommerstadt.

\bibitem{OhPasqualotto2021}
{\sc Oh, S.-J., and Pasqualotto, F.}
\newblock Gradient blow-up for dispersive and dissipative perturbations of the
  {B}urgers equation.
\newblock {\em Arch. Ration. Mech. Anal. 248}, 3 (2024), Paper No. 54, 61.

\bibitem{Ono1975}
{\sc Ono, H.}
\newblock {Algebraic solitary waves on stratified fluids}.
\newblock {\em J. Phys.Soc. Japan 39}, 4 (1975), 1082--1091.

\bibitem{Angulo2018}
{\sc Pava, J.~A.}
\newblock Stability properties of solitary waves for fractional {K}d{V} and
  {BBM} equations.
\newblock {\em Nonlinearity 31}, 3 (2018), 920--956.

\bibitem{PS}
{\sc Pelinovsky, D.~E., and Shrira, V.~I.}
\newblock {Collapse transformation for self-focusing solitary waves in
  boundary-layer type shear flows}.
\newblock {\em Physics Letters A 206}, 3 (1995), 195 -- 202.

\bibitem{Ponce1991}
{\sc Ponce, G.}
\newblock {On the global well-posedness of the Benjamin-Ono equation}.
\newblock {\em {Differential Integral Equations} 4}, 3 (1991), 527--542.

\bibitem{Riano2021}
{\sc Ria\~{n}o, O.}
\newblock On persistence properties in weighted spaces for solutions of the
  fractional {K}orteweg--de {V}ries equation.
\newblock {\em Nonlinearity 34}, 7 (2021), 4604--4660.

\bibitem{RianoRoudenko2022}
{\sc Ria\~{n}o, O., and Roudenko, S.}
\newblock {Stability and instability of solitary waves in fractional
  generalized KdV equation in all dimensions}.
\newblock {\em arXiv:2210.09159\/} (2022).

\bibitem{RianoRoudenkoYang2022}
{\sc Ria\~{n}o, O., Roudenko, S., and Yang, K.}
\newblock Higher dimensional generalization of the {B}enjamin-{O}no equation:
  2{D} case.
\newblock {\em Stud. Appl. Math. 148}, 2 (2022), 498--542.

\bibitem{OscarWHBO}
{\sc Ria\~{n}o, O.~G.}
\newblock The {IVP} for a higher dimensional version of the {B}enjamin-{O}no
  equation in weighted {S}obolev spaces.
\newblock {\em J. Funct. Anal. 279}, 8 (2020), 108707, 53.

\bibitem{Riano2021II}
{\sc Ria\~{n}o, O.~G.}
\newblock Well-posedness for a two-dimensional dispersive model arising from
  capillary-gravity flows.
\newblock {\em J. Differential Equations 280\/} (2021), 1--65.

\bibitem{SautWang2021II}
{\sc Saut, J.-C., and Wang, Y.}
\newblock Global dynamics of small solutions to the modified fractional
  {K}orteweg--de {V}ries and fractional cubic nonlinear {S}chr\"{o}dinger
  equations.
\newblock {\em Comm. Partial Differential Equations 46}, 10 (2021), 1851--1891.

\bibitem{SautWang2021}
{\sc Saut, J.-C., and Wang, Y.}
\newblock Long time behavior of the fractional {K}orteweg--de {V}ries equation
  with cubic nonlinearity.
\newblock {\em Discrete Contin. Dyn. Syst. 41}, 3 (2021), 1133--1155.

\bibitem{SautWang2022}
{\sc Saut, J.-C., and Wang, Y.}
\newblock The wave breaking for {W}hitham-type equations revisited.
\newblock {\em SIAM J. Math. Anal. 54}, 2 (2022), 2295--2319.

\bibitem{Schippa2020}
{\sc {Schippa}, R.}
\newblock {On the Cauchy problem for higher dimensional Benjamin-Ono and
  Zakharov-Kuznetsov equations}.
\newblock {\em Discrete Contin. Dyn. Syst. 40}, 9 (2020), 5189--5215.

\bibitem{ScottChuMcLau1973}
{\sc Scott, A., Chu, F., and McLaughlin, D.}
\newblock {The soliton: A new concept in applied science}.
\newblock {\em Proceedings of the IEEE 61}, 10 (1973), 1443--1483.

\bibitem{Shrira1989}
{\sc Shrira, V.}
\newblock {On the subsurface waves in the oceanic upper mixed layer}.
\newblock {\em Dokl. Akad. Nauk SSSR 308}, 3 (1989), 732--736.

\bibitem{ShriraVoronoVyac1996}
{\sc Shrira, V., and Voronovich, V.}
\newblock {Nonlinear dynamics of vorticity waves in the coastal zone}.
\newblock {\em J. Fluid Mech. 326\/} (1996), 181–203.

\bibitem{Stein1961}
{\sc Stein, E.~M.}
\newblock The characterization of functions arising as potentials.
\newblock {\em Bull. Amer. Math. Soc. 67\/} (1961), 102--104.

\bibitem{Vento2010}
{\sc Vento, S.}
\newblock Well-posedness for the generalized {B}enjamin-{O}no equations with
  arbitrary large initial data in the critical space.
\newblock {\em Int. Math. Res. Not. IMRN}, 2 (2010), 297--319.

\bibitem{VS}
{\sc Voronovich, V.~V., and Shrira, V.~I.}
\newblock Internal wave--shear flow resonance and wave breaking in the
  subsurface layer.
\newblock In {\em Nonlinear instability analysis}, vol.~II of {\em Advances in
  fluid mechanics}. 28 WIT Press, 2001, pp.~133--177.

\bibitem{ZakharovKuznet1974}
{\sc Zakharov, V., and Kuznetsov, E.}
\newblock {Three-dimensional solitons}.
\newblock {\em Soviet Physics JETP 29\/} (01 1974), 594--597.

\end{thebibliography}

\end{document}